\documentclass[10pt,reqno]{amsart}

\usepackage[latin2]{inputenc}
\usepackage{amsmath}
\usepackage{graphicx}
\usepackage{amssymb}
\usepackage{esint}
\usepackage{xcolor}
\usepackage{amsthm}
\usepackage{epsfig}
\usepackage{enumitem}
\usepackage{mathtools}
\usepackage{esvect}
\usepackage{dsfont}
\usepackage{tikz}
\usetikzlibrary{calc,intersections,through,backgrounds,patterns,arrows.meta, math,through}
\usepackage{graphicx}
\usepackage{caption}
\usepackage{subcaption}
\usepackage[english]{babel}

\usepackage{tikz-3dplot}
\usepackage{pgfplots}

\newtheorem{theorem}{Theorem}
\newtheorem{proposition}[theorem]{Proposition}
\newtheorem{lemma}[theorem]{Lemma}

\newtheorem{definition}[theorem]{Definition}
\newtheorem{construction}[theorem]{Construction}
\newtheorem{remark}[theorem]{Remark}
\newtheorem*{theorem*}{Theorem}

\theoremstyle{definition}

\allowdisplaybreaks[2]

\newcommand{\supp}{\operatorname{supp}}

\newcommand{\dist}{\operatorname{dist}} 
\newcommand{\Id}{\operatorname{Id}}

\newcommand{\R}{\mathbb{R}}

\newcommand{\h}{\operatorname{H}}

\newcommand{\vareps}{\varepsilon}

\newcommand{\ds}{\,\mathrm{d}s}
\newcommand{\dx}{\,\mathrm{d}x}

\newcommand{\ddt}{\frac{\mathrm{d}}{\mathrm{d}t}}
\newcommand{\eps}{\varepsilon}

\newcommand{\chit}{\chi_{\mathcal{T}_{i,j}^{{\color{white} b }}}}
\newcommand{\sumij}{\sum^N_{\substack{i,j=1\\i < j} } }
\newcommand{\sumkij}{\sum^N_{\substack{i,j,k=1 \\ i < j, k \notin \{i,j\}}} }

\newcommand{\sumk}{\sum^N_{\substack{ k=1 \\ k \notin \{i,j\}}}  }

\definecolor{Yellow}{rgb}{0.95,0.9,0.0} 
\definecolor{Red}{rgb}{0.8,0.1,0.1}
\definecolor{Green}{rgb}{0.1,0.65,0.2}
\definecolor{Blue}{rgb}{0.1,0.1,0.8}
\definecolor{Purple}{rgb}{0.7,0.1,0.7}
\definecolor{Grey}{rgb}{0.6,0.6,0.6}

\definecolor{YELLOW}{rgb}{0.95,0.9,0.0} 
\definecolor{RED}{rgb}{0.8,0.1,0.1}
\definecolor{GREEN}{rgb}{0.25,0.65,0.1}
\definecolor{BLUE}{rgb}{0.1,0.1,0.8}
\definecolor{PURPLE}{rgb}{0.7,0.1,0.7}

\usepackage[linktocpage=true,colorlinks=true,linkcolor=Blue,citecolor=Green]{hyperref}

\newcommand{\Rd}{{\mathbb{R}^d}}

\begin{document}

\title[Sharp interface limit of the vectorial Allen-Cahn equation]
{Quantitative convergence of the vectorial Allen-Cahn equation towards multiphase mean curvature flow}
    
\author{Julian Fischer}
\address[Julian Fischer]{Institute of Science and Technology Austria (IST Austria), Am Campus 1, 3400 Klosterneuburg, Austria}
\email{julian.fischer@ist.ac.at}

\author{Alice Marveggio}
\address[Alice Marveggio]{Institute of Science and Technology Austria (IST Austria), Am Campus 1, 3400 Klosterneuburg, Austria} 
\email{alice.marveggio@ist.ac.at}

    
\thanks{This project has received funding from the European Research Council (ERC) under the European Union's Horizon 2020 research and innovation programme (grant agreement No 948819)\smash{\includegraphics[scale=0.03]{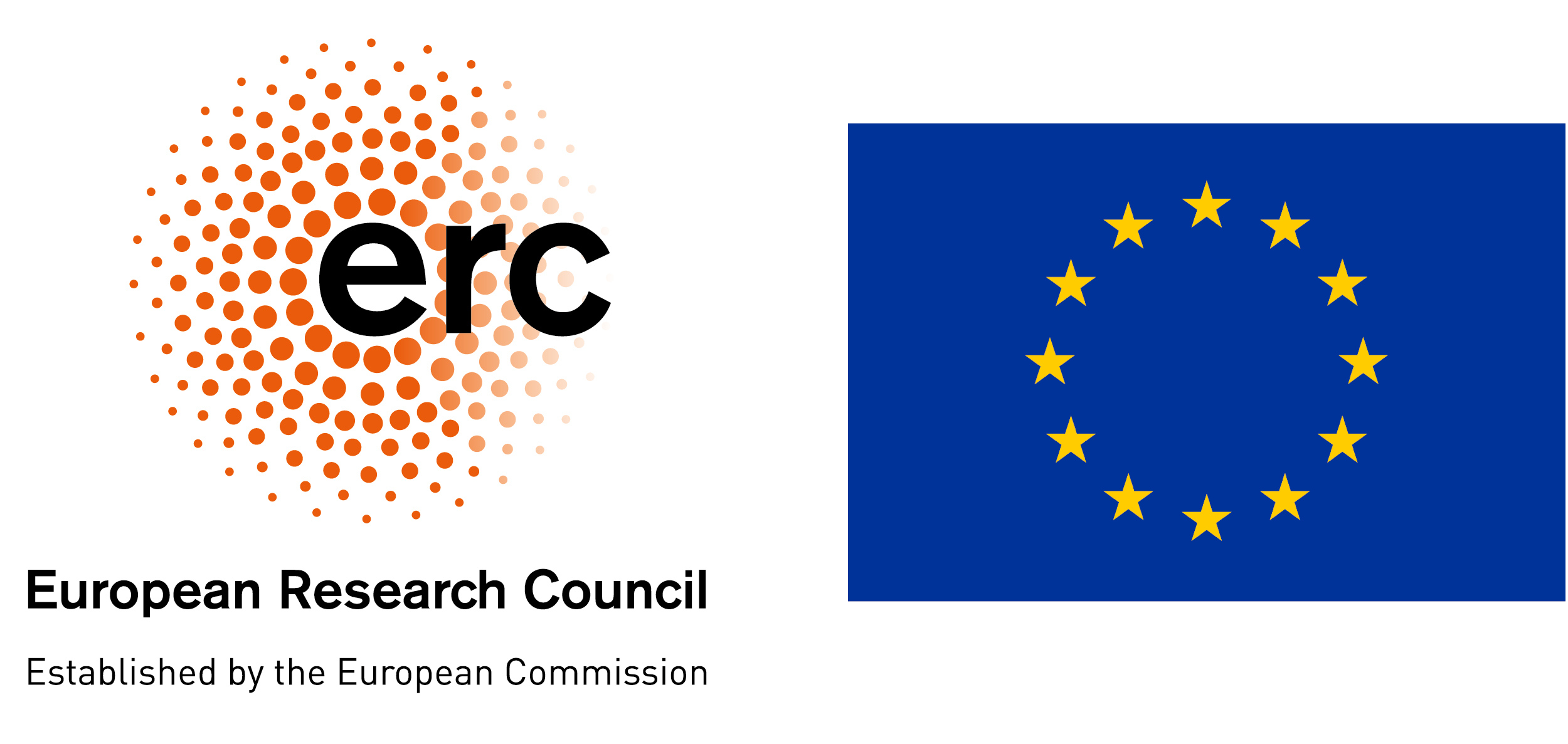}}.}

\begin{abstract}
Phase-field models such as the Allen-Cahn equation may give rise to the formation and evolution of geometric shapes, a phenomenon that may be analyzed rigorously in suitable scaling regimes.
In its sharp-interface limit, the vectorial Allen-Cahn equation with a potential with $N\geq 3$ distinct minima has been conjectured to describe the evolution of branched interfaces by multiphase mean curvature flow.

In the present work, we give a rigorous proof for this statement in two and three ambient dimensions and for a suitable class of potentials: As long as a strong solution to multiphase mean curvature flow exists, solutions to the vectorial Allen-Cahn equation with well-prepared initial data converge towards multiphase mean curvature flow in the limit of vanishing interface width parameter $\varepsilon\searrow 0$. We even establish the rate of convergence $O(\varepsilon^{1/2})$.

Our approach is based on the gradient flow structure of the Allen-Cahn equation and its limiting motion: Building on the recent concept of ``gradient flow calibrations'' for multiphase mean curvature flow, we introduce a notion of relative entropy for the vectorial Allen-Cahn equation with multi-well potential. This enables us to overcome the limitations of other approaches, e.\,g.\ avoiding the need for a stability analysis of the Allen-Cahn operator or additional convergence hypotheses for the energy at positive times.
\end{abstract}
	
\maketitle 

\section{Introduction}

In the present work, we study the behavior of solutions to the vector-valued Allen-Cahn equation
\begin{equation} \label{eq:AC}
	\partial_t u_\varepsilon = \Delta u_\varepsilon - \frac{1}{\varepsilon^2} \partial_u W(u_\varepsilon)
\end{equation}
(with $W$ being an $N$-well potential, see e.\,g.\ Figure~\ref{Fig_pot}a, and $u_\varepsilon: \mathbb{R}^d \times [0,T] \rightarrow \mathbb{R}^{N-1}$)
in the limit of vanishing interface width $\varepsilon\rightarrow 0$.
We prove that for a suitable class of $N$-well potentials $W$, in the limit $\eps\rightarrow 0$ the solutions $u_\eps$ describe a branched interface evolving by multiphase mean curvature flow (see Figure~\ref{Fig_pot}b), provided that a classical solution to the latter exists and provided that one starts with a sequence of well-prepared initial data $u_\eps(\cdot,0)$. For quantitatively well-prepared initial data $u_\eps(\cdot,0)$, we even establish a rate of convergence $O(\eps^{1/2})$ towards the multiphase mean curvature flow limit.

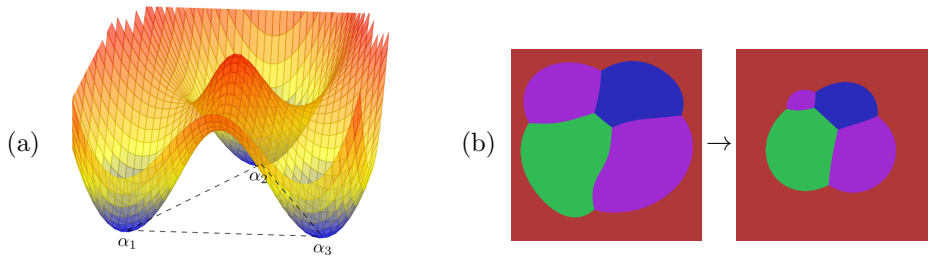
\begin{figure}
\begin{center}
\begin{tikzpicture}
\draw[color=white,opacity=0] (0,-1.7) -- (0,1.7);
\node at (0,0) {(a)};
\end{tikzpicture}
\includegraphics[scale=0.85]{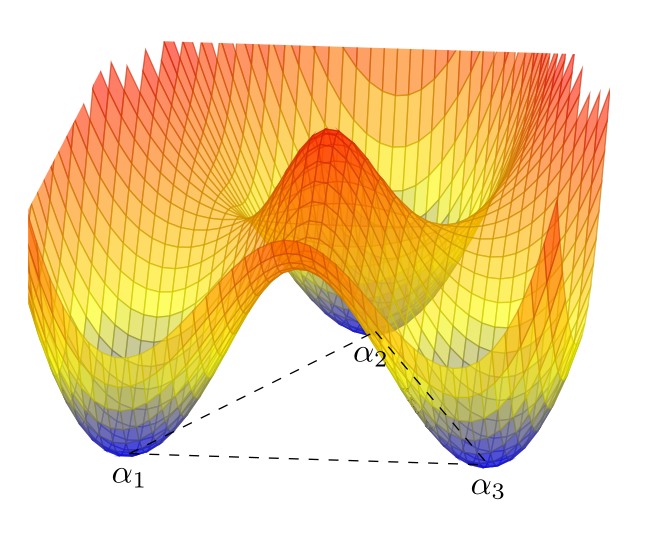}
~~~
\begin{tikzpicture}
\draw[color=white,opacity=0] (0,-1.7) -- (0,1.7);
\node at (0,0) {(b)};
\node at (1.7,0) {\includegraphics[scale=0.285]{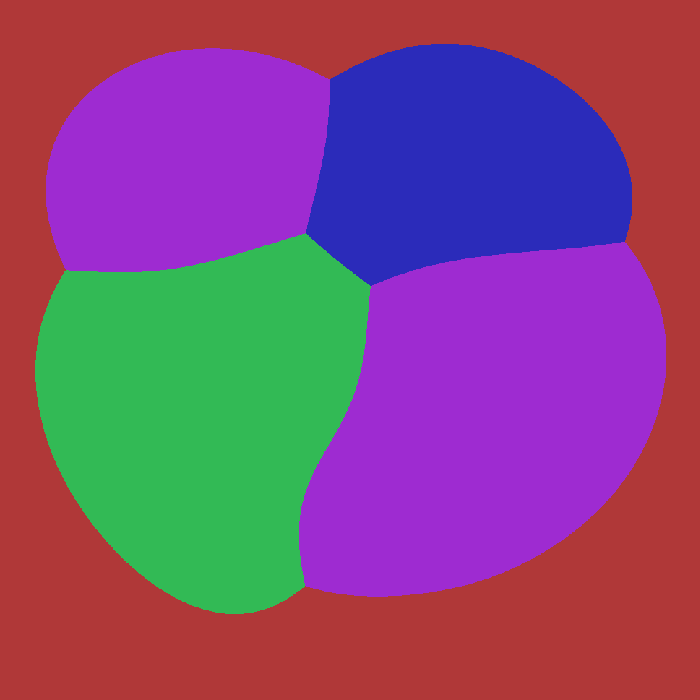}};
\node at (3.2,0) {$\rightarrow$};
\node at (4.7,0) {\includegraphics[scale=0.285]{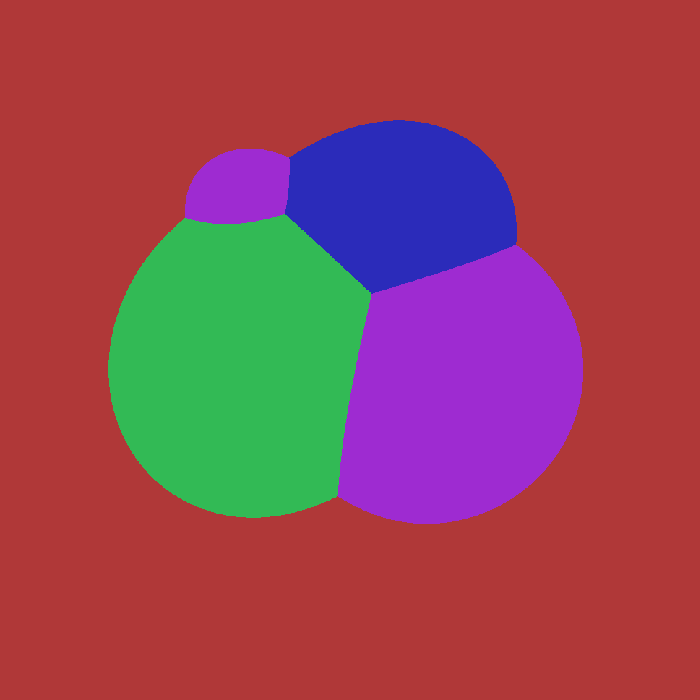}};
\end{tikzpicture}
	\end{center}
\caption{(a) A triple-well potential that attains its minimum at the three points $\alpha_1$, $\alpha_2$, $\alpha_3$. (b) A partition of $\mathbb{R}^2$ evolving by multiphase mean curvature flow, corresponding to the sharp-interface limit $\eps\rightarrow 0$ of the vectorial Allen-Cahn equation \eqref{eq:AC} with $N$-well potential $W$.}
\label{Fig_pot}
\end{figure}

The Allen-Cahn equation \eqref{eq:AC} with $N$-well potential is an important example of a phase-field model, an evolution equation for an order parameter $u_\eps$ that may vary in space and time.
Phase-field models may give rise to the formation and evolution of geometric shapes, a phenomenon that becomes amenable to a rigorous mathematical analysis in suitable scaling regimes.
For several important structural classes of potentials $W$, such a rigorous analysis has long been available for the Allen-Cahn equation: For instance, for the scalar Allen-Cahn equation with two-well potential $W$ -- that is, for \eqref{eq:AC} with $N=2$ -- the convergence towards (two-phase) mean curvature flow in the limit $\eps\rightarrow 0$ has been established by De~Mottoni and Schatzman \cite{DeMottoniSchatzman}, Bronsard and Kohn \cite{BronsardKohn}, Chen \cite{ChenReactDiff}, Ilmanen \cite{IlmanenConvergenceOfAllenCahn}, and Evans, Soner, and Souganidis \cite{EvansSonerSouganidis} in the context of three different notions of solutions to mean curvature flow (namely, strong solutions, Brakke solutions, respectively viscosity solutions).
In such two-phase situations, sharp-interface limits have also been established for more complex phase-field models \cite{Chen,AlikakosBatesChen,AbelsLiu,FeiLiu,AbelsMoserContactAngle}, typically based on an approach that relies on matched asymptotic expansions and a stability analysis of the PDE linearized around a transition profile.
Beyond the case of two-well potentials, results have been much more scarce.
One of the few well-understood settings is the case of the Ginzburg-Landau equation, which corresponds to the Allen-Cahn equation \eqref{eq:AC} with a Sombrero-type potential $W(u)=(1-|u|^2)^2$ and $N=3$, i.\,e.\ with a potential that features a continuum of minima at $\{u\in \mathbb{R}^2:|u|=1\}$. In this case, the convergence of solutions to (codimension two) vortex filaments evolving by mean curvature has been shown in dimensions $d\geq 3$ by Jerrard and Soner \cite{JerrardSonerGL}, Lin \cite{LinGinzburgLandau}, and Bethuel, Orlandi, and Smets \cite{BethuelOrlandiSmets}.

In contrast, for the (vectorial) Allen-Cahn equation \eqref{eq:AC} with a potential $W$  with $N\geq 3$ distinct minima, the only previous results on the sharp-interface limit have been a formal expansion analysis by Bronsard and Reitich \cite{BronsardReitich} and a convergence result that is conditional on the convergence of the Allen-Cahn energy
\begin{align*}
E[u_\eps] := \int_\Rd \frac{\eps}{2}|\nabla u_\eps|^2 + \frac{W(u_\eps)}{\eps} \,dx
\end{align*}
at positive times (more precisely, in $L^1([0,T])$) by Laux and Simon \cite{LauxSimon}.
In particular, to the best of our knowledge not even an unconditional proof of qualitative convergence for well-prepared initial data has been available so far. One of the main challenges that has prevented a full analysis is the emergence of ``branching'' interfaces in the (conjectured) limit of multiphase mean curvature flow (see Figure~\ref{Fig_pot}b), corresponding to a geometric singularity in the limiting motion.

In the present work, we introduce a relative energy approach for the problem of the sharp-interface limit of the vectorial Allen-Cahn equation in a multiphase setting: Building on the concept of ``gradient flow calibrations'' that has been introduced by Hensel, Laux, Simon, and the first author \cite{FischerHenselLauxSimon} precisely for the purpose of handling these branching singularities in multiphase mean curvature flow and combining it with ideas from \cite{FischerLauxSimon}, we introduce a notion of relative energy for the Allen-Cahn equation
\begin{align*}
E[u_\eps|\xi] := \int_{\Rd} \frac{\eps}{2}|\nabla u_\eps|^2 + \frac{W(u_\eps)}{\eps} + \sum_{i=1}^N \xi_i \cdot \nabla \psi_i(u_\eps) \,dx.
\end{align*}
Here, the $\xi_i$ denote a ``gradient flow calibration'' for the strong solution to multiphase mean curvature flow; in particular, $\xi_{i,j}(x,t):=\xi_i-\xi_j$ is an extension of the unit normal vector field of the interface between phases $i$ and $j$ in the strong solution to mean curvature flow at time $t$. The $\psi_i:\mathbb{R}^{N-1}\rightarrow [0,1]$ are suitable $C^{1,1}$ functions that serve as phase indicator functions; in particular, denoting the $N$ minima of the $N$-well potential $W$ by $\alpha_j$ ($1\leq j\leq N$), the functions $\psi_i$ satisfy $\psi_i(\alpha_j)=\delta_{ij}$.
Note that the functions $\psi_j-\psi_i$ will play a role that is somewhat similar to the role of the functions $\psi(u)=\int_{0}^u \sqrt{2W(s)} \,ds$ in the Modica-Mortola trick for a two-well potential $W:\mathbb{R}\rightarrow \mathbb{R}_0^+$ like $W(u)=\frac{9}{8}(1-u^2)^2$.

The properties of the gradient flow calibration $\xi_i$ and the assumptions on the functions $\psi_i:\mathbb{R}^{N-1}\rightarrow [0,1]$ will ensure that the estimate $\big|\sum_{i=1}^N \xi_i \cdot \nabla \psi_i(u_\eps) \big| \leq \tfrac{\eps}{2} |\nabla u_\eps|^2 + \tfrac{1}{\eps} W(u_\eps)$ holds, thereby guaranteeing coercivity of the relative energy $E[u_\eps|\xi]$. In our main result, we prove that for suitable initial data $u_\eps(\cdot,0)$ we have $||\psi_i(u_\eps(\cdot,t))-\bar \chi_i(\cdot,t)||_{L^1(\Rd)} \leq C \eps^{1/2}$ for all $t\leq T$, where the $\bar \chi_i$ denote the phase indicator functions from the strong solution to multiphase mean curvature flow.

Rigorous results on sharp-interface limits for phase-field models -- such as our result -- are also of particular interest from a numerical perspective: In evolution equations for interfaces like e.\,g.\ mean curvature flow, the occurrence of topological changes typically poses a challenge for numerical simulations. One approach to the simulation of evolving interfaces is to construct a mesh that discretizes the initial interface and to numerically evolve the resulting mesh over time; however, it is then a highly nontrivial (and still widely open) question how to continue the numerical mesh beyond a topology change in a numerically consistent way. An alternative approach to the simulation of evolving interfaces that avoids this issue are phase-field models, in which the geometric evolution equation for the interface is replaced by an evolution equation for an order parameter posed on the entire space, allowing also for ``mixtures'' of the phases at the transition regions.
The natural diffuse-interface approximation for multiphase mean curvature flow is given by the vector-valued Allen-Cahn equation with $N$-well potential \eqref{eq:AC}.
The advantage of phase-field approximations for geometric motions such as \eqref{eq:AC} is that one may solve them numerically using standard discretization schemes for parabolic PDEs; however, to establish convergence of the overall scheme towards the orignal interface evolution problem, it is necessary to rigorously justify the sharp-interface limit for the diffuse-interface model.

{\bf Notation.} Throughout the paper, we use standard notation for parabolic PDEs. By $\dot H^1(\Rd)$ we denote the space of functions have a weak derivative $\nabla u\in L^2(\Rd)$ and (in case $d\geq 3$) decay at infinity. In particular, for a function $u\in L^2([0,T];\dot H^1(\Rd))$ we denote by $\nabla u$ its (weak) spatial gradient and by $\partial_t u$ its (weak) time derivative. For functions defined on phase space, like our potential $W:\mathbb{R}^{N-1}\rightarrow [0,\infty)$ or the approximate phase indicator functions $\psi_i:\mathbb{R}^{N-1}\rightarrow [0,\infty)$, we denote their gradient by $\partial_u W$ respectively $\partial_u \psi_i$. For a smooth interface $I_{i,j}$, we denote its mean curvature vector by $\vec{H}_{i,j}$.

\section{Main Results}\label{sec:mainresults}

Our main result identifies the sharp-interface limit $\eps\rightarrow 0$ for the vectorial Allen-Cahn equation \eqref{eq:AC} for a sufficiently broad class of $N$-well potentials $W$ characterized by the following conditions.
\begin{itemize}
\item[(A1)] Let $W:\mathbb{R}^{N-1}\rightarrow [0,\infty)$ be an $N$-well potential of class $C^{1,1}_{loc}(\mathbb{R}^{N-1})$ that attains its minimum $W(u)=0$ precisely in $N$ distinct points $\alpha_1,\ldots,\alpha_N \in \mathbb{R}^{N-1}$. Assume that there exists an integer $q \geq 2$ and constants $C,c >0$ such that in a neighborhood of each $\alpha_i$ we have
\begin{align*}
	c |u - \alpha_i|^q  \leq W(u) \leq C |u - \alpha_i|^q
	.
\end{align*}
\item[(A2)] Let $U \subset \mathbb{R}^{N-1}$ be a bounded convex open set with piecewise $C^1$ boundary and $\{\alpha_1,\ldots,\alpha_N\}\subset \overline{U}$. Suppose that $\partial_u W(u)$ points towards $U$ for any $u\in \partial U$.
\item[(A3)] Suppose that for any two distinct $i,j\in \{1,\ldots,N\}$, there exists a unique minimizing path $\gamma_{i,j}$ connecting $\alpha_i$ to $\alpha_j$ in the sense $\int_{\gamma_{i,j}} \sqrt{2W(\gamma_{i,j})} \,\mathrm{d}\gamma_{i,j}=\smash{\inf_\gamma \int_{\gamma} \sqrt{2W(\gamma)} \,\mathrm{d} \gamma}=1,$ where the infimum is taken over all continuously differentiable paths $\gamma$ connecting $\alpha_i$ to $\alpha_j$.
\item[(A4)] Suppose that there exist continuously differentiable functions $\psi_i:\overline{U}\rightarrow [0,1]$, $1\leq i\leq N$, and a disjoint partition of $\overline{U}$ into sets $\mathcal{T}_{i,j}$, $i<j\in \{1,\ldots,N\}$, subject to the following properties:
\begin{itemize}
\item For any $i\in \{1,\ldots,N\}$, we have $\psi_i(\alpha_i)=1$ and $\psi_i(u)<1$ for $u\neq \alpha_i$.
\item Suppose that on $\mathcal{T}_{i,j}$, all $\psi_k$ with $k\notin \{i,j\}$ vanish.
\item Set $\psi_0:= 1- \sum_{i=1}^N \psi_i$ to achieve $\sum_{i=0}^N \psi_i\equiv 1$ and define $\psi_{i,j}:=\psi_j-\psi_i$. Suppose that there exists $\delta>0$ such that for any mutually distinct $i,j,k\in \{1,\ldots,N\}$ and any $u\in \mathcal{T}_{i,j}$ we have 
\vspace{-1mm}
\begin{align*}
&~~~~~~~~~~~~~\big|\tfrac{1}{2}\partial_u \psi_{i,j}(u)\big|^2 
+ \left( 1 +   \delta \right)  \big|\tfrac{1}{2\sqrt{3}}\partial_u \psi_0(u) \big|^2		
 +  \delta   \left| \partial_u \psi_{i,j} (u)\cdot   \partial_u \psi_{0} (u) \right| \\
 &~~~~~~~~~~~~~\leq 2W(u).
\vspace{-2mm}
\end{align*}
Additionally, suppose there exists a constant $C>0$ such that for any distinct $i,j \in \{1,\ldots,N\}$ and any $u\in \mathcal{T}_{i,j}$ it holds that $|\partial_u \psi_{i}(u) |\leq C \sqrt{2W(u)}$.
\end{itemize}
\end{itemize}
The assumption that our potential $W$ has a finite set of minima as stated in (A1) is fundamental for the scaling limit we consider, as a different structure of the potential would give rise to a different limiting motion -- recall that for instance a Sombrero-type potential would lead to (codimension two) vortex filaments structures \cite{BethuelOrlandiSmets,JerrardSonerGL}. The assumption (A2) is rather mild, ensuring the existence of bounded weak solutions to the vectorial Allen-Cahn equation by a maximum principle (see Remark~\ref{RemarkExistence}). The condition (A3) ensures that for each pair of minima, there is a unique optimal profile connecting the two phases; furthermore, it fixes the surface energy density for an interface between any pair of phases $i$ and $j$ to be $1$. We expect that it would be possible to generalize our results to more general classes of surface tensions as considered in \cite{FischerHenselLauxSimon}; to avoid even more complex notation, we refrain from doing so in the present manuscript.

The assumption (A4) is the only truly restrictive condition in our assumptions; in fact, it does not include potentials which at the same time feature quadratic growth at the minima $\alpha_i$ (i.\,e., with $q=2$ in (A1)) and regularity of class $C^2$. Nevertheless, as we shall see in Proposition~\ref{prop:hpW} below, there exists a broad class of $N$-well potentials -- including in particular potentials of class $C^{1,1}$ with quadratic growth at the minima $\alpha_i$ -- that satisfy all of our assumptions.

Our main result on the quantitative convergence of the vectorial Allen-Cahn equation towards multiphase mean curvature flow reads as follows.
\begin{theorem}
\label{MainTheorem}
Let $d\in \{2,3\}$. In case $d=2$, let $(\bar \chi_1,\ldots,\bar \chi_N)$ be a classical solution to multiphase mean curvature flow on $\mathbb{R}^d$ on a time interval $[0,T]$ in the sense of Definition~\ref{DefinitionStrongSolution} below; in case $d=3$, let $(\bar \chi_1,\ldots,\bar \chi_N)$ be a classical solution to multiphase mean curvature flow of double bubble type in the sense of \cite[Definition~10]{HenselLauxMultiD}. Let $\xi$ be a corresponding gradient flow calibration in the sense of Definition~\ref{DefinitionGradFlowCalibration} below.
Suppose that $W$ is a potential satisfying the assumptions (A1)--(A4). For every $\eps>0$, let $u_\eps\in L^\infty([0,T];\dot H^1(\Rd;\overline{U}))$ be a bounded weak solution to the vectorial Allen-Cahn equation \eqref{eq:AC}.

Assume furthermore that the initial data $u_\eps(\cdot,0)$ are well-prepared in the sense that
\begin{align*}
E[u_\eps|\xi](0)&\leq C\eps,
\\
\max_{i\in \{1,\ldots,N\}} \int_{\Rd} \big|\psi_i(u_\eps(\cdot,0))-\bar \chi_i(\cdot,0)\big| \min\{\dist(x,\partial \supp \bar\chi_i(\cdot,0)),1\} \,dx &\leq C \eps,
\end{align*}
where $E[u_\eps|\xi]$ denotes the relative entropy given as
\begin{align}
\label{eq:relativeenergyfirst}
E[u_\eps|\xi]
:=
\int_{\Rd} \frac{\eps}{2} |\nabla u_\eps|^2 + \frac{1}{\eps} W(u_\eps) + \sum_{i=1}^N \xi_i \cdot \nabla (\psi_i \circ u_\eps) \,dx.
\end{align}
Then the solutions $u_\eps$ to the vectorial Allen-Cahn equation converge towards multiphase mean curvature flow with the rate $O(\eps^{1/2})$ in the sense that
\begin{align*}
\sup_{t\in [0,T]} E[u_\eps|\xi]
&\leq C \eps,
\\
\sup_{t\in [0,T]} \max_{i\in \{1,\ldots,N\}} ||\psi_i(u_\eps(\cdot,t))-\bar \chi_i(\cdot,t)||_{L^1(\Rd)}
&\leq C \eps^{1/2}.
\end{align*}
\end{theorem}
First, let us remark that in the planar case strong solutions to multiphase mean curvature flow are known to exist prior to the first topology change for quite general initial data \cite{BronsardReitich,MantegazzaNovagaPludaSchulze}. Beyond topology changes, in general the evolution by multiphase mean curvature flow may become unstable and uniqueness of solutions may fail, see e.\,g.\ the discussion in \cite{MantegazzaNovagaPludaSchulze} or \cite{FischerHenselLauxSimon}. Thus, quantitative approximation results for multiphase mean curvature flow of the form of our Theorem~\ref{MainTheorem} should not be expected to hold beyond the first topology change. In this sense, our result is optimal.

Second, let us emphasize that by \cite{FischerHenselLauxSimon} and \cite{HenselLauxMultiD} the existence of a gradient flow calibration is ensured
in the following situations:
\begin{itemize}
\item In the planar case $d=2$, gradient flow calibrations exist as long as a strong solution exists.
\item In the three-dimensional case $d=3$, gradient flow calibrations exist as long as a strong solution of double bubble type
(i.\,e.\ in particular with at most $3$ phases meeting at each point) exists.
\end{itemize}
Note that more generally we expect gradient flow calibrations to exist as long as a classical solution to multiphase mean curvature flow exists. Since the construction becomes increasingly technical when the geometrical features become more complex, the construction has not yet been carried out in these more general situations. Nevertheless, as soon as gradient flow calibration becomes available, our results below apply and yield the convergence of the vectorial Allen-Cahn equation to multiphase mean curvature flow in the corresponding setting.

Next, let us remark that we may weaken the assumptions on the sequence of initial data if we are content with lower rates of convergence or merely qualitative convergence statements.
\begin{remark}
As an inspection of the proof of Theorem~\ref{MainTheorem} readily reveals, the assumption of quantitative well-preparedness of the initial data in our theorem can be relaxed, even to a qualitative one. For instance, by merely assuming the qualitative convergences $\lim_{\eps \rightarrow 0} E[u_\eps|\xi](0)=0$ and $\lim_{\eps\rightarrow 0} \max_{i\in \{1,\ldots,N\}} ||\psi_i(u_\eps(\cdot,0))-\bar \chi_i(\cdot,0)||_{L^1(\Rd)}=0$ at initial time, from Theorem~\ref{TheoremRelativeEnergyInequality} and Proposition~\ref{PropBulkConvergence} we are able to obtain the qualitative convergence statement
\begin{align*}
\lim_{\eps\rightarrow 0} \sup_{t\in [0,T]} E[u_\eps|\xi]
 =0=\lim_{\eps\rightarrow 0} \sup_{t\in [0,T]} \max_{i\in \{1,\ldots,N\}} ||\psi_i(u_\eps(\cdot,t))-\bar \chi_i(\cdot,t)||_{L^{1}(\Rd)}.
\end{align*}
Observe furthermore that by the definition of the relative entropy, the convergence $\lim_{\eps \rightarrow 0} E[u_\eps|\xi](0)=0$ is in fact implied by the convergence of the initial energies $E[u_\eps](0)\rightarrow E[\bar \chi](0)=\frac{1}{2}\sum_{i} |\nabla \bar \chi_i(\cdot,0)|(\Rd)$ and the convergence of the initial data $u_\eps(\cdot,0)\rightarrow \sum_{i=1}^N \alpha_i \bar \chi_i(\cdot,0)$ in $L^1(\Rd)$.

To summarize, under the assumptions of Theorem~\ref{MainTheorem} but given now a sequence of solutions $(u_\eps)_\eps$ to the Allen-Cahn equation \eqref{eq:AC} satisfying only the qualitative converge properties at initial time
\begin{align*}
u_\eps(\cdot,0) &\underset{\eps\rightarrow 0}{\longrightarrow} \sum_{i=1}^N \alpha_i \bar \chi_i(\cdot,0) &&\text{in }L^1(\Rd),
\\
E[u_\eps](0) &\underset{\eps\rightarrow 0}{\longrightarrow} E[\bar \chi](0),
\end{align*}
the solutions $u_\eps$ converge to multiphase mean curvature flow in the sense that
\begin{align*}
u_\eps (\cdot,t) &\underset{\eps\rightarrow 0}{\longrightarrow} \sum_{i=1}^N \alpha_i \bar \chi_i(\cdot,t)&&\text{in }L^{1}(\Rd)\text{ for all }t\in [0,T].
\end{align*}
\end{remark}

As the next proposition (and its rather straightforward proof, proceeding by glueing together one-dimensional Modica-Mortola profiles) shows, well-prepared initial data satisfying the upper bound $O(\eps)$ for the relative energy actually exist.
\begin{proposition}
\label{PropositionInitialData}
Let assumptions (A1)--(A4) be in place.
Let $d=2$ and let $(\bar \chi_1(\cdot,0),\ldots,\bar \chi_N(\cdot,0))$ be any initial data whose interfaces consist of finitely many $C^1$ curves that meet at finitely many triple junctions at angles of $120^\circ$. Alternatively, let $d=3$ and let $(\bar \chi_1(\cdot,0),\ldots,\bar \chi_N(\cdot,0))$ be any initial data whose interfaces consist of finitely many $C^1$ interfaces that meet at finitely many triple lines of class $C^1$ at angles of $120^\circ$.

Then for any $\eps>0$ there exists initial data $u_\eps(\cdot,0)$ that is well-prepared in the sense that
\begin{align*}
E_\eps[u_\eps|\xi](0) &\leq C \eps,
\\
\max_{i\in \{1,\ldots,N\}} \int_{\Rd} |\psi_i(u_\eps(\cdot,0))-\bar \chi_i(\cdot,0)| \dist(x,\partial \supp \bar\chi_i(\cdot,0)) \,dx &\leq C \eps,
\end{align*}
where the constant $C$ depends on the initial data $(\bar \chi_1(\cdot,0),\ldots,\bar \chi_N(\cdot,0))$ and on the potential $W$.
\end{proposition}
Nevertheless, note that in the presence of triple junctions this rate of convergence $O(\eps)$ for the relative entropy cannot be improved without modifying either the definition of the relative entropy \eqref{eq:relativeenergy} or our assumptions (A1)--(A4), as it may be impossible to construct initial data $u_\eps(\cdot,0)$ with $E_\eps[u_\eps|\xi]\ll \eps$. Let us illustrate the reason for this limitation in the case $d=2$: Suppose that the initial data $\bar \chi(\cdot,0)$ for the strong solution contain at least one triple junction. By virtue of the term $\int \tfrac{\eps}{2}|\nabla u_\eps|^2 \,dx$ in the energy and the pointwise nonnegativity of the integrand in the relative entropy, if we were to have $E[u_\eps|\xi](0)\ll \eps$, the approximating initial data $u_\eps(\cdot,0)$ would have to contain a true mixture of three phases in an $\eps$-ball $B_\eps(y)$ somewhere. At the same time, our assumptions (A1)--(A4) allow the potential $W$ to be arbitrarily large for a true mixture of three phases (i.\,e., away from the boundary of the triangle in Figure~\ref{Fig_pot}a for a three-well potential as in Definition~\ref{def:3wellW}), independently of the functions $\psi_i$. If $W$ is large enough, on $B_\eps(y)$ the energy density $\tfrac{\eps}{2}|\nabla u_\eps|^2 + \tfrac{1}{\eps} W(u_\eps)$ then cannot be compensated by the term involving $\nabla \psi_i(u_\eps)$ in the relative entropy, resulting in a lower bound for the relative entropy of the order of $\smash{\int_{B_\eps(y)} \frac{1}{2\eps} W(u_\eps) \,dx \geq c \eps^{-1}\times \eps^2=c\eps}$. This limits the overall convergence rate for our method to $O(\eps^{1/2})$ when measured e.\,g.\ in the $L^1$ norm. We expect this to be a limitation of our method, caused by an insufficient control of the precise dynamics of the diffuse-interface model at triple junctions by the relative entropy $E[u_\eps|\xi]$; for suitably prepared initial data, we would anticipate a convergence rate $O(\eps)$. Whether such an improved convergence rate can be deduced by a more refined relative entropy approach is an open question.

Observe that the assumptions (A1) and (A2) are indeed sufficient to deduce global existence of bounded solutions to the Allen-Cahn equation \eqref{eq:AC}, starting from any measurable initial data taking values in $\overline U$.
\begin{remark} \label{RemarkExistence}
Let $W$ be any potential of class $C^{1,1}_{loc}$ satisfying our assumption (A2).
Given any measurable initial data $u_\varepsilon (\cdot,0)$ taking values in $\overline U$, for any $T>0$ there exists a unique bounded weak solution $u_\eps$ to the Allen-Cahn equation \eqref{eq:AC} on the time interval $[0,T]$. To see this, one may first show existence of a weak solution for a slightly modified PDE obtained by replacing $\partial_u W$ outside of $\overline U$ by a Lipschitz extension. For this modified PDE, existence of a weak solution can be shown in a standard way. A comparison argument (using (A2) and in particular the convexity of $\overline U$) then ensures that the weak solution to this modified equation may only take values in $\overline U$, proving both that it is bounded and that it actually solves the original equation.
Uniqueness is shown via the standard argument of a Gronwall-type estimate for the squared $L^2(\Rd)$ norm of the difference between two solutions.
\end{remark}

We next recall the definition of strong solutions to multiphase mean curvature flow in the case of two dimensions. For intuitive but technical-to-state geometric notions, we will refer to the precise definitions in \cite{FischerHenselLauxSimon}.
\begin{definition}[Strong solution for multiphase mean curvature flow]
\label{DefinitionStrongSolution}
Let $d=2$, let $P\geq 2$ be an integer,
and let $T>0$ be a finite time horizon.
Let $\bar\chi_0=(\bar\chi_1^0,\ldots,\bar\chi_P^0)$ be an initial regular partition of $\mathbb{R}^2$ with finite interface energy
in the sense of \cite[Definition~14]{FischerHenselLauxSimon}.

A measurable map
\begin{align*}
\bar\chi=(\bar\chi_1,\ldots,\bar\chi_P)\colon \Rd\times [0,T]\to\{0,1\}^P,
\end{align*}
is called a \emph{strong solution for multiphase mean curvature flow with initial data $\bar\chi_0$} if it satisfies the following conditions:
\begin{subequations}
\begin{itemize}[leftmargin=0.7cm]
\item[i)] (Smoothly evolving regular partition with finite interface energy)
Denote by $I_{i,j}:=\supp \bar \chi_i \cap \supp \bar \chi_j$ for $i\neq j$ the interface between phases $i$ and $j$.
The map~$\bar\chi$ is a smoothly evolving regular partition of $\mathbb{R}^d {\times}[0,T]$
and $\mathcal{I}:=\bigcup_{i,j\in\{1,\ldots,P\},i\neq j}I_{i,j}$ 
is a smoothly evolving regular network of interfaces in 
$\mathbb{R}^d{\times}[0,T]$ in the sense of \cite[Definition~15]{FischerHenselLauxSimon}. In particular, 
for every ${t}\in [0,T]$, $\bar\chi(\cdot,t)$ is a regular partition of $\mathbb{R}^d$
and $\smash{\bigcup_{i\neq j}I_{i,j}(t)}$ is a regular network of interfaces in $\mathbb{R}^d$
in the sense of \cite[Definition~14]{FischerHenselLauxSimon} such that
\begin{align}\label{GlobalEnergyBoundStrongSolution}
\sup_{t\in [0,T]} E[\bar\chi(\cdot,{t})] 
= \sup_{t\in [0,T]} \sum_{i,j=1,i<j}^P 
\int_{{I}_{i,j}(t)}1\,\mathrm{d}S < \infty.
\end{align}
\item[ii)] (Evolution by mean curvature) For $i,j =1,\ldots, P$ with 
$i\neq j$ and $(x,t)\in {I}_{i,j}$ let $\bar V_{i,j}(x,t)$ denote 
the normal speed of the interface at the point $x\in {I}_{i,j}(t)$.
Denoting by $\vec{H}_{i,j}(x,t)$ and $\vec{n}_{i,j}(x,t)$ the mean curvature vector and the normal vector of ${I}_{i,j}(t)$
at~$x\in I_{i,j}(t)$, the interfaces ${I}_{i,j}$ evolve by mean curvature in the sense
\begin{align}\label{MotionByMeanCurvature}
\bar V_{i,j}(x,t) \vec{n}_{i,j}(x,t) = \vec{H}_{i,j}(x,t),
\quad\text{for all } t\in [0,T],\, x\in {I}_{i,j}(t).
\end{align}
	\item[iii)] (Initial conditions) We have $\bar \chi_i(x,0) = \bar \chi_{i}^0(x)$ for all points $x\in \Rd$ and 
	each phase $i\in\{1,\ldots, P\}$.
\end{itemize}
\end{subequations}
\end{definition}

Our main results centrally rely on the concept of gradient flow calibrations introduced in \cite{FischerHenselLauxSimon}, whose definition we next recall.

\begin{definition}
\label{DefinitionGradFlowCalibration}
Let $d\geq 2$.
Let $(\bar \chi_1,\ldots,\bar \chi_N)$ be a smoothly evolving partition of $\Rd$ on a time interval $[0,T)$. Denote by $I_{i,j}:=\supp \bar\chi_i \cap \supp \bar\chi_j$, $1\leq i,j\leq N$, $i\neq j$, the corresponding interfaces. We say that a collection of $C^{1,1}$ vector fields $\xi_i:\mathbb{R}^d \times [0,T) \rightarrow \mathbb{R}^d$, $1\leq i\leq N$, and $B:\mathbb{R}^d \times [0,T) \rightarrow \mathbb{R}^d$ is a gradient flow calibration if the following conditions are satisfied:
\begin{subequations}
		\begin{align} \label{eq:timeevxiij} 
		\partial_t \xi_{i,j} + (B\cdot \nabla ) \xi_{i,j}  + (\nabla B)^\mathsf{T}\xi_{i,j}
		= O(\dist(\cdot, I_{i,j} )),
		\end{align}
		\begin{align} \label{eq:lenghtxiij}
	\frac12  \xi_{i,j} \cdot ( \partial_t \xi_{i,j} + (B\cdot \nabla ) \xi_{i,j}   ) = O(\dist^2(\cdot, I_{i,j} )),
		\end{align}
		\begin{align} \label{eq:normalB}
		(B\cdot {\xi}_{i,j}) {\xi}_{i,j} + (\nabla \cdot \xi_{i,j})\xi_{i,j},
			= O(\dist(\cdot, I_{i,j} ))
		\end{align}
		\begin{align} \label{eq:gradBnormal} 
		\nabla B : \xi_{i,j} \otimes \xi_{i,j} 	= O(\dist(\cdot, I_{i,j} )),
		\end{align}
		\begin{align} \label{eq:gradBmixed} 
	\nabla B :( \xi^\perp_{i,j}  \otimes {\xi}_{i,j} +   {\xi}_{i,j} \otimes \xi^\perp_{i,j}  ) 	= O(\dist(\cdot, I_{i,j} )),
		\end{align}
		\begin{align}
		\label{eq:lenghtconstrxiij}
		1-C_{\text{len}}\dist^2(\cdot, I_{i,j}) \leq |\xi_{i,j}|^2\leq \max\{1-c_{\text{len}} \dist^2(\cdot, I_{i,j})),0\},
		\end{align}
		\begin{align} \label{eq:normalOnInterface} 
	\xi_{i,j} = \vec{n}_{i,j} \qquad \text{on }I_{i,j},
		\end{align}
		\begin{align}
		|\sqrt{3} \xi_k|\leq 1 \qquad \text{and} \qquad \sum_{i=1}^N \xi_i = 0,
		\end{align}
\begin{align}
	|\sqrt{3} \xi_{i,j} \cdot \xi_k | \leq \min\{c_\perp \dist(\cdot, I_{i,j}), \sqrt{\delta_{\text{cal}}} \}
	\quad 
\text{ and } \quad
	\delta_{\text{cal}} \geq \frac{c_\perp^2}{c_{\text{len}}},
\end{align}
\end{subequations}
for some $\delta_{\text{cal}}>0$, some $c_{len}>0$ and some $c_\perp^2>0$.

Moreover, we call a family of $C^{1,1}$ functions $\vartheta_i$ a family of evolving distance weights if they satisfy
\begin{subequations}
\label{DistanceWeightConditions}
\begin{align}
\vartheta_i(\cdot,t) &\leq -c \min\{\dist(\cdot,I_{i,j}(t)),1\} &&\text{in }\{\bar \chi_i(\cdot,t)=1\},
\\
\vartheta_i(\cdot,t) &\geq c \min\{\dist(\cdot,I_{i,j}(t)), 1\} &&\text{outside of }\{\bar \chi_i(\cdot,t)=1\},
\\
|\vartheta_i(\cdot,t)|&\leq C \min\{\dist(\cdot,I_{i,j}(t)),1\} &&\text{globally},
\end{align}
\end{subequations}
and
	\begin{align} \label{eq:timeevtheta}
	|\partial_t \vartheta_i + B \cdot \nabla \vartheta_i| \leq C |\vartheta_i|.
	\end{align}
\end{definition}
Note that the existence of a calibration for a given smoothly evolving partition entails that the partition must evolve by multiphase mean curvature flow (i.\,e., the partition must be a strong solution to multiphase mean curvature flow). In fact, the conditions \eqref{eq:timeevxiij}, \eqref{eq:normalB}, \eqref{eq:normalOnInterface}, and \eqref{eq:lenghtconstrxiij} are sufficient to deduce the property \eqref{MotionByMeanCurvature}.

For many geometries, $(\bar \chi_1,\ldots,\bar \chi_N)$ being a strong solution to multiphase mean curvature flow is also sufficient to construct a gradient flow calibration.

\begin{theorem}[{Existence of gradient flow calibrations, \cite[Theorem~6]{FischerHenselLauxSimon} and \cite[Theorem~1]{HenselLauxMultiD}}]
Let $d\in \{2,3\}$ and let $\bar \chi_0$ be a regular partition of $\mathbb{R}^d$ with finite surface energy; for $d=3$, assume furthermore that the partition corresponds to a double bubble type geometry. Let $\bar \chi$ be a strong solution to multiphase mean curvature flow on the time interval $[0,T]$ in the sense of Definition~\ref{DefinitionStrongSolution} (for $d=2$) respectively in the sense of \cite[Definition~10]{HenselLauxMultiD} (for $d=3$). Then for any $\delta_{\text{cal}}>0$ and any $c_{len}\geq 1$ there exists a gradient flow calibration in the sense of Definition~\ref{DefinitionGradFlowCalibration} up to time $T$. Furthermore, there also exists a family of evolving distance weights.
\end{theorem}

Finally, we conclude this section by showing that the class of potentials $W$ satisfying the assumptions (A1)--(A4) is indeed sufficiently broad. In fact, given
\begin{itemize}
\item a prescribed set of $N$ minima $\alpha_i\in \mathbb{R}^{N-1}$, $1\leq i\leq N$,
\item a prescribed set of non-intersecting minimal paths $\gamma_{i,j}$, $1\leq i<j\leq N$, that meet at the $\alpha_i$ at positive angles, and
\item a potential $\tilde W:\cup_{i,j:i<j} \gamma_{i,j}\rightarrow [0,\infty)$ defined on the minimal paths $\gamma_{i,j}$ and subject to (A1) and (A3), i.\,e.\ in particular with $\int_{\gamma_{i,j}} \sqrt{2\tilde W(u)} \,d\gamma(u)=1$,
\end{itemize}
it is always possible to extend the potential $\tilde W$ to a potential $W:\mathbb{R}^{N-1}\rightarrow [0,\infty)$ that satisfies condition (A4). More precisely, to satisfy (A4) it is sufficient to require $W(u)\geq (1+M |u-\alpha_i|^{-4} \dist(u,\gamma)^4)W(P_\gamma u)$ in some neighborhood $\mathcal{U}_i$ of $\alpha_i$ (with $P_\gamma$ denoting the projection onto the nearest point among all paths $\gamma:=\cup_{j,k}\gamma_{j,k}$) as well as $W(u)\geq M \dist(u,\cup_{i<j}\gamma_{i,j})^2$ in $\mathbb{R}^{N-1}\setminus \cup_i \mathcal{U}_i$. Here, $M$ is a constant depending only on $\tilde W$, the paths $\gamma_{i,j}$, and the neighborhoods $\mathcal{U}_i$.

For the sake of simplicity, we limit ourselves in our rigorous statement to the study of potentials defined on a simplex $\triangle^{N-1}$; however, it is not too difficult to see that our construction would generalize to the aforementioned situation.
\begin{proposition} \label{prop:hpW}
	Let $N\geq 3$. Let $\triangle^{N-1}$ be an $(N-1)$-simplex with edges of unit length in $\R^{N-1}$.  
	Let $W:\triangle^{N-1} \rightarrow [0,\infty)$ be a strongly coercive symmetric $N$-well potential on the simplex $\triangle^{N-1}$ in the sense of Definition~\ref{def:3wellW} below. 
	Then, the assumptions (A1)--(A4) (see Section \ref{sec:mainresults}) are satisfied. In particular, (A4) holds true for the set of functions $\psi_i:\triangle \rightarrow [0,1]$, $1\leq i\leq N$, provided by Construction~\ref{def:approxpartition} below.
\end{proposition}

\section{Strategy of the proof}

The key idea for our proof is the notion of relative entropy (or, more accurately, relative energy) given by
\begin{align} 
E[u_\varepsilon|\xi] &:= E[u_\varepsilon] + \sum_{i=1}^{3} \int_{\R^d} \xi_i \cdot \nabla (\psi_i  \circ u_\varepsilon ) \dx \notag \\&
=  \int_{\R^d} \frac{ \varepsilon }{2} |\nabla u_\varepsilon|^2 + \frac{1}{\varepsilon} W(u_\varepsilon)  + \sum_{i=1}^{N} \xi_i \cdot \nabla (\psi_i  \circ u_\varepsilon ) \dx. \label{eq:relativeenergy}
\end{align}
The form of the ansatz for the relative entropy is inspired by two earlier approaches:
\begin{itemize}
\item
The concept of gradient flow calibrations introduced in \cite{FischerHenselLauxSimon} by the first author, Hensel, Laux, and Simon to derive weak-strong uniqueness and stability results for distributional solutions to multiphase mean curvature flow. Gradient flow calibrations provide a lower bound of the form $-\sum_i \int \xi_i \cdot d\nabla \chi_i$ on the interface energy functional $\frac{1}{2}\sum_i \int 1 \,d|\nabla \chi_i|$, thereby facilitating a relative entropy approach to weak-strong uniqueness principles for multiphase mean curvature flow. We emphasize that gradient flow calibrations are specifically designed to handle the (singular) geometries at triple junctions in the strong solution. We refer to \cite{JerrardSmets,FischerHensel} for earlier uses of relative entropy techniques for weak-strong uniqueness for geometric evolution problems with smooth geometries (in the strong solution).
\item The relative entropy approach to the sharp-interface limit of the scalar Allen-Cahn equation by the first author, Laux, and Simon \cite{FischerLauxSimon}, relying on the Modica-Mortola trick to obtain a lower bound of the form $\int \xi\cdot \nabla \psi(u_\eps) \,dx$ for the Ginzburg-Landau energy $E[u_\eps]=\int_\Rd \frac{\eps}{2}|\nabla u_\eps|^2 + \frac{1}{\eps}W(u_\eps) \,dx$ (see also \cite{LauxLiu} for a subsequent application of the relative energy method to a problem in the context of liquid crystals).
\end{itemize}
The two key steps towards establishing our main results are as follows:
\begin{itemize}
\item Establishing a number of coercivity properties of the relative entropy $E[u_\eps|\xi]$, including for example
\begin{align}
E[u_\eps|\xi]\geq c\int \min\{\dist^2(\cdot,\cup_{i\neq j} I_{i,j}),1\} (\tfrac{\eps}{2}|\nabla u_\eps|^2+\tfrac{1}{\eps}W(u_\eps)) \,dx.
\label{SimpleCoercivity}
\end{align}
\item Deriving a Gronwall-type estimate 
for the time evolution of the relative energy of the type
\begin{align*}
\partial_t E[u_\eps|\xi]\leq C E[u_\eps|\xi].
\end{align*}
\end{itemize}
We shall illustrate this strategy by stating the main intermediate results in the present section below.

As it central for our strategy, let us first give the main argument for the coercivity of the relative entropy \eqref{eq:relativeenergy} (despite it being slightly technical). It makes use of the following elementary lemma.
\begin{lemma} \label{prop:sum}
Let $\xi_i$, $1\leq i\leq N$, be vector fields of class $C^1$ satisfying $\sum_{i=1}^N \xi_i=0$; suppose that at any point $(x,t)\in \Rd\times [0,T]$ at most three of the $\xi_i$ do not vanish.
Let $\psi_i:\mathbb{R}^{N-1}\rightarrow [0,1]$, $1 \leq i\leq N$, be functions as in assumption (A4). In particular, set $\psi_0:=1-\sum_{i=1}^N \psi_i$. Let $u_\eps\in L^\infty([0,T];\dot H^1(\Rd))$.
Defining $\psi_{i,j}:= \psi_{j}- \psi_{i}$ and $\xi_{i,j}:=\xi_i-\xi_j$, we have for any distinct $i,j,k \in \{1,\ldots,N\}$
	\begin{align} 
		\sum_{\ell=1}^{N} \xi_\ell \otimes \nabla (\psi_\ell  \circ u_\varepsilon ) 
		&= 
		- \frac12 \xi_{i,j} \otimes \nabla (\psi_{i,j} \circ u_\varepsilon)
		+ \sumk \frac12 \xi_{k} \otimes  \nabla (\psi_{0} \circ u_\varepsilon)  
		\label{eq:sum} 
	\end{align}
almost everywhere in $\{u_\eps\in \mathcal{T}_{i,j}\}$ as well as
	\begin{align}
		\sum_{\ell=1}^{N} \nabla \xi_\ell \otimes \nabla (\psi_\ell  \circ u_\varepsilon ) 
		&= 
		- \frac12 \nabla \xi_{i,j} \otimes \nabla (\psi_{i,j} \circ u_\varepsilon)
		+ \sumk \frac12 \nabla \xi_{k} \otimes  \nabla (\psi_{0} \circ u_\varepsilon),
		\label{eq:sum2}
		\\
		\sum_{\ell=1}^{N} \nabla \xi_\ell \otimes \partial_u \psi_\ell (u_\varepsilon)
		&= 
		- \frac12 \nabla \xi_{i,j} \otimes \partial_u \psi_{i,j} (u_\varepsilon)
		+ \sumk \frac12 \nabla \xi_{k} \otimes  \partial_u \psi_{0} (u_\varepsilon)
		\label{eq:sum3}
	\end{align}
almost everywhere in $\{u_\eps\in \mathcal{T}_{i,j}\}$.
\end{lemma}
\begin{proof}
	By adding zeros, using the definitions $\xi_{i,j}:=\xi_i-\xi_j$ and $\psi_{i,j}:=\psi_j-\psi_i$, we obtain
    \begin{align*}
    \sum_{\ell=1}^{N}  \xi_\ell \cdot \nabla (\psi_\ell  \circ u_\varepsilon ) 
    = &
    \,- \frac12 \xi_{i,j} \cdot \nabla (\psi_{i,j} \circ u_\varepsilon) 
    + \frac12 \xi_{i} \cdot \nabla ((\psi_{i}+ \psi_{j}) \circ u_\varepsilon) \\
    &+ \frac12 \xi_{j} \cdot \nabla ((\psi_{i}+ \psi_{j}) \circ u_\varepsilon) 
    + \sum_{\substack{k=1 \\k \notin\{i,j\} }}^{N} \xi_{k} \cdot \nabla (\psi_k \circ u_\varepsilon),
    \end{align*}	
almost everywhere in $\R^d \times (0,T)$. The equation \eqref{eq:sum} now follows by exploiting that $\partial_u \psi_k=0$ on $\mathcal{T}_{i,j}$ for $k\notin \{i,j\}$, inserting the definition of $\psi_0$, and using $\sum_{\ell=1}^N \xi_\ell =0$.
The proof of the other properties is analogous.
\end{proof}
With the previous lemma and our assumptions (A1)--(A4), it becomes rather straightforward to establish coercivity of our relative energy:
Observe that we may compute for $(x,t)$ with $u_\eps(x,t) \in \mathcal{T}_{i,j}$ 
\begin{align} \label{eq:startcoer0}
&\frac{\vareps}{2}|\nabla u_\vareps|^2
+\frac{1}{\vareps} W(u_\vareps)
+\sum^N_{\ell=1} \xi_\ell \cdot \nabla (\psi_\ell \circ u_\vareps)
\notag\\&
=
\frac{\vareps}{2}|\nabla u_\vareps|^2
+\frac{1}{\vareps} W(u_\vareps)
-\bigg(\tfrac{1}{2}\partial_u \psi_{i,j} (u_\vareps)\otimes \xi_{i,j} 
-  \sumk  \tfrac1{2\sqrt{3}} \partial_u \psi_0 (u_\vareps)
\otimes \sqrt{3}\xi_k 
\bigg) :\nabla u_\vareps
\notag\\&
=
\frac{1}{2}\bigg|\sqrt{\vareps}\nabla u_\vareps - \frac{1}{\sqrt{\vareps}} \bigg(\tfrac{1}{2}\partial_u \psi_{i,j} (u_\vareps)\otimes \xi_{i,j} 
- \sumk \tfrac1{2\sqrt{3}} \partial_u \psi_0 (u_\vareps)
\otimes \sqrt{3}\xi_k 
\bigg) \bigg|^2
\notag\\&~~~+
\frac{1}{2\vareps} \Bigg[ 2W(u_\vareps)-\bigg|\tfrac{1}{2}\partial_u \psi_{i,j}(u_\vareps) \otimes \xi_{i,j}
- \sumk \tfrac1{2\sqrt{3}} \partial_u \psi_0 (u_\vareps) \otimes \sqrt{3}\xi_k 
\bigg|^2\Bigg]
\end{align}
due to the fact that $\psi_k \equiv 0$ on $\mathcal{T}_{i,j}$ for any $k \in \{1,...,N\} \setminus \{i,j\}$. 
This will be the starting point to prove the coercivity properties satisfied by the relative energy functional \eqref{eq:relativeenergy}; note in particular that
\begin{align*}
&\bigg|
\tfrac{1}{2}\partial_u \psi_{i,j} (u_\vareps)\otimes \xi_{i,j}
-  \sumk \tfrac1{2\sqrt{3}} \partial_u \psi_0 (u_\vareps) \otimes \sqrt{3}\xi_k 
 \bigg|^2
\\&
= |\xi_{i,j}|^2 \big|\tfrac{1}{2}\partial_u \psi_{i,j} (u_\vareps)\big|^2
+ \sumk | \sqrt{3}\xi_{k}|^2 \big|\tfrac1{2\sqrt{3}} \partial_u \psi_0 (u_\vareps)  \big|^2\\
&~~~- \sumk \frac{1}{2} (\xi_{i,j}\cdot \xi_k) \partial_u \psi_{i,j} (u_\vareps)\cdot \partial_u \psi_0 (u_\vareps).
\end{align*}
Using our assumption (A4) and the properties of the gradient flow calibration $|\xi_i|\leq \frac{1}{\sqrt{3}}$, $|\xi_{i,j}|\leq \max\{1-c_{len}\dist^2(\cdot,I_{i,j}),0\}$, and $|\sqrt{3} \xi_{i,j}\cdot \xi_k|\leq \min\{c_\perp \dist(\cdot,I_{i,j}), \sqrt{\delta_{cal}}\}$ for pairwise distinct $i,j,k$, this establishes a first coercivity bound like \eqref{SimpleCoercivity}.
Going substantially beyond this simple estimate, we shall see that in fact we have the following coercivity properties.
\begin{proposition} \label{prop:maincoer}
Let $W$ and $\psi_i$ be functions subject to assumption (A4).
Let $\xi_i$, $1\leq i\leq N$, be any collection of $C^1$ vector fields satisfying $\sum_{i=1}^N \xi_i=0$, $|\sqrt{3}\xi_i|\leq 1$ for all $i$, as well as with the notation $\xi_{i,j}:=\xi_i-\xi_j$
\begin{align}
\label{AssumptionLengthXi}
|\xi_{i,j}|^2 + \frac1{\delta_{\text{cal}}}  \sum^N_{\substack{ k=1 : k \notin \{i,j\}}}  |\sqrt{3} \xi_{i,j} \cdot \xi_{k}|^2 \leq 1.
\end{align}
Furthermore, suppose that at each point at most three of the vector fields $\xi_i$ do not vanish.
For any function $u_\eps\in \dot H^1(\Rd;\overline{U})$ with $E[u_\eps]<\infty$, we then have the estimates
\begin{subequations}
		\begin{align}
	\label{eq:coeren}
	\int_{\R^d} \left(\sqrt{\vareps} |\nabla u_\vareps| - \frac{1}{\sqrt{\vareps} } \sqrt{2W(u_\vareps)} \right)^2 \dx 
		&\leq C E[u_\varepsilon|\xi] \,,
\\ \label{eq:coernormalpsi}
 \sumij	\int_{\R^d}    \left|\frac{\nabla (\psi_{i,j}  \circ u_\vareps)}{|\nabla (\psi_{i,j}  \circ u_\vareps)|} - \xi_{i,j} \right|^2
	|\nabla (\psi_{i,j} \circ u_\vareps)|  \chit (u_\vareps) \dx
&	\leq C E[u_\varepsilon|\xi] ,
	\\ \label{eq:coerdistpsi}
\sumij    \int_{\R^d} \min \{ \dist^2(x, I_{i,j}) ,1\} |\nabla (\psi_{i,j} \circ u_\vareps)|  \chi_{\mathcal{T}_{i,j}}(u_\vareps) \dx 
&	\leq C E[u_\varepsilon|\xi] ,
			\\ \label{eq:coerdist}
		\sumij   \int_{\R^d} \min \{ \dist^2(x, I_{i,j}) ,1\} \left(\frac{\varepsilon}{2} | \nabla u_\vareps|^2 + \frac1\varepsilon W(u_\vareps)\right)\chit (u_\vareps) \dx 
		&	\leq C E[u_\varepsilon|\xi] \,,\\
		 \label{eq:coertangentgradu}
		 \sumij	 \int_{\R^d}   \varepsilon | (\Id-\xi_{i,j} \otimes \xi_{i,j})  \nabla u_\vareps^\mathsf{T} |^2 \chit (u_\vareps) \dx 
		 & \leq C E[u_\varepsilon|\xi].
\end{align}
\end{subequations}
\end{proposition} 
These coercivity estimates will be derived as a consequence of the computation \eqref{eq:startcoer0} and the following coercivity properties.
\begin{proposition}  \label{prop:coerpartialpsi}
Let $W$ and $\psi_i$ be functions subject to assumption (A4). Let $\xi_i$, $1\leq i\leq N$, be as in Proposition~\ref{prop:maincoer}.
For any function $u_\eps\in \dot H^1(\Rd;\overline{U})$ with $E[u_\eps]<\infty$, we then have the estimates
\begin{subequations}
	\begin{align} 
	\label{eq:coerpartialpsik}
	\sumij  \int_{\R^d} \frac{1}{ \vareps}  \left| \partial_u \psi_{0} (u_\vareps)  \right|^2   \chit (u_\vareps) \dx 
	&\leq C E[u_\varepsilon|\xi] ,
\\ 
\label{eq:coerpartialpsiijk}
\sumij	\int_{\R^d} \frac{1}{\vareps}  \left| \partial_u \psi_{i,j} (u_\vareps)  \cdot   \partial_u \psi_{0} (u_\vareps) \right|  \chit(u_\vareps) \dx 
	&\leq C E[u_\varepsilon|\xi] \,,
	\\
	\label{eq:coergradpsik}
\sumij  \int_{\R^d}    | \nabla ( \psi_{0} \circ u_\vareps) | \chit(u_\vareps) \dx 
	&\leq C E[u_\varepsilon|\xi] \,.
	\end{align}	
\end{subequations}
\end{proposition} 

To introduce a proxy at the level of the Allen-Cahn equation for the limiting mean curvature (or, more precisely, a quantity $\h_\vareps$ such that $|\h_\vareps|^2$ is a proxy for the dissipation in mean curvature flow), we introduce the abbreviation
\begin{equation} \label{eq:h}
	\h_\vareps := - \vareps \left(\Delta u_\vareps - \frac{1}{\vareps^2} \partial_u W(u_\vareps)\right) \cdot \frac{\nabla u_\vareps}{|\nabla u_\vareps|}.
\end{equation}

The key step in our proof is to establish the following estimate for the relative energy using a Gronwall-type argument.
\begin{theorem}[Relative energy inequality]
\label{TheoremRelativeEnergyInequality}
Let $\bar{\chi} = (\bar{\chi}_1,\ldots, \bar{\chi}_N)$ be a smoothly evolving partition of $\Rd$; let $((\xi_i)_i,B)$ be an associated gradient flow calibration in the sense of Definition~\ref{DefinitionGradFlowCalibration}.
Let $W$ be a potential subject to assumptions (A1)--(A4).
Let $u_\varepsilon$ be a bounded solution to the vector-valued Allen-Cahn equation \eqref{eq:AC} with initial data $u_\varepsilon (\cdot,0) \in \dot H^1(\Rd;\overline{U})$ with finite energy $E[u_\varepsilon(\cdot,0)]<\infty$. Then for any $t \in [0, T ]$ the estimate
\begin{align} \label{ineq:relativeenergy}
	\ddt &E[u_\varepsilon|\xi] +\ \sumij \int_{\R^d} \frac{1}{2 \vareps}|\h_\vareps - \vareps  (B \cdot {\xi}_{i,j}){\xi}_{i,j} |\nabla u_\vareps| |^2  \chi_{\mathcal{T}_{i,j}}(u_\vareps)\dx \notag\\
	& +\int_{\R^d} \frac{1}{2 \vareps} \left( \Big|\vareps \Delta u_\vareps - \frac{1}{\vareps} \partial_u W(u_\vareps)\Big|^2 - |\h_\vareps|^2\right) \dx \notag\\
	& + \int_{\R^d} \frac{1}{4 \vareps}  \Big| \Big(\vareps \Delta u_\vareps - \frac{1}{\vareps} \partial_u W(u_\vareps)\Big) + \sum_{i=1}^N (\nabla \cdot \xi_i) \partial_u \psi_i (u_\vareps) \Big|^2 \dx \notag\\
	& \leq C(d,\bar \chi) E[u_\varepsilon|\xi] 
\end{align}
holds true, with $\h_\eps$ as defined in \eqref{eq:h} and $E[u_\eps|\xi]$ as defined in \eqref{eq:relativeenergyfirst}.
\end{theorem}
Building on the previous estimate and the coercivity properties of the relative entropy, we will show the following error estimate at the level of the indicator functions.
\begin{proposition}
\label{PropBulkConvergence}
Let the assumptions of Theorem~\ref{TheoremRelativeEnergyInequality} be in place. In addition, let $\vartheta_i$ be a family of evolving distance weights as defined in Definition~\ref{DefinitionGradFlowCalibration}. We then have for all $i\in \{1,\ldots,N\}$
\begin{align*} 
&\sup_{t \in [0,T]} \int_{\R^d} |\psi_i(u_\vareps ) - \bar{\chi}_i|  \min\{\dist(\cdot, \partial \supp \bar \chi_i(\cdot,t) ),1\} \dx
\\&
\leq C(d,T, (\bar{\chi}(t))_{t \in [0,T]}) E[u_\eps|\xi](0)
\\&~~~~
+ C(d,T, (\bar{\chi}(t))_{t \in [0,T]}) \int_{\R^d} |\psi_i(u_\vareps(\cdot,0)) - \bar{\chi}_i(\cdot,0)|  \min\{\dist(\cdot, \partial \supp \bar \chi_i(\cdot,0) ),1\} \dx .
\end{align*}
\end{proposition}

The proof of Theorem~\ref{TheoremRelativeEnergyInequality} crucially relies on the coercivity properties of Proposition~\ref{prop:coerpartialpsi} and \ref{prop:maincoer} and the following simplification of the evolution equation for the relative entropy.
\begin{lemma}
\label{LemmaRelativeEnergyEvolution}
Let $W$ be a potential of class $C^{1,1}_{loc}(\mathbb{R}^{N-1})$ subject to assumptions (A1)--(A4).
Let $u_\varepsilon$ be a solution to the vector-valued Allen-Cahn equation \eqref{eq:AC} with initial data $u_\varepsilon (\cdot,0) \in \dot H^1(\Rd;\overline{U})$ with finite energy $E[u_\eps(\cdot,0)]<\infty$. 
Let $(\xi_i,B)$ be a gradient flow calibration in the sense of Definition~\ref{DefinitionGradFlowCalibration}.
The time evolution of the relative energy \eqref{eq:relativeenergy} is then given by
\begin{align}
\label{eq:timeevolutionrelativeenergy}
&\ddt E[u_\varepsilon|\xi]
\\
\nonumber
&= - \sumij  \int_{\R^d} \frac{1}{2 \vareps}\big|\h_\vareps - \vareps  (B \cdot {\xi}_{i,j}){\xi}_{i,j} |\nabla u_\vareps| \big|^2  \chit (u_\vareps)\dx
\\&~~~
\nonumber
- \int_{\R^d} \frac{1}{2 \vareps}  \left| \Big(\vareps \Delta u_\vareps - \frac{1}{\vareps} \partial_u W(u_\vareps)\Big) + \sum^N_{i=1} (\nabla \cdot \xi_i) \partial_u \psi_i (u_\vareps) \right|^2 \dx
\\&~~~
\nonumber
-\int_{\R^d} \frac{1}{2 \vareps} \left( \Big|\vareps \Delta u_\vareps - \frac{1}{\vareps} \partial_u W(u_\vareps)\Big|^2 - |\h_\vareps|^2\right) \dx
\\&~~~
\nonumber
+\operatorname{Err}_{AllenCahn}
+\operatorname{Err}_{instab}+\operatorname{Err}_{dt\xi}+\operatorname{Err}_{MC\xi}+\operatorname{Err}_{OtherPhases}
\end{align}
where we have abbreviated
\begin{subequations}
\begin{align}
&\operatorname{Err}_{instab}
:=\notag
\\&~~~
\int_{\R^d} (\nabla \cdot B) \Big( \frac{\vareps}{2} |\nabla u_\vareps|^2 + \frac1\vareps W(u_\vareps) + \sum^N_{i=1} \xi_i \cdot \nabla (\psi_i \circ u_\vareps)\Big) \dx
\label{Errinstab}
\\
&~~~
- \sumij \int_{\R^d}\frac12 \nabla B : \left( \xi_{i,j} - \frac{\nabla (\psi_{i,j}  \circ u_\vareps)}{|\nabla (\psi_{i,j}  \circ u_\vareps)|}  \right)
\otimes  \left( \xi_{i,j} - \frac{\nabla (\psi_{i,j}  \circ u_\vareps)}{|\nabla (\psi_{i,j}  \circ u_\vareps)|}  \right)
\notag
\\
&~~~~~~ \hspace{6.7cm}
\times 
|\nabla (\psi_{i,j}  \circ u_\vareps)| \chit(u_\vareps)  \dx
\notag
\end{align}
and
\begin{align}
&\operatorname{Err}_{AllenCahn}
:=
\notag
\\&~~~~~~~~~~~~
\label{ErrAllenCahn}
\sumij  \int_{\R^d}\nabla B : \bigg( \frac{\nabla (\psi_{i,j}  \circ u_\vareps)}{|\nabla (\psi_{i,j}  \circ u_\vareps)|}\otimes \frac{\nabla (\psi_{i,j}  \circ u_\vareps)}{|\nabla (\psi_{i,j}  \circ u_\vareps)|} \frac{1}{2} |\nabla (\psi_{i,j}  \circ u_\vareps)|
\\&~~~~~~~~~~~~~~~~~~~~~~~~~~~~~~~~~~~~~~~~~~~~~~~~~~~~~~~~~~~~~~~~
- \vareps \nabla u_\vareps^\mathsf{T} \nabla u_\vareps \bigg) \chit (u_\vareps) \dx
\notag
\end{align}
and
\begin{align}
\operatorname{Err}_{dt\xi}
:=
\label{Errdtxi}
&\sumij  \int_{\R^d} \frac12 (\partial_t \xi_{i,j} + (B\cdot \nabla ) \xi_{i,j}  + (\nabla B)^\mathsf{T}\xi_{i,j} )
\\
& \hspace{2cm}\cdot \left( \xi_{i,j} - \frac{\nabla (\psi_{i,j}  \circ u_\vareps)}{|\nabla (\psi_{i,j}  \circ u_\vareps)|}  \right)   |\nabla (\psi_{i,j}  \circ u_\vareps)| \chit(u_\vareps) \dx \notag
\\&
- \sumij  \int_{\R^d} \frac12  \xi_{i,j} \cdot ( \partial_t \xi_{i,j} + (B\cdot \nabla ) \xi_{i,j}   ) |\nabla (\psi_{i,j}  \circ u_\vareps)| \chit (u_\vareps) \dx \notag
\end{align}
as well as
\begin{align}
\nonumber
\operatorname{Err}_{MC\xi}
:=&
\int_{\R^d} \frac{1}{2 \vareps}  \left| \sum^N_{i=1} (\nabla \cdot \xi_i) \partial_u \psi_i (u_\vareps) \right|^2 \dx
\\&~~~
\label{Errmcxi}
- \sum^N_{i=1} \int_{\R^d}  (\nabla \cdot \xi_i ) B  \cdot  \nabla(\psi_i \circ u_\vareps) \dx
\\&~~~
\nonumber
+ \sumij  \int_{\R^d} \frac{\vareps}{2} |B \cdot {\xi}_{i,j}|^2 |\xi_{i,j}|^2 |\nabla u_\vareps|^2 \chit (u_\vareps)\dx 
\\&~~~
\nonumber
+ \sumij  \int_{\R^d} (\Id - \xi_{i,j} \otimes \xi_{i,j}):\h_\vareps \otimes B    |\nabla u_\vareps | \chit(u_\vareps)\dx
\end{align}
and
\begin{align}
\nonumber
&\operatorname{Err}_{OtherPhases}
:=
\\ \notag
&~
\sumkij \int_{\R^d} \frac12  (\partial_t \xi_{k} + (B\cdot \nabla ) \xi_{k}  + (\nabla B)^\mathsf{T}\xi_{k} )\cdot  \nabla (\psi_{0} \circ u_\vareps ) \chit (u_\vareps) \dx 
\\&
\label{ErrOtherPhases}
~-\sumkij  \int_{\R^d}\frac12 \nabla B :  \nabla (\psi_0 \circ u_\vareps ) \otimes \xi_{k} \chit(u_\vareps)  \dx
\\&
\nonumber
~ -\sumkij  \int_{\R^d} \frac12\nabla B : \xi_{k}\otimes  \nabla (\psi_{0} \circ u_\vareps )   \chit(u_\vareps)\dx. \notag
\end{align}
\end{subequations}
\end{lemma}

\section{The relative energy argument}

\subsection{Derivation of the Gronwall inequality for the relative entropy}

We first show how the evolution estimate for the relative entropy from Lemma~\ref{LemmaRelativeEnergyEvolution} and the coercivity properties of our relative entropy together imply a Gronwall-type estimate for the evolution of the relative entropy.

\begin{proof}[Proof of Theorem~\ref{TheoremRelativeEnergyInequality}]
We proceed by estimating the terms on the right-hand side of the equation \eqref{eq:timeevolutionrelativeenergy} for the time evolution of the relative energy. Note that it will be sufficient to prove
\begin{align*}
&\operatorname{Err}_{AllenCahn} + 
\operatorname{Err}_{instab}+\operatorname{Err}_{dt\xi}+
\operatorname{Err}_{MC\xi}+\operatorname{Err}_{OtherPhases}
\\&
\leq C(\xi(t),B(t),\bar\delta ) E[u_\vareps|\xi]
\\&~~~
+\bar\delta  \sumij \int_{\R^d} \frac{1}{2 \vareps}\big|\h_\vareps - \vareps  (B \cdot {\xi}_{i,j}){\xi}_{i,j} |\nabla u_\vareps| \big|^2  \chit (u_\vareps)\dx
\\&~~~
+ \bar\delta \int_{\R^d} \frac{1}{2 \vareps}  \left| \Big(\vareps \Delta u_\vareps - \frac{1}{\vareps} \partial_u W(u_\vareps)\Big) + \sum^N_{i=1} (\nabla \cdot \xi_i) \partial_u \psi_i (u_\vareps) \right|^2 \dx
\end{align*}
for any $\bar\delta >0$, as then an absorption argument applied to \eqref{eq:timeevolutionrelativeenergy} (for $\bar\delta<\tfrac{1}{4}$) yields
\begin{align*}
\ddt E[u_\varepsilon| \xi] \leq C(\xi(t),B(t),\bar\delta) E[u_\vareps| \xi].
\end{align*}
The Gronwall inequality then implies our conclusion.

\emph{Step 1: Estimates for $\operatorname{Err}_{OtherPhases}$, $\operatorname{Err}_{dt\xi}$, and $\operatorname{Err}_{instab}$.} We first show that
\begin{align*}
\operatorname{Err}_{instab}+\operatorname{Err}_{dt\xi}+\operatorname{Err}_{OtherPhases}
\leq C(\xi(t),B(t)) E[u_\vareps| \xi].
\end{align*}
Indeed, it is immediate by the definition \eqref{ErrOtherPhases} and the coercivity property \eqref{eq:coergradpsik} of our relative energy that the inequality
\begin{align*}
\operatorname{Err}_{OtherPhases}
&\leq \sumij C(\xi(t),B(t))  \int_{\R^d} |\nabla \psi_0| \chit (u_\vareps) \dx
\\&
\leq C(\xi(t),B(t)) E[u_\vareps|\xi]
\end{align*}
holds.
Using the defining properties \eqref{eq:timeevxiij} and \eqref{eq:lenghtxiij} of the calibration $\xi$ and the coercivity properties \eqref{eq:coernormalpsi} and \eqref{eq:coerdistpsi} of our relative energy, we likewise deduce from the definition \eqref{Errdtxi} that $\operatorname{Err}_{dt\xi}\leq C(\xi(t),B(t)) E[u_\vareps|\xi]$, using for instance the estimate
\begin{align*}
&\sumij  \int_{\R^d} \frac12 (\partial_t \xi_{i,j} + (B\cdot \nabla ) \xi_{i,j}  + (\nabla B)^\mathsf{T}\xi_{i,j} )
\\
& \hspace{2cm}\cdot \left( \xi_{i,j} - \frac{\nabla (\psi_{i,j}  \circ u_\vareps)}{|\nabla (\psi_{i,j}  \circ u_\vareps)|}  \right)   |\nabla (\psi_{i,j}  \circ u_\vareps)| \chit (u_\vareps) \dx
\\&
\leq
\sumij \bigg(\int_{\R^d} \big|\partial_t \xi_{i,j} + (B\cdot \nabla ) \xi_{i,j}  + (\nabla B)^\mathsf{T}\xi_{i,j}\big|^2 |\nabla (\psi_{i,j}  \circ u_\vareps)| \chit (u_\vareps) \dx \bigg)^{1/2}
\\
&~~~~~~~~~~~~~~~~
\times
\bigg(\int_{\R^d} \bigg|\xi_{i,j} - \frac{\nabla (\psi_{i,j}  \circ u_\vareps)}{|\nabla (\psi_{i,j}  \circ u_\vareps)|}\bigg|^2 |\nabla (\psi_{i,j}  \circ u_\vareps)| \chit (u_\vareps) \dx \bigg)^{1/2}
\\&
\stackrel{\eqref{eq:timeevxiij},\eqref{eq:coernormalpsi},\eqref{eq:coerdistpsi}}{\leq} C(\xi(t),B(t)) E[u_\vareps|\xi].
\end{align*}
Similarly, recalling the definition \eqref{Errinstab} and 
\eqref{eq:coernormalpsi}
as well as \eqref{eq:relativeenergyfirst}, we immediately get $\operatorname{Err}_{instab} \leq C(\xi(t),B(t)) E[u_\vareps|\xi]$. It therefore only remains to estimate $\operatorname{Err}_{AllenCahn}$ and $\operatorname{Err}_{MC\xi}$.

\emph{Step 2: Estimate for $\operatorname{Err}_{AllenCahn}$.} By adding zeroes, we may rewrite
\begin{align*}
&\operatorname{Err}_{AllenCahn}
\\&
= \sumij \int_{\R^d}\nabla B : \left( \frac{\nabla (\psi_{i,j}  \circ u_\vareps)}{|\nabla (\psi_{i,j}  \circ u_\vareps)|}\otimes \frac{\nabla (\psi_{i,j}  \circ u_\vareps)}{|\nabla (\psi_{i,j}  \circ u_\vareps)|}    - \xi_{i,j} \otimes\xi_{i,j}  \right)
\\&~~~~~~~~~~~~~~~~~~~~~
\hspace{1.4cm} \times \left(\frac12|\nabla (\psi_{i,j}  \circ u_\vareps) | - \vareps |\nabla u_\vareps|^2 \right) \chit (u_\vareps)\dx 
\\&~~~
+ \sumij \int_{\R^d}\nabla B : \xi_{i,j} \otimes\xi_{i,j} \left(\frac12|\nabla (\psi_{i,j}  \circ u_\vareps) | - \vareps |\nabla u_\vareps|^2 \right) \chit (u_\vareps)\dx
\\&~~~
+ \sumij \int_{\R^d}\nabla B : \left( \frac{\nabla (\psi_{i,j}  \circ u_\vareps)}{|\nabla (\psi_{i,j}  \circ u_\vareps)|}\otimes \frac{\nabla (\psi_{i,j}  \circ u_\vareps)}{|\nabla (\psi_{i,j}  \circ u_\vareps)|}   - \frac{ \nabla u_\vareps^\mathsf{T} \nabla u_\vareps }{|\nabla u_\vareps|^2} \right) \vareps |\nabla u_\vareps|^2 \chit (u_\vareps)\dx.
\end{align*}
The first term on the right-hand side can be bounded by $C(\xi(t),B(t))E[u_\vareps|\xi]$ by writing
\begin{align*}
	 &\frac{\nabla (\psi_{i,j}  \circ u_\vareps)}{|\nabla (\psi_{i,j}  \circ u_\vareps)|}\otimes \frac{\nabla (\psi_{i,j}  \circ u_\vareps)}{|\nabla (\psi_{i,j}  \circ u_\vareps)|}    - \xi_{i,j} \otimes\xi_{i,j}  \\&= \left( \frac{\nabla (\psi_{i,j}  \circ u_\vareps)}{|\nabla (\psi_{i,j}  \circ u_\vareps)|} - \xi_{i,j} \right)\otimes \frac{\nabla (\psi_{i,j}  \circ u_\vareps)}{|\nabla (\psi_{i,j}  \circ u_\vareps)|}    + \xi_{i,j} \otimes \left( \frac{\nabla (\psi_{i,j}  \circ u_\vareps)}{|\nabla (\psi_{i,j}  \circ u_\vareps)|}- \xi_{i,j} \right)
\end{align*} 
and using Young's inequality together with the coercivity estimates \eqref{eq:coernormalu} and \eqref{eq:coersurfmeaslenght} for our relative energy.
The second term on the right-hand side in the above formula can be estimated similarly by exploiting Young's inequality as well as the gradient flow calibration property \eqref{eq:gradBnormal} and the coercivity estimates \eqref{eq:coerdist} and \eqref{eq:coersurfmeaslenght}.

It remains to bound the third term on the right-hand side.
To this aim, we note that for any symmetric matrix $A$ we have
\begin{align*}
\nabla B : A = & \;\, \nabla B : (\Id-\xi_{i,j} \otimes \xi_{i,j})A(\Id-\xi_{i,j}  \otimes \xi_{i,j})
\\&
+ (\Id-\xi_{i,j} \otimes \xi_{i,j}) \big( \nabla B + (\nabla B)^T \big) \xi_{i,j} \cdot (\xi_{i,j}\cdot A (\Id-\xi_{i,j}  \otimes \xi_{i,j}) )
\\&
+ \xi_{i,j}\cdot (\Id-\xi_{i,j} \otimes \xi_{i,j}) \big( \nabla B + (\nabla B)^T \big)\xi_{i,j} (\xi_{i,j}\cdot A \xi_{i,j})
\\&
+ \xi_{i,j} \cdot \nabla B \,\xi_{i,j} (\xi_{i,j}\cdot A \xi_{i,j}).
\end{align*}
This entails by \eqref{eq:gradBmixed} and $|\xi_{i,j}(\Id-\xi_{i,j}\otimes\xi_{i,j})|\leq C \dist^2(\cdot,I_{i,j})$ (the latter being a consequence of \eqref{eq:lenghtconstrxiij})
\begin{align*}
&\sumij \int_{\R^d}\nabla B : \left( \frac{\nabla (\psi_{i,j}  \circ u_\vareps)}{|\nabla (\psi_{i,j}  \circ u_\vareps)|}\otimes \frac{\nabla (\psi_{i,j}  \circ u_\vareps)}{|\nabla (\psi_{i,j}  \circ u_\vareps)|}   - \frac{ \nabla u_\vareps^\mathsf{T} \nabla u_\vareps }{|\nabla u_\vareps|^2} \right) \vareps |\nabla u_\vareps|^2 \chit (u_\vareps)\dx
\\&
\leq
C(\xi(t),B(t))
\sumij \int_{\R^d} \bigg( \bigg|(\Id-\xi_{i,j}\otimes \xi_{i,j}) \frac{\nabla (\psi_{i,j}  \circ u_\vareps)}{|\nabla (\psi_{i,j}  \circ u_\vareps)|} \bigg|^2 \vareps |\nabla u_\vareps|^2
+ \eps |(\Id-\xi_{i,j}\otimes \xi_{i,j})\nabla u_\vareps^\mathsf{T}|^2 \bigg) \chit (u_\vareps)\dx
\\&~~~
+C(\xi(t),B(t))
\sumij \int_{\R^d} \min \{ \dist(x, I_{i,j}) ,1\} \bigg( \bigg|(\Id-\xi_{i,j}\otimes \xi_{i,j}) \frac{\nabla (\psi_{i,j}  \circ u_\vareps)}{|\nabla (\psi_{i,j}  \circ u_\vareps)|} \bigg| \vareps |\nabla u_\vareps|^2 	\\
& \hspace{7.5cm}+ \vareps |(\Id-\xi_{i,j}\otimes \xi_{i,j})\nabla u_\vareps^\mathsf{T}| |\nabla u_\vareps| \bigg) \chit (u_\vareps)\dx
\\&~~~
+C(\xi(t),B(t))
\sumij \int_{\R^d} \min \{\dist^2(x, I_{i,j}) ,1\} \vareps |\nabla u_\vareps|^2 \chit (u_\vareps)\dx
\\&~~~
+C(\xi(t),B(t))
\sumij \int_{\R^d}  \Bigg| \bigg|\xi_{i,j} \cdot \frac{\nabla (\psi_{i,j}  \circ u_\vareps)}{|\nabla (\psi_{i,j}  \circ u_\vareps)|} \bigg|^2 \vareps |\nabla u_\vareps|^2 - \vareps |(\xi_{i,j} \cdot \nabla) u_\vareps|^2 \Bigg| \chit (u_\vareps)\dx
\\&
\leq
C(\xi(t),B(t))
\sumij \int_{\R^d} \bigg( \bigg|(\Id-\xi_{i,j}\otimes \xi_{i,j}) \frac{\nabla (\psi_{i,j}  \circ u_\vareps)}{|\nabla (\psi_{i,j}  \circ u_\vareps)|} \bigg|^2 \vareps |\nabla u_\vareps|^2
+ \eps |(\Id-\xi_{i,j}\otimes \xi_{i,j})\nabla u_\vareps^\mathsf{T}|^2 \bigg) \chit (u_\vareps)\dx
\\&~~~
+C(\xi(t),B(t))
\sumij \int_{\R^d} \Bigg( \bigg(1-\bigg|\xi_{i,j} \cdot \frac{\nabla (\psi_{i,j}  \circ u_\vareps)}{|\nabla (\psi_{i,j}  \circ u_\vareps)|} \bigg|^2 \bigg) \vareps |\nabla u_\vareps|^2\\
& \hspace{4.5cm} + \vareps \big(|\nabla u_\vareps|^2-|(\xi_{i,j} \cdot \nabla) u_\vareps|^2 \big)\Bigg) \chit (u_\vareps) \dx
\\&~~~
+C(\xi(t),B(t))
\sumij \int_{\R^d} \min \{\dist^2(x, I_{i,j}) ,1\} \vareps |\nabla u_\vareps|^2 \chit (u_\vareps)\dx,
\end{align*}
where in the last step we have used Young's inequality.
By the coercivity properties \eqref{eq:coerdist}, \eqref{eq:coernormalu}, \eqref{eq:coerxigradu}, and \eqref{eq:coertangentgradu}, we conclude that
\begin{align*}
\operatorname{Err}_{AllenCahn}
\leq
C(\xi(t),B(t))E[u_\vareps|\xi].
\end{align*}

\emph{Step 3: Estimate for $\operatorname{Err}_{MC\xi}$.}
For the estimate on $\operatorname{Err}_{MC\xi}$, we have to work a bit more. We begin by adding zeroes and using \eqref{eq:sum3}
to obtain
\begin{align}
&\operatorname{Err}_{MC\xi}
\notag
\\&
\leq
\sumij
\int_{\R^d} \frac{1}{2 \vareps}  \bigg| - \frac{1}{2} (\nabla \cdot \xi_{i,j}) \partial_u \psi_{i,j} (u_\vareps)  + \sumk \frac{1}{2} (\nabla \cdot \xi_{k}) \partial_u \psi_0 ( u_\vareps)  \bigg|^2    \chit (u_\vareps) \dx \notag
\\&~~~
+\sumij \int_{\R^d}  \frac12 (\nabla \cdot \xi_{i,j}) B  \cdot  \nabla(\psi_{i,j} \circ u_\vareps) \chit (u_\vareps)\dx \notag
\\&~~~
+ \sumij \int_{\R^d} \frac{\vareps}{2} |B \cdot {\xi}_{i,j}|^2|\nabla u_\vareps|^2 \chit (u_\vareps)\dx \notag
\\&~~~
+ \sumij \int_{\R^d} (\Id - \xi_{i,j} \otimes \xi_{i,j}):\h_\vareps \otimes B     |\nabla u_\vareps | \chit (u_\vareps)\dx \notag
\\&~~~
- \sumkij \int_{\R^d}  \frac12(\nabla \cdot \xi_k ) B  \cdot  \nabla(\psi_0 \circ u_\vareps)   \chit (u_\vareps)\dx \notag
\\&
= \sumij  \int_{\R^d} \frac{1}{2 }  \bigg|
\frac{1}{\sqrt{\vareps}} (\nabla \cdot \xi_{i,j}) \frac{\nabla (\psi_{i,j}  \circ u_\vareps)}{|\nabla (\psi_{i,j}  \circ u_\vareps)|} \otimes \frac12 \partial_u \psi_{i,j} (u_\vareps)  \otimes  \frac{\nabla (\psi_{i,j}  \circ u_\vareps)}{|\nabla (\psi_{i,j}  \circ u_\vareps)|}
 \notag
\\&~~~
\hspace{6cm}
+ \sqrt{\vareps} (B\cdot {\xi}_{i,j}) {\xi}_{i,j} \otimes \nabla u_\vareps  \bigg|^2   \chit (u_\vareps) \dx \notag
\\&~~~
+ \sumij \int_{\R^d} \frac{1}{2 \vareps}  \bigg| \sumk \frac{1}{2} (\nabla \cdot \xi_{k}) \partial_u\psi_0 (u_\vareps)  \bigg|^2   \chit (u_\vareps) \dx \notag
\\&~~~
- \sumkij  \int_{\R^d} \frac{1}{4 \vareps}   (\nabla \cdot \xi_{i,j})  (\nabla \cdot \xi_{k}) \partial_u \psi_{i,j} (u_\vareps)  \cdot  \partial_u \psi_0 ( u_\vareps)  \chit (u_\vareps) \dx \notag
\\&~~~
+ \sumij \int_{\R^d}  \frac12 (\nabla \cdot \xi_{i,j}) (\Id - \xi_{i,j} \otimes \xi_{i,j}):B \otimes  \nabla(\psi_{i,j} \circ u_\vareps)\chit (u_\vareps) \dx \notag
\\&~~~
+ \sumij \int_{\R^d} \frac{\vareps}{2} (1-|\xi_{i,j}|^2) |B \cdot {\xi}_{i,j}|^2 |\nabla u_\vareps|^2 \chit (u_\vareps)\dx \notag
\\&~~~
+\sumij  \int_{\R^d} (\Id - \xi_{i,j} \otimes \xi_{i,j}):\h_\vareps \otimes B   |\nabla u_\vareps | \chit (u_\vareps)\dx\notag
\\&~~~
-  \sumkij \int_{\R^d}  \frac12(\nabla \cdot \xi_k ) B  \cdot  \nabla(\psi_0 \circ u_\vareps) \chit (u_\vareps)  \dx, \label{eq:relentineq1}
\end{align}
where in the second step we have also used $\partial_u \psi_{i,j}(u_\vareps) \otimes \frac{\nabla (\psi_{i,j}\circ u_\vareps)}{|\nabla (\psi_{i,j}\circ u_\vareps)|}:\nabla u_\eps=|\nabla (\psi_{i,j}\circ u_\vareps)|$.

Now note that the three terms on the right-hand side of \eqref{eq:relentineq1} that involve a $\partial_u \psi_0(u_\vareps)$ or $\nabla (\psi_0\circ u_\vareps)$ can be directly estimated by $CE[u_\vareps|\xi]$ by relying on the coercivity properties \eqref{eq:coerpartialpsik}, \eqref{eq:coerpartialpsiijk}, and \eqref{eq:coergradpsik}. Similarly, the third-to-last term on the right-hand side is estimated by $CE[u_\vareps|\xi]$ using \eqref{eq:lenghtconstrxiij} and \eqref{eq:coerdist}. This shows
\begin{align}
&\operatorname{Err}_{MC\xi}
\\
&\leq
\sumij  \int_{\R^d} \frac{1}{2 }  \bigg|
\frac{1}{\sqrt{\vareps}} (\nabla \cdot \xi_{i,j}) \frac{\nabla (\psi_{i,j}  \circ u_\vareps)}{|\nabla (\psi_{i,j}  \circ u_\vareps)|} \otimes \frac12 \partial_u \psi_{i,j} (u_\vareps)  \otimes  \frac{\nabla (\psi_{i,j}  \circ u_\vareps)}{|\nabla (\psi_{i,j}  \circ u_\vareps)|}
 \notag
\\&~~~
\hspace{6cm}
+ \sqrt{\vareps} (B\cdot {\xi}_{i,j}) {\xi}_{i,j} \otimes \nabla u_\vareps  \bigg|^2   {\chit} (u_\vareps) \dx \notag
\\&~~~
+ \sumij \int_{\R^d}  \frac12 (\nabla \cdot \xi_{i,j}) (\Id - \xi_{i,j} \otimes \xi_{i,j}):B \otimes  \nabla(\psi_{i,j} \circ u_\vareps) {\chit}(u_\vareps) \dx \notag
\\&~~~
+\sumij  \int_{\R^d} (\Id - \xi_{i,j} \otimes \xi_{i,j}):\h_\vareps \otimes B   |\nabla u_\vareps | {\chit} (u_\vareps)\dx\notag
\\&~~~
+CE[u_\vareps|\xi]
.
\label{eq:relentineq1b}
\end{align}
By adding zeros, the first term on the right-hand side of \eqref{eq:relentineq1b} can be rewritten as
\begin{align*}
\sumij  \int_{\R^d}& \frac{1}{2 }  \bigg| [(B\cdot {\xi}_{i,j}) {\xi}_{i,j} + (\nabla \cdot \xi_{i,j})\xi_{i,j}]\otimes \sqrt{\vareps}  \nabla u_\varepsilon
\\&
+(\nabla \cdot \xi_{i,j}) \frac{\nabla (\psi_{i,j}  \circ u_\vareps)}{|\nabla (\psi_{i,j}  \circ u_\vareps)|}\otimes \left(\frac{1}{2\sqrt{\vareps}}  \partial_u \psi_{i,j} (u_\vareps)  \otimes  \frac{\nabla (\psi_{i,j}  \circ u_\vareps)}{|\nabla (\psi_{i,j}  \circ u_\vareps)|} - \sqrt{\vareps} \nabla u_\eps\right)
\\&
+(\nabla \cdot \xi_{i,j}) \left(\frac{\nabla (\psi_{i,j}  \circ u_\vareps)}{|\nabla (\psi_{i,j}  \circ u_\vareps)|} - \xi_{i,j} \right) \otimes \sqrt{\vareps} \nabla u_\vareps \bigg|^2   {\chit} (u_\vareps) \dx
\\
\leq  
\sumij \frac{3}{2}  &\int_{\R^d} \left| (B\cdot {\xi}_{i,j}) {\xi}_{i,j} + (\nabla \cdot \xi_{i,j})\xi_{i,j} \right|^2 \vareps |\nabla u_\vareps|^2  {\chit}(u_\vareps) \dx
\\
+ \sumij \frac{3}{2 }&  \|\nabla \cdot \xi_{i,j} \|^2_{L_x^\infty}  \int_{\R^d} \left|  \frac{1}{2\sqrt{\vareps}}  \partial_u \psi_{i,j} (u_\vareps)  \otimes  \frac{\nabla (\psi_{i,j}  \circ u_\vareps)}{|\nabla (\psi_{i,j}  \circ u_\vareps)|} - \sqrt{\vareps} \nabla u_\eps \right|^2 {\chit} (u_\vareps) \dx 
\\
+ \sumij  \frac{3}{2 } & \|\nabla \cdot \xi_{i,j} \|^2_{L_x^\infty}  \int_{\R^d}    \left|\frac{\nabla (\psi_{i,j}  \circ u_\vareps)}{|\nabla (\psi_{i,j}  \circ u_\vareps)|} - \xi_{i,j} \right|^2 \vareps |\nabla u_\vareps|^2  {\chit} (u_\vareps) \dx,
\end{align*}
where in the second step we have used Young's inequality. Using the property \eqref{eq:normalB} of the calibration $(\xi,B)$ and the coercivity propeties \eqref{eq:coerdist}, \eqref{eq:coersurfmeas} and \eqref{eq:coernormalu} of the relative energy, we see that the right-hand side is bounded by $CE[u_\vareps|\xi]$.
   
It remains to estimate the second and third term on the right-hand side of \eqref{eq:relentineq1b}. Adding zero, these terms are seen to be equal to
\begin{align*}
&\sumij \int_{\R^d} (\Id - \xi_{i,j} \otimes \xi_{i,j}): B  \otimes \nabla u_\vareps^\mathsf{T} \cdot \bigg( \frac12 (\nabla \cdot \xi_{i,j}) \partial_u \psi_{i,j} (u_\vareps)
\\&
 \hspace{1.5cm}- \sumk \frac12 (\nabla \cdot \xi_{k})  \partial_u \psi_{0} (u_\vareps) - \Big( \vareps \Delta u_\vareps - \frac{1}{\vareps} \partial_u W(u_\vareps )\Big)\bigg) {\chit} (u_\vareps) \dx
\\
+&\sumkij \int_{\R^d}  \frac12 (\nabla \cdot \xi_{k})    (\Id - \xi_{i,j} \otimes \xi_{i,j}): B  \otimes \nabla u_\vareps^\mathsf{T} \cdot  \partial_u \psi_{0} (u_\vareps)  {\chit} (u_\vareps) \dx
\\
\stackrel{\eqref{eq:sum3}}{\leq} &\,\sumij C(\bar\delta) \| B \|^2_{L_x^\infty}  \int_{\R^d}   \varepsilon |(\Id-\xi_{i,j}\otimes \xi_{i,j}) \nabla u_\vareps^\mathsf{T} |^2 {\chit} (u_\vareps) \dx
\\&
+ \bar\delta \int_{\R^d} \frac{1}{2\vareps} \Big| \Big(\vareps \Delta u_\vareps - \frac{1}{\vareps} \partial_u W(u_\vareps)\Big) + \sum^N_{i=1} (\nabla \cdot \xi_i) \partial_u \psi_i (u_\vareps) \Big|^2 \dx
\\&
+ \sumkij \frac{1}{2 }  \|\nabla \cdot \xi_{k} \|_{L_x^\infty} \| B \|_{L_x^\infty}  \int_{\R^d}  | \nabla ( \psi_{0} \circ u_\vareps) | {\chit} (u_\vareps) \dx.
\end{align*}
Here,
in the last step we have used Young's inequality for $\bar\delta>0$ small enough.
Using the coercivity properties \eqref{eq:coertangentgradu}, \eqref{eq:coerdist}, and \eqref{eq:coergradpsik} of the relative energy, we see that the first and last term on the right-hand side are bounded by $CE[u_\vareps|\xi]$.

Overall, we have shown
\begin{align*}
\operatorname{Err}_{MC\xi}\leq C(\bar\delta) E[u_\vareps|\xi]
+ \bar\delta \int_{\R^d} \frac{1}{2\vareps} \Big| \Big(\vareps \Delta u_\vareps - \frac{1}{\vareps} \partial_u W(u_\vareps)\Big) + \sum^N_{i=1} (\nabla \cdot \xi_i) \partial_u \psi_i (u_\vareps) \Big|^2 \dx,
\end{align*}
which was the only missing ingredient for the proof of the theorem.
\end{proof}

In the above estimates, we have used the following additional coercivity properties of the relative entropy. We shall defer their proof to that of the other coercivity properties from Proposition~\ref{prop:coerpartialpsi} and \ref{prop:maincoer}.
\begin{lemma}\label{lemma:additionalcoer}
Let $W$, $\psi_i$, $\psi_{i,j}$, $\xi_i$, and $\xi_{i,j}$ be as in Proposition~\ref{prop:coerpartialpsi}. We then have
	\begin{align}
	\label{eq:coernormalu}
	\sumij 	\int_{\R^d}    \left|\frac{\nabla (\psi_{i,j}  \circ u_\vareps)}{|\nabla (\psi_{i,j}  \circ u_\vareps)|} - \xi_{i,j} \right|^2
	\vareps |\nabla u_\vareps|^2  {\chit}(u_\vareps) \dx
	&	\leq C E[u_\varepsilon|\xi] ,\\
	 \label{eq:coersurfmeaslenght}
	\sumij  \int_{\R^d}  
	\left|  \frac{1}{2\sqrt{\vareps}}   \frac{|\nabla (\psi_{i,j}  \circ u_\vareps)|}{|\nabla u_\vareps|} - \sqrt{\vareps} |\nabla u_\eps | \right|^2   {\chit}(u_\vareps) \dx & \leq C E[u_\varepsilon|\xi] ,	
	\\ \label{eq:coersurfmeas}
	\sumij  \int_{\R^d}  
	\left|  \frac{1}{2\sqrt{\vareps}}  \partial_u \psi_{i,j} (u_\vareps)  \otimes  \frac{\nabla (\psi_{i,j}  \circ u_\vareps)}{|\nabla (\psi_{i,j}  \circ u_\vareps)|} - \sqrt{\vareps} \nabla u_\eps \right|^2  {\chit}(u_\vareps) \dx 
	&	\leq C E[u_\varepsilon|\xi] ,\\
		 \label{eq:coerxigradu}
		\sumij   \int_{\R^d} \left|  \xi_{i,j} \otimes \xi_{i,j} 
		- \frac{ \nabla u_\vareps^\mathsf{T} \nabla u_\vareps }{|\nabla u_\vareps|^2}
		\right|^2 \vareps |\nabla u_\vareps|^2 {\chit}(u_\vareps)\dx
		&\leq C E[u_\varepsilon|\xi] .
	\end{align}
\end{lemma}

\subsection{Time evolution of the relative energy}

We next give the technical computation that provides the estimate for the evolution of the relative entropy stated in Lemma~\ref{LemmaRelativeEnergyEvolution}. Although in parts technical, it is at the very heart of the proof of our results.
\begin{proof}[Proof of Lemma~\ref{LemmaRelativeEnergyEvolution}]
	By direct computations, using the definitions \eqref{eq:relativeenergyfirst} and \eqref{eq:AC} as well as (an analogue for $\partial_t \xi_i$ of) the relation \eqref{eq:sum}, we obtain
	\begin{align*}
			\ddt E[u_\varepsilon|\xi]
			=&
			- \int_{\R^d} \left( \vareps \Delta u_\vareps - \frac{1}{\vareps} \partial_u W(u_\vareps)\right) \partial_t u_\vareps  \dx
			\\
			& - \sum^N_{i=1} \int_{\R^d} (\nabla \cdot \xi_i) \partial_u \psi_i (u_\vareps) \cdot \partial_t u_\vareps  \dx
			\\
			& + \sum^N_{i=1} \int_{\R^d} \partial_t \xi_i \cdot \nabla (\psi_i \circ u_\vareps) \dx
			\\
			=&
			- \int_{\R^d} \frac{1}{\vareps}\left| \vareps \Delta u_\vareps - \frac{1}{\vareps} \partial_u W(u_\vareps)\right|^2  \dx \\
			& - \sum^N_{i=1} \int_{\R^d} (\nabla \cdot \xi_i) \partial_u \psi_i (u_\vareps) \cdot \left( \Delta u_\vareps - \frac{1}{\vareps^2} \partial_u W(u_\vareps) \right)  \dx
			\\
			& - \sumij \int_{\R^d} \frac12 \partial_t \xi_{i,j} \cdot \nabla (\psi_{i,j} \circ u_\vareps ) \chit(u_\vareps) \dx \\
			& + \sumkij  \int_{\R^d} \frac12 \partial_t \xi_{k} \cdot   \nabla (\psi_{0} \circ u_\vareps ) \chit (u_\vareps) \dx.
	\end{align*}
	By adding zeros and using again \eqref{eq:sum} as well as \eqref{eq:sum2}, we get
	\begin{align}
	\nonumber
	&\ddt E[u_\varepsilon|\xi]
	\\
\label{eqEvolutionRelEntropyFirstComputation}
	&= - \int_{\R^d} \frac{1}{\vareps}\left| \vareps \Delta u_\vareps - \frac{1}{\vareps} \partial_u W(u_\vareps)\right|^2 \dx \\ \nonumber
	&~~~ - \sum^N_{i=1} \int_{\R^d} (\nabla \cdot \xi_i) \partial_u \psi_i (u_\vareps) \cdot \left( \Delta u_\vareps - \frac{1}{\vareps^2} \partial_u W(u_\vareps) \right) \dx
	\\ 	\nonumber
	&~~~ - \sumij \int_{\R^d} \frac12 (\partial_t \xi_{i,j} + (B\cdot \nabla ) \xi_{i,j}  + (\nabla B)^\mathsf{T}\xi_{i,j} )\cdot \nabla (\psi_{i,j} \circ u_\vareps ) \chit (u_\vareps) \dx \\ \nonumber
	&~~~ + \sumkij \int_{\R^d} \frac12  (\partial_t \xi_{k} + (B\cdot \nabla ) \xi_{k}  + (\nabla B)^\mathsf{T}\xi_{k} )\cdot   \nabla (\psi_{0} \circ u_\vareps ) 	\chit (u_\vareps) \dx \\ \nonumber
	&~~~ - \sum^N_{i=1} \int_{\R^d}  \nabla \xi_i :  \nabla(\psi_i \circ u_\vareps) \otimes B \dx - \sum^N_{i=1} \int_{\R^d}   \nabla B : \xi_i \otimes \nabla (\psi_i \circ u_\vareps)  \dx .
	\end{align}
Integrating by parts several times and making use of an approximation argument for $((\xi_i)_i,B)$, the last two terms in the equation above can be rewritten as
	\begin{align*}
		 -\sum^N_{i=1} & \int_{\R^d}  \nabla \xi_i :  \nabla(\psi_i \circ u_\vareps) \otimes B \dx - \sum^N_{i=1}\int_{\R^d}   \nabla B : \xi_i \otimes \nabla (\psi_i \circ u_\vareps)  \dx \\
		 = & \sum^N_{i=1} \int_{\R^d}  \psi_i( u_\vareps) ( B \cdot \nabla ) ( \nabla \cdot \xi_i)   \dx 
		 +\sum^N_{i=1} \int_{\R^d} \psi_i( u_\vareps) (\nabla B)^\mathsf{T} : \nabla \xi_i \dx \\
		 &- \sum^N_{i=1} \int_{\R^d}   \nabla B : \xi_i \otimes \nabla (\psi_i \circ u_\vareps)  \dx \\
		  = & - \sum^N_{i=1} \int_{\R^d}   (\nabla \cdot \xi_i ) B  \cdot  \nabla(\psi_i \circ u_\vareps) \dx 
		 - \sum^N_{i=1} \int_{\R^d} \psi_i( u_\vareps) (\nabla \cdot B) (\nabla  \cdot \xi_i) \dx \\
		 & - \sum^N_{i=1} \int_{\R^d}  \psi_i( u_\vareps) ( \xi_i \cdot \nabla ) ( \nabla \cdot B )   \dx 
		 - \sum^N_{i=1} \int_{\R^d}  (\nabla B)^\mathsf{T} : \xi_i \otimes  \nabla(\psi_i \circ u_\vareps) \dx \\
		 &- \sum^N_{i=1} \int_{\R^d}   \nabla B : \xi_i \otimes \nabla (\psi_i \circ u_\vareps)  \dx \\
		 = &  \sum^N_{i=1} \int_{\R^d}  ((\nabla \cdot B ) \xi_i - (\nabla \cdot \xi_i ) B )  \cdot  \nabla(\psi_i \circ u_\vareps) \dx \\
		 & -\sum^N_{i=1} \int_{\R^d}   \nabla B : \nabla (\psi_i \circ u_\vareps) \otimes \xi_i  \dx 
		 - \sum^N_{i=1} \int_{\R^d}   \nabla B : \xi_i \otimes \nabla (\psi_i \circ u_\vareps)  \dx \\
		 \stackrel{\eqref{eq:sum}}{=} &  \sum^N_{i=1} \int_{\R^d}  ((\nabla \cdot B ) \xi_i - (\nabla \cdot \xi_i ) B )  \cdot  \nabla(\psi_i \circ u_\vareps) \dx \\
		 & + \sumij \int_{\R^d}\frac12 \nabla B : \nabla (\psi_{i,j}  \circ u_\vareps) \otimes \xi_{i,j}  \chit (u_\vareps)  \dx  \\
		 &+ \sumij \int_{\R^d} \frac12\nabla B : \xi_{i,j}\otimes \nabla (\psi_{i,j} \circ u_\vareps)   \chit (u_\vareps)\dx\\
		 & - \sumkij  \int_{\R^d}\frac12 \nabla B : \nabla (\psi_{0}  \circ u_\vareps)   \otimes \xi_{k}  \chit (u_\vareps)  \dx  \\
		 &- \sumkij \int_{\R^d} \frac12\nabla B : \xi_{k}\otimes  \nabla (\psi_{0}  \circ u_\vareps)   \chit (u_\vareps)\dx.
\end{align*}
By adding zero, we obtain
\begin{align}
\nonumber
 -&\sum^N_{i=1}  \int_{\R^d}  \nabla \xi_i :  \nabla(\psi_i \circ u_\vareps) \otimes B \dx - \sum^N_{i=1}  \int_{\R^d}   \nabla B : \xi_i \otimes \nabla (\psi_i \circ u_\vareps)  \dx \\ \label{eqRewriteTwoTerms}
		 = &  \sum^N_{i=1}  \int_{\R^d}  ((\nabla \cdot B ) \xi_i - (\nabla \cdot \xi_i ) B )  \cdot  \nabla(\psi_i \circ u_\vareps) \dx \\ \nonumber
		  & - \sumij \int_{\R^d}\frac12 \nabla B : \left( \xi_{i,j} - \frac{\nabla (\psi_{i,j}  \circ u_\vareps)}{|\nabla (\psi_{i,j}  \circ u_\vareps)|}  \right) \otimes  \left( \xi_{i,j} - \frac{\nabla (\psi_{i,j}  \circ u_\vareps)}{|\nabla (\psi_{i,j}  \circ u_\vareps)|}  \right) \\ \nonumber
		  & \hspace{7.5cm}
		   \times |\nabla (\psi_{i,j}  \circ u_\vareps)|  {\chit}(u_\vareps)  \dx \\ \nonumber
		  & + \sumij \int_{\R^d} \frac12\nabla B : \frac{\nabla (\psi_{i,j}  \circ u_\vareps)}{|\nabla (\psi_{i,j}  \circ u_\vareps)|}\otimes \frac{\nabla (\psi_{i,j}  \circ u_\vareps)}{|\nabla (\psi_{i,j}  \circ u_\vareps)|}  |\nabla (\psi_{i,j}  \circ u_\vareps)|  {\chit}(u_\vareps)\dx \\ \nonumber
		  &+ \sumij \int_{\R^d} \frac12\nabla B : \xi_{i,j} \otimes \xi_{i,j}   |\nabla (\psi_{i,j}  \circ u_\vareps)| {\chit}(u_\vareps)\dx \\ \nonumber
		   & -  \sumkij  \int_{\R^d}\frac12 \nabla B :  \nabla (\psi_{0}  \circ u_\vareps)  \otimes \xi_{k}   {\chit}(u_\vareps)  \dx  \\ \nonumber
		   &-  \sumkij \int_{\R^d} \frac12\nabla B : \xi_{k}\otimes  \nabla (\psi_{0}  \circ u_\vareps)   {\chit}(u_\vareps)\dx.
	\end{align}
	Using the relation
	\begin{align*}
		- &\int_{\R^d} \h_\vareps \cdot B |\nabla u_\vareps | \dx \\
		&\stackrel{\eqref{eq:h}}{=} - \int_{\R^d} \vareps \nabla B : \nabla u_\vareps^\mathsf{T} \nabla u_\vareps \dx + \int_{\R^d} (\nabla \cdot B) \left( \frac{\vareps}{2} |\nabla u_\vareps|^2 + \frac1\vareps W(u_\vareps)\right) \dx ,
	\end{align*}
in view of $\sum^N_{i,j=1:i < j}  {\chit}(u_\vareps) = 1$ we can rewrite the third term on the right-hand side as
	\begin{align*}
		&\sumij \int_{\R^d} \frac12\nabla B : \frac{\nabla (\psi_{i,j}  \circ u_\vareps)}{|\nabla (\psi_{i,j}  \circ u_\vareps)|}\otimes \frac{\nabla (\psi_{i,j}  \circ u_\vareps)}{|\nabla (\psi_{i,j}  \circ u_\vareps)|}  |\nabla (\psi_{i,j}  \circ u_\vareps)|  {\chit}(u_\vareps)\dx \\
		&=   \int_{\R^d} \h_\vareps \cdot B |\nabla u_\vareps | \dx + \int_{\R^d} (\nabla \cdot B) \left( \frac{\vareps}{2} |\nabla u_\vareps|^2 + \frac1\vareps W(u_\vareps)\right) \dx	 \\
		& \quad + \sumij \int_{\R^d}\nabla B : \bigg( \frac{\nabla (\psi_{i,j}  \circ u_\vareps)}{|\nabla (\psi_{i,j}  \circ u_\vareps)|}\otimes \frac{\nabla (\psi_{i,j}  \circ u_\vareps)}{|\nabla (\psi_{i,j}  \circ u_\vareps)|}  \frac12|\nabla (\psi_{i,j}  \circ u_\vareps) | -\vareps \nabla u_\vareps^\mathsf{T} \nabla u_\vareps  \bigg) \\
		& \hspace{10.5cm}\times  {\chit}(u_\vareps)\dx.
	\end{align*}
	Inserting this relation into \eqref{eqRewriteTwoTerms} and inserting the resulting equation into \eqref{eqEvolutionRelEntropyFirstComputation}, we obtain by collecting terms and adding and subtracting $\sum^N_{i,j=1:i < j} \int_{\R^d} \frac12  \xi_{i,j} \cdot ( \partial_t \xi_{i,j} + (B\cdot \nabla ) \xi_{i,j}   ) |\nabla (\psi_{i,j}  \circ u_\vareps)|  {\chit}(u_\vareps) \dx$
	\begin{align}
	&\ddt E[u_\varepsilon|\xi] \notag
	\\
	&= 	 - \int_{\R^d} \frac{1}{\vareps}\left| \vareps \Delta u_\vareps - \frac{1}{\vareps} \partial_u W(u_\vareps)\right|^2 \dx+  \int_{\R^d} \h_\vareps \cdot B |\nabla u_\vareps | \dx 
	\notag
	\\
	&\quad - \sum^N_{i=1} \int_{\R^d} (\nabla \cdot \xi_i) \partial_u \psi_i (u_\vareps) \cdot \left( \Delta u_\vareps - \frac{1}{\vareps^2} \partial_u W(u_\vareps) \right) \dx \notag
	\\
	&\quad- \sum^N_{i=1} \int_{\R^d}  (\nabla \cdot \xi_i ) B  \cdot  \nabla(\psi_i \circ u_\vareps) \dx \notag
	\\	
	&\quad+ \sumij  \int_{\R^d} \frac12 (\partial_t \xi_{i,j} + (B\cdot \nabla ) \xi_{i,j}  + (\nabla B)^\mathsf{T}\xi_{i,j} )\notag\\
	& \quad\hspace{2cm}\cdot \left( \xi_{i,j} - \frac{\nabla (\psi_{i,j}  \circ u_\vareps)}{|\nabla (\psi_{i,j}  \circ u_\vareps)|}  \right)   |\nabla (\psi_{i,j}  \circ u_\vareps)|  {\chit}(u_\vareps) \dx \notag
	\\
	& \quad-  \sumij \int_{\R^d} \frac12  \xi_{i,j} \cdot ( \partial_t \xi_{i,j} + (B\cdot \nabla ) \xi_{i,j}   ) |\nabla (\psi_{i,j}  \circ u_\vareps)|  {\chit}(u_\vareps) \dx \notag
	\\
	&\quad+  \int_{\R^d} (\nabla \cdot B) \Big( \frac{\vareps}{2} |\nabla u_\vareps|^2 + \frac1\vareps W(u_\vareps) + \sum_{i=1}^N \xi_i \cdot \nabla (\psi_i \circ u_\vareps)\Big) \dx	 \notag
	\\
	&\quad + \sumkij \int_{\R^d} \frac12  (\partial_t \xi_{k} + (B\cdot \nabla ) \xi_{k}  + (\nabla B)^\mathsf{T}\xi_{k} )\cdot   \nabla (\psi_{0}  \circ u_\vareps)  {\chit}(u_\vareps)  \dx \notag
	\\
	&\quad -\sumkij \int_{\R^d}\frac12 \nabla B : \nabla (\psi_{0}  \circ u_\vareps) \otimes \xi_{k}  {\chit}(u_\vareps)  \dx \notag
	\\ &\quad- \sumkij \int_{\R^d} \frac12\nabla B : \xi_{k}\otimes  \nabla (\psi_{0}  \circ u_\vareps) {\chit}(u_\vareps)\dx \notag
	\\
	& \quad-  \sumij \int_{\R^d}\frac12 \nabla B : \left( \xi_{i,j} - \frac{\nabla (\psi_{i,j}  \circ u_\vareps)}{|\nabla (\psi_{i,j}  \circ u_\vareps)|}  \right)\otimes  \left( \xi_{i,j} - \frac{\nabla (\psi_{i,j}  \circ u_\vareps)}{|\nabla (\psi_{i,j}  \circ u_\vareps)|}  \right) 
	  \notag 
	\\
	& \quad \hspace{8.0cm} \times |\nabla (\psi_{i,j}  \circ u_\vareps)| {\chit}(u_\vareps)  \dx \notag
	\\	
	& \quad +  \sumij \int_{\R^d}\nabla B : \bigg( \frac{\nabla (\psi_{i,j}  \circ u_\vareps)}{|\nabla (\psi_{i,j}  \circ u_\vareps)|}\otimes \frac{\nabla (\psi_{i,j}  \circ u_\vareps)}{|\nabla (\psi_{i,j}  \circ u_\vareps)|}  \frac12|\nabla (\psi_{i,j}  \circ u_\vareps) | -\vareps \nabla u_\vareps^\mathsf{T} \nabla u_\vareps  \bigg) \notag
	\\
	& \hspace{10.3cm} \times  {\chit}(u_\vareps)\dx . \notag
	\end{align}
Lemma~\ref{LemmaRelativeEnergyEvolution} follows from this equation using the definitions of the errors \eqref{Errinstab} --  \eqref{ErrOtherPhases} and the next formula (whose derivation relies on $\sum^N_{i,j=1:i < j} {\chit}(u_\vareps) = 1$ and repeated addition of zero).
	\begin{align*}
	&- \int_{\R^d} \frac{1}{\vareps}\left| \vareps \Delta u_\vareps - \frac{1}{\vareps} \partial_u W(u_\vareps)\right|^2 +  \int_{\R^d} \h_\vareps \cdot B |\nabla u_\vareps | \dx
	\\&
	- \sum^N_{i=1} \int_{\R^d} (\nabla \cdot \xi_i) \partial_u \psi_i (u_\vareps) \cdot \left( \Delta u_\vareps - \frac{1}{\vareps^2} \partial_u W(u_\vareps) \right) \dx
	\\&
	- \sum^N_{i=1} \int_{\R^d}  (\nabla \cdot \xi_i ) B  \cdot  \nabla(\psi_i \circ u_\vareps) \dx
	\\&
	= - \sumij  \int_{\R^d} \frac{1}{2 \vareps}\big|\h_\vareps - \vareps  (B \cdot {\xi}_{i,j}){\xi}_{i,j} |\nabla u_\vareps| \big|^2  {\chit}(u_\vareps)\dx
	\\&~~~
	- \int_{\R^d} \frac{1}{2 \vareps} \left( \Big|\vareps \Delta u_\vareps - \frac{1}{\vareps} \partial_u W(u_\vareps)\Big|^2 - |\h_\vareps|^2\right) \dx
	\\&~~~
	- \int_{\R^d} \frac{1}{2 \vareps}  \left|\vareps \Delta u_\vareps - \frac{1}{\vareps} \partial_u W(u_\vareps)\right|^2  \dx
	\\&~~~
	- \sum^N_{i=1}\int_{\R^d} (\nabla \cdot \xi_i) \partial_u \psi_i (u_\vareps) \cdot \left( \Delta u_\vareps - \frac{1}{\vareps^2} \partial_u W(u_\vareps) \right) \dx
	\\&~~~
	- \sum^N_{i=1} \int_{\R^d}  (\nabla \cdot \xi_i ) B  \cdot  \nabla(\psi_i \circ u_\vareps) \dx 
	\\&~~~
	+ \sumij  \int_{\R^d} \frac{\vareps}{2} |B \cdot {\xi}_{i,j}|^2  |\xi_{i,j}|^2 |\nabla u_\vareps|^2{\chit}(u_\vareps)\dx
	\\&~~~
	+ \sumij  \int_{\R^d} (\Id - \xi_{i,j} \otimes \xi_{i,j}):\h_\vareps \otimes B   |\nabla u_\vareps | {\chit}(u_\vareps)\dx
	\\&
	= - \sumij  \int_{\R^d} \frac{1}{2 \vareps}\big|\h_\vareps - \vareps  (B \cdot {\xi}_{i,j}){\xi}_{i,j} |\nabla u_\vareps| \big|^2  {\chit}(u_\vareps)\dx
	\\&~~~
	- \int_{\R^d} \frac{1}{2 \vareps}  \left| \Big(\vareps \Delta u_\vareps - \frac{1}{\vareps} \partial_u W(u_\vareps)\Big) + \sum^N_{i=1} (\nabla \cdot \xi_i) \partial_u \psi_i (u_\vareps) \right|^2 \dx
	\\&~~~
	- \int_{\R^d} \frac{1}{2 \vareps} \left( \Big|\vareps \Delta u_\vareps - \frac{1}{\vareps} \partial_u W(u_\vareps)\Big|^2 - |\h_\vareps|^2\right) \dx
	\\&~~~
	+  \int_{\R^d} \frac{1}{2 \vareps}  \left| \sum^N_{i=1}(\nabla \cdot \xi_i) \partial_u \psi_i (u_\vareps) \right|^2 \dx
	\\&~~~
	- \sum^N_{i=1} \int_{\R^d}  (\nabla \cdot \xi_i ) B  \cdot  \nabla(\psi_i \circ u_\vareps) \dx
	\\&~~~
	+ \sumij  \int_{\R^d} \frac{\vareps}{2} |B \cdot {\xi}_{i,j}|^2 |\xi_{i,j}|^2 |\nabla u_\vareps|^2 {\chit}(u_\vareps)\dx 
	\\&~~~
	+ \sumij  \int_{\R^d} (\Id - \xi_{i,j} \otimes \xi_{i,j}):\h_\vareps \otimes B    |\nabla u_\vareps | {\chit}(u_\vareps)\dx .
	\end{align*}
\end{proof}

\subsection{Derivation of the coercivity properties}

We next show how our assumption (A4) implies the coercivity properties of our relative entropy.
\begin{proof}[Proof of Proposition \ref{prop:coerpartialpsi}]
To prove \eqref{eq:coerpartialpsik}-\eqref{eq:coergradpsik}, let $i,j\in \{1,\ldots, N\}$, $i\neq j$; suppose that $(x,t)$ is such that $u_\vareps(x,t) \in \mathcal{T}_{i,j}$. In particular, we then have \(\chit(u_\vareps)=1\).
	
  \emph{Proof of \eqref{eq:coerpartialpsik} and \eqref{eq:coerpartialpsiijk}:} Starting from \eqref{eq:startcoer0}, expanding the second square, and making use of Young's inequality, the fact that for each $(x,t)$ there exists only one index $k\in \{1,\ldots,N\}\setminus \{i,j\}$ with $\xi_k(x,t)\neq 0$, and \eqref{AssumptionLengthXi}, we obtain
   \begin{align*}
   &\frac{\vareps}{2}|\nabla u_\vareps|^2
   +\frac{1}{\vareps} W(u_\vareps)
   +\sum^N_{\ell=1} \xi_\ell \cdot \nabla (\psi_\ell \circ u_\vareps)
   \\&
   \geq
   \frac{1}{2}\bigg|\sqrt{\vareps}\nabla u_\vareps - \frac{1}{\sqrt{\vareps}} \bigg(\tfrac{1}{2}\partial_u \psi_{i,j} (u_\vareps)\otimes \xi_{i,j} 
- \sumk \tfrac1{2\sqrt{3}} \partial_u \psi_0 (u_\vareps)
\otimes \sqrt{3}\xi_k 
\bigg) \bigg|^2
\\&~~~
+\frac{1}{2\vareps} \Bigg[ 2W(u_\vareps)-\bigg(|\xi_{i,j}|^2 + \frac1{\delta_{\text{cal}}}  \sum^N_{\substack{ k=1 \\ k \notin \{i,j\}}}  |\sqrt{3} \xi_{i,j} \cdot \xi_{k}|^2\bigg) \big|\tfrac{1}{2}\partial_u \psi_{i,j} (u_\vareps) \big|^2\\
   &\quad\quad \; \;  - \sumk  \left( |\sqrt{3}\xi_{k}|^2 + \delta_{\text{cal}}  \right)   \big|    \tfrac1{2\sqrt{3}} \partial_u \psi_0 (u_\vareps)\big|^2 
   \,\Bigg] .
   \end{align*}
   Then, by adding zeros and using $|\sqrt{3}\xi_k|\leq 1$, we obtain 
   \begin{align} \label{eq:startcoer}
   &\frac{\vareps}{2}|\nabla u_\vareps|^2
   +\frac{1}{\vareps} W(u_\vareps)
   +\sum^N_{\ell=1}  \xi_\ell \cdot \nabla (\psi_\ell \circ u_\vareps)
   \notag\\&
   \geq
   \frac{1}{2}\bigg|\sqrt{\vareps}\nabla u_\vareps - \frac{1}{\sqrt{\vareps}} \bigg(\tfrac{1}{2}\partial_u \psi_{i,j} (u_\vareps)\otimes \xi_{i,j} 
- \sumk \tfrac1{2\sqrt{3}} \partial_u \psi_0 (u_\vareps)
\otimes \sqrt{3}\xi_k 
\bigg) \bigg|^2
   \\&~~~
   +\frac{1}{2\vareps} \Bigg[ 2W(u_\vareps)- \big|\tfrac{1}{2}\partial_u \psi_{i,j}(u_\vareps)\big|^2 \notag\\
   & ~~~~~ \quad\quad \; \;  -( 1 +   \delta_{\text{cal}} + \delta_{\text{coer}, 1} ) \big|\tfrac{1}{2\sqrt{3}}\partial_u \psi_0(u_\vareps) \big|^2	\notag\\	
   &
   \quad~~ \quad\quad \; \; -  \delta_{\text{coer}, 2}    \left| \partial_u \psi_{i,j} (u_\vareps)\cdot   \partial_u \psi_{0} (u_\vareps) \right| 
   \,\Bigg] \notag\\&
   ~~~+ \frac{ \delta_{\text{coer}, 1}}{2\vareps}  \big|\tfrac{1}{2\sqrt{3}}\partial_u \psi_0(u_\vareps) \big|^2 \notag\\
\label{EstimateCoerIntermediate}
    & ~~~+  \frac{ \delta_{\text{coer}, 2}}{2\vareps}  \left| \partial_u \psi_{i,j}  (u_\vareps) \cdot   \partial_u \psi_{0}(u_\vareps) \right| 
   \,,
   \end{align}
   where \(\delta_{\text{coer}, 1},\delta_{\text{coer}, 2}, \delta_{\text{coer}, 3} >0\) are arbitrarily small constants.
   Finally, using (A4) for a suitable $\delta$ given by the gradient-flow calibration and integrating over the set $\{x:u_\eps(x,t)\in \mathcal{T}_{i,j} \}$, we can conclude about the validity of \eqref{eq:coerpartialpsik} and \eqref{eq:coerpartialpsiijk}.
   
   	\emph{Proof of \eqref{eq:coergradpsik}:} By adding zero, we can write
   \begin{align*}
   &	 \nabla (\psi_0 \circ u_\vareps) = \partial_u \psi_0 (u_\vareps) \cdot \nabla u_\vareps \\
   &	=
   \frac{1}{\sqrt{\vareps}} \partial_u \psi_0(u_\vareps) \cdot \bigg[
   \sqrt{\vareps}\nabla u_\vareps - \frac{1}{\sqrt{\vareps}} \bigg(\tfrac{1}{2}\partial_u \psi_{i,j} \otimes \xi_{i,j} - \sumk  \tfrac1{2\sqrt{3}} \partial_u \psi_0 (u_\vareps) \otimes \sqrt{3}\xi_k  \bigg)	
   \bigg]\\
   &\quad 	+\frac{1}{2\vareps}  \partial_u \psi_0(u_\vareps) \cdot 
   \partial_u \psi_{i,j} (u_\vareps)\xi_{i,j}
   - \frac{1}{\vareps}  \sumk   \tfrac1{2\sqrt{3}}  |\partial_u \psi_0 (u_\vareps) |^2 \sqrt{3}\xi_k \,.
   \end{align*}
    Then, Young's inequality yields
   \begin{align*}
   &| \nabla (\psi_0 \circ u_\vareps) |\\
   &	\leq 
   \frac{1}{2\vareps^2} |\partial_u \psi_0(u_\vareps) |^2 
   \\&
   + \frac12 \bigg|
   \sqrt{\vareps}\nabla u_\vareps - \frac{1}{\sqrt{\vareps}} \bigg(\tfrac{1}{2}\partial_u \psi_{i,j} (u_\vareps)\otimes \xi_{i,j} - \sumk  \tfrac1{2\sqrt{3}} \partial_u \psi_0 (u_\vareps) \otimes \sqrt{3}\xi_k  \bigg)	
   \bigg|^2\\
   &\quad 	+\frac{1}{2\vareps}   |\partial_u \psi_0(u_\vareps) \cdot 
   \partial_u \psi_{i,j} (u_\vareps)||\xi_{i,j}| \\
   & \quad + \frac{1}{\vareps}  \sumk     \tfrac1{2\sqrt{3}} |\partial_u \psi_0 (u_\vareps) |^2| \sqrt{3}\xi_k |
    \,,
   \end{align*}
   Using Young's inequality and the estimate \eqref{EstimateCoerIntermediate}, we can conclude about the validity of \eqref{eq:coergradpsik}.	
\end{proof}

\begin{proof}[Proof of Proposition \ref{prop:maincoer}]	

\emph{Proof of \eqref{eq:coeren}, \eqref{eq:coernormalpsi} and \eqref{eq:coerdistpsi}:}
	Using \eqref{eq:sum} and adding zero, we obtain
	\begin{align*} 
	E[u_\varepsilon| \xi] 
	&= E[u_\vareps] -  \sumij \int_{\R^d}\frac12 | \nabla (\psi_{i,j}  \circ u_\vareps) | \chit(u_\vareps) \dx \\
	&\quad
	+ \sumkij \int_{\R^d} \frac12 \xi_{k} \cdot 
	\nabla (\psi_{0}\circ u_\varepsilon)  {\chit}(u_\vareps)  \dx \notag \\
	& \quad + \sumij \int_{\R^d} \frac12 \Big( 1 - \xi_{i,j} \cdot\frac{\nabla (\psi_{i,j}  \circ u_\vareps)}{|\nabla (\psi_{i,j}  \circ u_\vareps)|}  \Big) |\nabla (\psi_{i,j}  \circ u_\vareps) |{\chit}(u_\vareps) \dx  \,.
		\end{align*}
	From assumption (A4), we can deduce
	\begin{align*}
	\frac12	| \nabla (\psi_{i,j}  \circ u_\vareps) | {\chit}(u_\vareps) \leq \frac12	| \partial_u \psi_{i,j}(u_\vareps)| | \nabla u_\vareps |  {\chit}(u_\vareps)\leq \sqrt{2W(u_\vareps)} |\nabla u_\vareps|  {\chit}(u_\vareps)\,.
	\end{align*}
	Hence, using the definition of $E[u_\eps]$, we have
	\begin{align*} 
	& E[u_\varepsilon| \xi] + \sumkij  \int_{\R^d} \frac12 \| \xi_{k}\|_{L_x^\infty} | \nabla (\psi_{0}\circ u_\varepsilon)| {\chit}(u_\vareps) \dx\\
	& \geq \sumij \int_{\R^d} \frac12 \bigg(\sqrt{\vareps}|\nabla u_\vareps| - \frac{\sqrt{2W(u_\vareps)}}{\sqrt{\vareps}} \bigg)^2  {\chit}(u_\vareps) \dx  \\
	& \quad +  \sumij \int_{\R^d} \frac12 \Big( 1 - \xi_{i,j} \cdot\frac{\nabla (\psi_{i,j}  \circ u_\vareps)}{|\nabla (\psi_{i,j}  \circ u_\vareps)|}  \Big) |\nabla (\psi_{i,j}  \circ u_\vareps)| {\chit}(u_\vareps) \dx \,.
	\end{align*}
	Then, noting that
	\begin{align*} 
	&\left|\frac{\nabla (\psi_{i,j}  \circ u_\vareps)}{|\nabla (\psi_{i,j}  \circ u_\vareps)|} - \xi_{i,j} \right|^2 |\nabla (\psi_{i,j} \circ u_\vareps)| {\chit}(u_\vareps) \\
	&\leq 2 \left( 1- \xi_{i,j} \cdot \frac{\nabla (\psi_{i,j}  \circ u_\vareps)}{|\nabla (\psi_{i,j}  \circ u_\vareps)|} \right) |\nabla (\psi_{i,j} \circ u_\vareps)|  {\chit}(u_\vareps)
	\end{align*}
	 together with the fact that $\sum^N_{i,j=1: i < j }{\chit}(u_\vareps) = 1$, we see that both \eqref{eq:coeren} and \eqref{eq:coernormalpsi} follow from the preceding two formulas and \eqref{eq:coergradpsik}.
Furthermore, using
\begin{align*}
	\min \{ \dist^2(x, I_{i,j}) ,1\} \leq C (1 - |\xi_{i,j}|) \leq C\bigg (1- \xi_{i,j} \cdot \frac{\nabla (\psi_{i,j}  \circ u_\vareps)}{|\nabla (\psi_{i,j}  \circ u_\vareps)|} \bigg),
\end{align*}
we obtain \eqref{eq:coerdistpsi}.	
	
\emph{Proof of \eqref{eq:coerdist}:}
By exploiting \eqref{eq:sum} and by adding zeros, we obtain
\begin{align} \label{eq:ineqforgradu}
&\left( \frac{\varepsilon}{2} |\nabla u_\vareps |^2  +  \frac{1}{\vareps} W(u_\vareps) \right){\chi}_{\mathcal{T}_{i,j}}(u_\vareps) \notag\\
&\leq \bigg[   \frac{\varepsilon}{2} |\nabla u_\vareps |^2 + \frac{1}{\vareps} W(u_\vareps)+ \sum_{\ell=1}^N \xi_\ell \cdot \nabla (\psi_\ell \circ u_\vareps ) 
\bigg]{\chit}(u_\vareps) \notag \\
& \quad + \bigg[   
 \frac{1}{2}|\nabla (\psi_{i,j} \circ u_\vareps)| 
+ \frac{1}{2}\left| \nabla (\psi_{0}\circ u_\varepsilon)\right|
  \bigg]{\chit}(u_\vareps)\,.
\end{align}
As a consequence, we have
\begin{align*}
&\sumij \int_{\R^d} \min \{ \dist^2(x, I_{i,j}) ,1\}  \left(\frac{\varepsilon}{2} |\nabla u_\vareps |^2  +  \frac{1}{\vareps} W(u_\vareps) \right)   {\chit}(u_\vareps) \dx \\
&\leq E[u_\vareps|\xi]
+\sumij \int_{\R^d} \min \{ \dist^2(x, I_{i,j}) ,1\} |\nabla (\psi_{i,j} \circ u_\vareps)| {\chit}(u_\vareps) \dx
\\
&\quad +   \sumij	 \int_{\R^d}  \left| \nabla (\psi_{0}\circ u_\varepsilon)\right| {\chit}(u_\vareps) \dx \,,
\end{align*}
whence we deduce \eqref{eq:coerdist} from \eqref{eq:coergradpsik} and \eqref{eq:coerdistpsi}.

\emph{Proof of  \eqref{eq:coertangentgradu}:}
Expanding the square and using $1-c\dist^2(\cdot,I_{i,j})\leq |\xi_{i,j}|\leq \max\{1-C\dist^2(\cdot,I_{i,j})),0\}$, we obtain
\begin{align*}
	 \varepsilon | (\Id-\xi_{i,j}\otimes \xi_{i,j}) \nabla u_\vareps^\mathsf{T} |^2 {\chi}_{\mathcal{T}^i_j}(u_\vareps)  
	 \leq &\, \left[  \varepsilon| \nabla u_\vareps |^2-  \vareps | (\xi_{i,j}  \cdot \nabla) u_\vareps |^2  \right] {\chi}_{\mathcal{T}_{i,j}}(u_\vareps) \\
	 &+ C\min\{\dist^2(\cdot, I_{i,j}), 1\}  \varepsilon| \nabla u_\vareps |^2 {\chi}_{\mathcal{T}_{i,j}}(u_\vareps) 	 
	 \,.
\end{align*}
Then, by adding zeros, we obtain
\begin{align*}
&\varepsilon | (\Id-\xi_{i,j}\otimes \xi_{i,j}) \nabla u_\vareps^\mathsf{T} |^2 {\chi}_{\mathcal{T}_{i,j}}(u_\vareps)  \\
&= \bigg[  \varepsilon| \nabla u_\vareps |^2 + \frac{1}{\vareps} 2W(u_\vareps)  - \xi_{i,j} \cdot \nabla (\psi_{i,j} \circ u_\vareps) 
+ \sumk   \xi_{k} \cdot 
\nabla (\psi_{0}\circ u_\varepsilon)  \bigg] {\chit}(u_\vareps) \\
&\quad - \sumk  \xi_{k} \cdot \nabla (\psi_{0}\circ u_\varepsilon) {\chit}(u_\vareps)  
+\left[ \frac{1}{4\vareps}  |\partial_u \psi_{i,j} (u_\vareps)  |^2 - \frac{1}{\vareps} 2W(u_\vareps)   \right] {\chit}(u_\vareps)\\
&\quad - \left[   \frac{1}{4\vareps}  |\partial_u \psi_{i,j} (u_\vareps)  |^2- \xi_{i,j} \cdot \nabla (\psi_{i,j} \circ u_\vareps)+  \vareps | (\xi_{i,j}  \cdot \nabla) u_\vareps |^2  \right] {\chit}(u_\vareps)\\
& \quad + C\min\{\dist^2(\cdot, I_{i,j}), 1\}  \varepsilon| \nabla u_\vareps |^2 {\chit}(u_\vareps) \\
&\leq 2 \left[  \frac{\vareps}{2}| \nabla u_\vareps |^2 + \frac{1}{\vareps} W(u_\vareps)  + \sum_{\ell=1}^N \xi_\ell \cdot \nabla (\psi_{\ell}  \circ u_\vareps)  \right] {\chit}(u_\vareps) \\
&\quad + \sumk \|\xi_{k} \|_{L_x^\infty} |  \nabla (\psi_{0} \circ u_\vareps) | {\chit}(u_\vareps)
\\
&\quad
+\left[ \frac{1}{4\vareps}  |\partial_u \psi_{i,j} (u_\vareps)  |^2 - \frac{1}{\vareps} 2W(u_\vareps)   \right] {\chit}(u_\vareps)\\
&\quad - \left| \frac{1}{2\sqrt{\vareps}}  \partial_u \psi_{i,j} (u_\vareps) -  \sqrt{\vareps} (\xi_{i,j}  \cdot \nabla) u_\vareps  \right|^2 {\chit}(u_\vareps)\\
& \quad + C \min\{\dist^2(\cdot, I_{i,j}), 1\}  \varepsilon| \nabla u_\vareps |^2 {\chit}(u_\vareps) \,,
\end{align*}
due to \eqref{eq:sum}.
Noting that assumption (A4) implies
\begin{align*}
\frac{1}{4} \big|\partial_u \psi_{i,j}  (u_\vareps)\big|^2  {\chit} (u_\vareps) \leq 2 W(u_\vareps){\chit}(u_\vareps)\,,
\end{align*}
the validity of \eqref{eq:coertangentgradu} follows from \eqref{eq:coergradpsik} and \eqref{eq:coerdist}.
\end{proof}

We next prove the additional coercivity properties stated in Lemma~\ref{lemma:additionalcoer}. 

\begin{proof}[Proof of Lemma \ref{lemma:additionalcoer}]

\emph{Proof of \eqref{eq:coernormalu}:}
Note that \eqref{eq:ineqforgradu} yields
\begin{align*}
& \sumij  \int_{\R^d}	\left|\frac{\nabla (\psi_{i,j}  \circ u_\vareps)}{|\nabla (\psi_{i,j}  \circ u_\vareps)|} - \xi_{i,j} \right|^2  \varepsilon |\nabla u_\vareps |^2 {\chit}(u_\vareps) \dx \\
&\leq  8 E[u_\vareps|\xi]
+  \sumij 	 \int_{\R^d} \left|\frac{\nabla (\psi_{i,j}  \circ u_\vareps)}{|\nabla (\psi_{i,j}  \circ u_\vareps)|} - \xi_{i,j} \right|^2 |\nabla (\psi_{i,j} \circ u_\vareps)| {\chit}(u_\vareps) \dx \\
&\quad + \sumij	4 \int_{\R^d} | \nabla (\psi_{0} \circ u_\vareps)| {\chit}(u_\vareps) \dx \,.
\end{align*}
Hence, using \eqref{eq:coergradpsik} and \eqref{eq:coernormalpsi}, we obtain \eqref{eq:coernormalu}.

\emph{Proof of \eqref{eq:coersurfmeaslenght} and  \eqref{eq:coersurfmeas}:}
First, we compute
	\begin{align*}
&\left|  \frac{1}{2\sqrt{\vareps}}   \frac{|\nabla (\psi_{i,j}  \circ u_\vareps)|}{|\nabla u_\vareps|} - \sqrt{\vareps} |\nabla u_\eps | \right|^2 {\chit}(u_\vareps) \\
&= \left[ \frac{1}{\vareps} \frac{|\nabla (\psi_{i,j}  \circ u_\vareps)|^2}{4|\nabla u_\vareps|^2}  + \vareps |\nabla u_\eps |^2   - |\nabla (\psi_{i,j}  \circ u_\vareps)| \right] {\chit}(u_\vareps) 
\end{align*}
and
	\begin{align*} 
	&\left|  \frac{1}{2\sqrt{\vareps}}  \partial_u \psi_{i,j} (u_\vareps)  \otimes  \frac{\nabla (\psi_{i,j}  \circ u_\vareps)}{|\nabla (\psi_{i,j}  \circ u_\vareps)|} - \sqrt{\vareps} \nabla u_\eps \right|^2   {\chit}(u_\vareps) \dx 	\\
	&=\left[   \frac{1}{4\vareps}  |\partial_u \psi_{i,j} (u_\vareps)  |^2 + \vareps |\nabla u_\eps |^2   - |\nabla (\psi_{i,j}  \circ u_\vareps)|\right] {\chit}(u_\vareps) \,.
	\end{align*}		
	Then, from assumption (A4) we deduce
	\begin{align*}
	\frac{|\nabla (\psi_{i,j}  \circ u_\vareps)|^2}{|\nabla u_\vareps|^2} {\chit}(u_\vareps) \leq |\partial_u \psi_{i,j}  (u_\vareps)|^2  {\chit}(u_\vareps)\leq 8 W(u_\vareps) {\chit}(u_\vareps)\,.
	\end{align*}
	Finally, by exploiting \eqref{eq:sum} and by adding zero, one can conclude about the validity of both \eqref{eq:coersurfmeas} and \eqref{eq:coersurfmeaslenght}, due to \eqref{eq:coerpartialpsik}.

\emph{Proof of \eqref{eq:coerxigradu}:}
Since we have
\begin{align*}
&	\left|  \xi_{i,j} \otimes \xi_{i,j} 
- \frac{ \nabla u_\vareps^\mathsf{T} \nabla u_\vareps }{|\nabla u_\vareps|^2}
\right|^2 \vareps |\nabla u_\vareps|^2 {\chit}(u_\vareps) \\
&	\leq \left[2- 2\frac{\xi_{i,j} \cdot  \nabla u_\vareps^\mathsf{T}  \nabla u_\vareps \cdot \xi_{i,j}}{ |\nabla u_\vareps|^2}  \right] \vareps |\nabla u_\vareps|^2 {\chit}(u_\vareps) \\
&	\leq 2 \left[  \varepsilon| \nabla u_\vareps |^2-  \vareps | (\xi_{i,j}  \cdot \nabla) u_\vareps |^2  \right] {\chit}(u_\vareps) \\
&	\leq 2   \vareps | (\Id-\xi_{i,j}\otimes \xi_{i,j}) \cdot \nabla u_\vareps^\mathsf{T} |^2  {\chi}_{\mathcal{T}_{i,j}}(u_\vareps) + C \min\{\dist^2(\cdot, I_{i,j}), 1\} \varepsilon| \nabla u_\vareps |^2 {\chit}(u_\vareps)
\end{align*}
(where in the last step we have used the estimate $1-C\dist^2(\cdot,I_{i,j}) \leq |\xi_{i,j}|\leq \max\{1-c\dist^2(\cdot,I_{i,j})),0\}$), the bound \eqref{eq:coerxigradu} follows from \eqref{eq:coertangentgradu} and \eqref{eq:coerdist}.

\end{proof}

\subsection{Convergence of the phase indicator functions}

We now show how to obtain the error estimate at the level of the indicator functions.
\begin{proof}[Proof of Proposition~\ref{PropBulkConvergence}]
Using \eqref{eq:AC} and the fact that $\supp \partial_t \bar \chi_i \subset \partial \supp \bar \chi_i$ as well as $\vartheta_i=0$ on $\partial \supp \bar \chi_i$, we compute
\begin{align}
\ddt& \int_{\R^d} (\psi_i(u_\vareps ) - \bar{\chi}_i) \vartheta_i \dx \notag \\
= &  \int_{\R^d} \frac1\vareps \partial_u \psi_i(u_\vareps ) \cdot \left( \vareps \Delta u_\vareps - \frac1\vareps \partial_u W(u_\vareps)\right) \vartheta_i \dx 
+ \int_{\R^d}(\psi_i(u_\vareps ) - \bar{\chi}_i) \partial_t \vartheta_i \dx \notag\\
= &  \int_{\R^d} \frac1\vareps \partial_u \psi_i(u_\vareps ) \cdot \left( \vareps \Delta u_\vareps - \frac1\vareps \partial_u W(u_\vareps)\right) \vartheta_i \dx 
+ \int_{\R^d} B \cdot \nabla (\psi_{i} \circ u_\vareps) \vartheta_i \dx \notag\\
& + \int_{\R^d}(\nabla \cdot B)(\psi_i(u_\vareps ) - \bar{\chi}_i) \vartheta_i  \dx 
+ \int_{\R^d}(\psi_i(u_\vareps ) - \bar{\chi}_i)  \left(\partial_t \vartheta_i + B \cdot \nabla \vartheta_i \right) \dx\,, \label{eq:netest1}
\end{align}
where we added a zero and then integrated by parts. Note that we used the fact that $\vartheta_i =0 $ on $\partial \{\bar{\chi}_i =1\}$.

By \eqref{eq:timeevtheta}, the last two terms on the right hand side of \eqref{eq:netest1} can be bounded by
\begin{align*}
\left(\| \nabla \cdot B\|_{L_x^\infty} + C\right)\int_{\R^d}|\psi_i(u_\vareps ) - \bar{\chi}_i| |\vartheta_i | \dx .
\end{align*}
As for the second term on the right hand side of \eqref{eq:netest1}, we perform the folllowing decomposition
\begin{align*}
\int_{\R^d} B \cdot \nabla (\psi_{i} \circ u_\vareps) \vartheta_i \dx 
=& \;  \sum^N_{\substack{j,k=1\\ j < k } } \int_{\R^d} (B \cdot \xi_{j,k}) \xi_{j,k} \cdot \nabla (\psi_{i} \circ u_\vareps) \vartheta_i \chi_{\mathcal{T}_{j,k}^{{\color{white} b }}}(u_\vareps) \dx \\
&+ \sum^N_{\substack{j,k=1\\ j < k } } \int_{\R^d} [(\Id- \xi_{j,k}  \otimes \xi_{j,k} )B)]
\cdot \nabla (\psi_{i} \circ u_\vareps) \vartheta_i \chi_{\mathcal{T}_{j,k}^{{\color{white} b }}} (u_\vareps)  \dx \,,
\end{align*}
whence, by adding a zero and using $\sum^N_{ j,k=1 : j  \neq k } \chi_{\mathcal{T}_{j,k}^{{\color{white} b }}}  = 1 $, 
\begin{align*}
\int_{\R^d} &B \cdot \nabla (\psi_{i} \circ u_\vareps) \vartheta_i \dx \\
=& \;   \int_{\R^d} \frac1\vareps \h_\vareps \cdot \frac{\nabla u_\vareps ^\mathsf{T}}{|\nabla u_\vareps|}\cdot \partial_u \psi_{i} ( u_\vareps)  \dx  \\
& + \sum^N_{\substack{j,k=1\\ j < k } }  \int_{\R^d} \frac1\vareps  \left[\vareps (B \cdot \xi_{j,k}) \xi_{j,k} |\nabla u_\vareps| - \h_\vareps \right]\cdot \frac{\nabla u_\vareps ^\mathsf{T}}{|\nabla u_\vareps|}\cdot \partial_u \psi_{i} ( u_\vareps) \vartheta_i \chi_{\mathcal{T}_{j,k}^{{\color{white} b }}} (u_\vareps) \dx \\
&+ \sum^N_{\substack{j,k=1\\ j < k } }  \int_{\R^d} [(\Id- \xi_{j,k}  \otimes \xi_{j,k} )B)]
\cdot \nabla (\psi_{i} \circ u_\vareps) \vartheta_i \chi_{\mathcal{T}_{j,k}^{{\color{white} b }}} (u_\vareps) \dx \,.
\end{align*}
Note that the last two terms are nonzero only if $j=i$ or $k=i$.
Hence, using Young's inequality, the second term can be estimated by 
\begin{align*}
 &  \frac12 \sum^N_{\substack{ k=1 \\ k \neq i }}  \int_{\R^d} \frac1\vareps  \left|\vareps (B \cdot \xi_{i,k}) \xi_{i,k} |\nabla u_\vareps| - \h_\vareps \right|^2 \chi_{\mathcal{T}_{i,k}^{{\color{white} b }}} (u_\vareps) \dx\\
&+ \frac12 \sum^N_{\substack{ k=1 \\ k \neq i }}  \int_{\R^d} \frac1\vareps  | \partial_u \psi_{i} ( u_\vareps)|^2| \vartheta_i |^2 \chi_{\mathcal{T}_{i,k}^{{\color{white} b }}}(u_\vareps) \dx\,.
\end{align*}
As for the third one, we obtain using Young's inequality and exploiting the coercivity property \eqref{eq:coertangentgradu}
\begin{align*}
&  \sum^N_{\substack{ k=1 \\ k \neq i }} \int_{\R^d} [(\Id- \xi_{i,k}  \otimes \xi_{i,k} )B)]
\cdot \nabla (\psi_{i} \circ u_\vareps) \vartheta_i  \chi_{\mathcal{T}_{i,k}^{{\color{white} b }}} (u_\vareps) \dx
\\
&\leq
\sum^N_{\substack{ k=1 \\ k \neq i }} \int_{\R^d} \eps |(\Id- \xi_{i,k}  \otimes \xi_{i,k} )B)] \nabla u_\eps^\mathsf{T} |^2 \chi_{\mathcal{T}_{i,k}^{{\color{white} b }}} (u_\vareps) \dx
\\&~~~~
+\sum^N_{\substack{ k=1 \\ k \neq i }} \int_{\R^d} \frac{1}{\eps} |\partial_u \psi_i|^2 |\vartheta_i|^2  \chi_{\mathcal{T}_{i,k}^{{\color{white} b }}} (u_\vareps) \dx
\\&
\leq C E[u_\eps|\xi] + \sum^N_{\substack{ k=1 \\ k \neq i }} \int_{\R^d} \frac{1}{\eps} |\partial_u \psi_i|^2 |\vartheta_i|^2  \chi_{\mathcal{T}_{i,k}^{{\color{white} b }}} (u_\vareps) \dx.
\end{align*}
In summary, we have shown
\begin{align*}
&\ddt \int_{\R^d} (\psi_i(u_\vareps ) - \bar{\chi}_i) \vartheta_i \dx
\\&
\leq C \int_{\R^d}|\psi_i(u_\vareps ) - \bar{\chi}_i| |\vartheta_i | \dx + \sum^N_{\substack{ k=1 \\ k \neq i }} \int_{\R^d} \frac{1}{\eps} |\partial_u \psi_i(u_\vareps )|^2 |\vartheta_i|^2  \chi_{\mathcal{T}_{i,k}^{{\color{white} b }}} (u_\vareps) \dx 
\\&~~~
+\frac12 \sum^N_{\substack{ k=1 \\ k \neq i }}  \int_{\R^d} \frac1\vareps  \left|\vareps (B \cdot \xi_{i,k}) \xi_{i,k} |\nabla u_\vareps| - \h_\vareps \right|^2 \chi_{\mathcal{T}_{i,k}^{{\color{white} b }}} (u_\vareps) \dx
\\&~~~
+\int_{\R^d} \frac1\vareps \partial_u \psi_i(u_\vareps ) \cdot \left( \vareps \Delta u_\vareps - \frac1\vareps \partial_u W(u_\vareps)\right) \vartheta_i \dx
+\int_{\R^d} \frac1\vareps \h_\vareps \cdot \frac{\nabla u_\vareps ^\mathsf{T}}{|\nabla u_\vareps|}\cdot \partial_u \psi_{i} ( u_\vareps)  \dx.
\end{align*}
We estimate the two terms in the last line
\begin{align*}
&\int_{\R^d} \frac1\vareps \partial_u \psi_i(u_\vareps ) \cdot \left( \vareps \Delta u_\vareps - \frac1\vareps \partial_u W(u_\vareps)\right) \vartheta_i \dx + \int_{\R^d} \frac1\vareps \h_\vareps \cdot \frac{\nabla u_\vareps ^\mathsf{T}}{|\nabla u_\vareps|}\cdot \partial_u \psi_{i} ( u_\vareps) \vartheta_i \dx  \\
&= \int_{\R^d} \frac1\vareps \bigg[\left( \vareps \Delta u_\vareps - \frac1\vareps \partial_u W(u_\vareps)\right) +  \h_\vareps \cdot \frac{\nabla u_\vareps ^\mathsf{T}}{|\nabla u_\vareps|} \bigg] \cdot \partial_u \psi_{i} ( u_\vareps) \vartheta_i \dx\\
&\leq \frac12 \int_{\R^d} \frac1\vareps \bigg[\left( \vareps \Delta u_\vareps - \frac1\vareps \partial_u W(u_\vareps)\right) +  \h_\vareps \cdot \frac{\nabla u_\vareps ^\mathsf{T}}{|\nabla u_\vareps|} \bigg]^2\dx
+ \frac12 \int_{\R^d} \frac1\vareps | \partial_u \psi_{i} ( u_\vareps)|^2| \vartheta_i|^2 \dx		\\
& \leq  \int_{\R^d} \frac1{2\vareps} \left(\bigg|\vareps \Delta u_\vareps - \frac1\vareps \partial_u W(u_\vareps) \bigg|^2 - |\h_\vareps|^2\right) \dx 
-  \int_{\R^d} \frac1{2\vareps} \biggl[ \Id - \frac{\nabla u_\vareps ^\mathsf{T} \nabla u_\vareps }{|\nabla u_\vareps|^2}   \biggr]: \h_\vareps \otimes \h_\vareps \dx \\
& \quad + \frac12 \int_{\R^d} \frac1\vareps | \partial_u \psi_{i} ( u_\vareps)|^2| \vartheta_i|^2 \dx\\
& \leq  \int_{\R^d} \frac1{2\vareps} \left(\bigg|\vareps \Delta u_\vareps - \frac1\vareps \partial_u W(u_\vareps) \bigg|^2 - |\h_\vareps|^2\right) \dx + \frac12 \int_{\R^d} \frac1\vareps | \partial_u \psi_{i} ( u_\vareps)|^2| \vartheta_i|^2 \dx\,,
\end{align*}
where we used the fact that $\big[ \Id - \frac{\nabla u_\vareps ^\mathsf{T} \nabla u_\vareps }{|\nabla u_\vareps|^2}   \big]$ is a positive semidefinite matrix.
Since $|\partial_u \psi_{i}(u_\vareps) |\leq C \sqrt{2W(u_\vareps)}$ and $| \vartheta_i | \leq \min\{\dist^2(\cdot, \partial \supp \bar \chi_i), 1\} \leq C\min\{\dist(\cdot, I_{i,k}),1\}$ (see \eqref{DistanceWeightConditions}), from \eqref{eq:coerdist} it follows that
\begin{align*}
& \int_{\R^d} \frac1\vareps | \partial_u \psi_{i} ( u_\vareps)|^2| \vartheta_i|^2 \dx =  \sum^N_{\substack{ k=1 \\ k \neq i }}  \int_{\R^d}  \frac1\vareps |\partial_u \psi_{i}(u_\vareps) |^2| \vartheta_i |^2 \chi_{\mathcal{T}_{i,k}^{{\color{white} b }}}   (u_\vareps) \dx \\
&\leq 	C \sum^N_{\substack{ k=1 \\ k \neq i }}     \int_{\R^d}  \frac1\vareps 2W(u_\vareps) | \vartheta_i |^2 \chi_{\mathcal{T}_{i,k}^{{\color{white} b }}} (u_\vareps) \dx	\\
&\leq C E[u_\vareps| \xi ] \,.
\end{align*}
Summarizing the previous estimates, we get
\begin{align*}
\ddt& \int_{\R^d} (\psi_i(u_\vareps ) - \bar{\chi}_i) \vartheta_i \dx
\\&
\leq
C E[u_\vareps| \xi ]
+C \int_{\R^d}|\psi_i(u_\vareps ) - \bar{\chi}_i| |\vartheta_i | \dx
\\&~~~
+\frac12 \sum^N_{\substack{ k=1 \\ k \neq i }}  \int_{\R^d} \frac1\vareps  \left|\vareps (B \cdot \xi_{i,k}) \xi_{i,k} |\nabla u_\vareps| - \h_\vareps \right|^2 \chi_{\mathcal{T}_{i,k}^{{\color{white} b }}} (u_\vareps) \dx
\\&~~~~
+\int_{\R^d} \frac1{2\vareps} \left(\bigg|\vareps \Delta u_\vareps - \frac1\vareps \partial_u W(u_\vareps) \bigg|^2 - |\h_\vareps|^2\right) \dx.
\end{align*}
An application of the Gronwall inequality to Theorem \ref{TheoremRelativeEnergyInequality} yields
\begin{align*}
&\sup_{t \in [0,T]}   E[u_\vareps|\xi](t) + 
 \sum_{i=1}^N \sum^N_{\substack{ k=1 \\ k \neq i }} \int_{\R^d} \frac1\vareps  \left|\vareps (B \cdot \xi_{i,k}) \xi_{i,k} |\nabla u_\vareps| - \h_\vareps \right|^2 \chi_{\mathcal{T}_{i,k}^{{\color{white} b }}}  (u_\vareps) \dx \\
& + \int_{\R^d} \frac1{2\vareps} \bigg(\bigg|\vareps \Delta u_\vareps - \frac1\vareps \partial_u W(u_\vareps) \bigg|^2 - |\h_\vareps|^2\bigg) \dx
\leq C(d,T, (\bar{\chi}(t))_{t \in [0,T]}) E[u_\eps|\xi](0).
\end{align*}
Integrating the previous formula in time and inserting this estimate, we deduce by the condition \eqref{DistanceWeightConditions} on the weight $\vartheta_i$
\begin{align*}
&\int_{\R^d} (\psi_i(u_\vareps(\cdot,T)) - \bar{\chi}_i(\cdot,T)) \vartheta_i(\cdot,T) \dx
\\&
\leq C(d,T, (\bar{\chi}(t))_{t \in [0,T]}) E[u_\eps|\xi](0) + \int_0^T \int_{\R^d} (\psi_i(u_\vareps) - \bar{\chi}_i) \vartheta_i \dx \,dt.
\end{align*}
The Gronwall inequality now implies our result, using again \eqref{DistanceWeightConditions}.
\end{proof}

\subsection{Proof of the main theorem}

Our main theorem (Theorem~\ref{MainTheorem}) is a simple consequence of Theorem~\ref{TheoremRelativeEnergyInequality} and Proposition~\ref{PropBulkConvergence}.
\begin{proof}
Combining Theorem~\ref{TheoremRelativeEnergyInequality} and Proposition~\ref{PropBulkConvergence}, we obtain the desired bounds
\begin{align*}
\sup_{t\in [0,T]} E_\eps[u_\eps|\xi](t) &\leq C \eps,
\\
\sup_{t\in [0,T]} 
\max_{i\in \{1,\ldots,N\}} \int_{\Rd} |\psi_i(u_\eps(\cdot,t))-\bar \chi_i(\cdot,t)| \dist(x,\partial \supp \bar\chi_i(\cdot,t)) \,dx &\leq C \eps.
\end{align*}
\end{proof}

\section{Construction of well-prepared initial data}
In this section we construct an initial datum $u_\vareps(\cdot,0)$ complying with the following relative energy estimate:
\begin{align} \label{eq:initialestimate}
E[u_\vareps|\xi](0) \leq C\vareps \,,
\end{align}
where the constant $C>0$ depends on the initial data $(\bar \chi_1(\cdot,0),\ldots,\bar \chi_N(\cdot,0))$ and the potential $W$ satisfying assumptions (A1)--(A4) (see Sec. \ref{sec:mainresults}).
In particular, we provide an explicit construction of $u_\vareps(\cdot,0)$ for a network of interfaces meeting at two-dimensional triple junctions ($d=2$) satisfying the 120 degree angle condition. To this aim, we adopt a geometric setting for the initial network which was already introduced in \cite[Sec.~5-6]{FischerHenselLauxSimon} in the general time-dependent case. 
A similar construction can be provided for three-dimensional double
bubbles ($d=3$) satisfying the correct angle condition along the triple line, this time by exploiting the corresponding geometric setting given by \cite[Sec.~3-4]{HenselLauxMultiD}.

Note that from our construction and the fast decay of the Modica-Mortola profiles towards the pure phases $\alpha_i$, $1\leq i\leq N$, it will also be apparent that our initial data $u_\eps(\cdot,0)$ also satisfies the estimate
\begin{align*}
\max_{i\in \{1,\ldots,N\}} \int_{\Rd} |\psi_i(u_\eps(\cdot,0))-\bar \chi_i(\cdot,0)| \dist(x,\partial \supp \bar\chi_i(\cdot,0)) \,dx &\leq C \eps.
\end{align*}
In fact, for this lower-order quantity one may even show the stronger bound $O(\eps^2)$. In summary, the considerations in the present section will establish Proposition~\ref{PropositionInitialData}.

\subsection{Rescaled  one-dimensional equilibrium profiles}
For any distinct $i,j \in \{1,...,N\}$, let $\gamma_{i , j }: [-1,1] \rightarrow \R^2$ be the unique constant-speed $C^1$ path connecting $\alpha_i$ to $\alpha_j$ such that $\gamma_{i,j}(-1)= \alpha_i$,  $\gamma_{i,j}(1)= \alpha_j$, and
\begin{align*}
\int_{-1}^{1} \sqrt{2W(\gamma_{i , j }(r))} |\gamma_{i , j }^\prime(r)| \mathrm{d} r =1 \,.
\end{align*}
Let $\tilde{\theta}_{i,j} : \R \rightarrow [-1,+1]$ be the unique solution of the ODE 
\begin{align*}
\tilde{\theta}_{i,j}^\prime(s)= |\gamma_{i , j }^\prime(\tilde{\theta}_{i,j}(s))|^{-1}\sqrt{2W(\gamma_{i , j }(\tilde{\theta}_{i,j}(s))}
\end{align*}
with boundary conditions $\tilde{\theta}_{i,j}(\pm \infty) = \pm 1$.
Due to the growth properties of $W$ in the neighborhoods of $\alpha_{i}$ and $\alpha_{j}$ (see condition (A1) in Sec. \ref{sec:mainresults}), the profile $\tilde{\theta}_{i,j}$ approaches its boundary values $\pm 1$ at $\pm \infty$ with a power law of order $\frac{2}{q-2}$ for $q > 2$ and an exponential rate for $q=2$ \cite{Sternberg}.

Let $s^0_{i,j} \in \R$ such that $\tilde{\theta}_{i,j}(s_{i,j}^0)=0$.
Let $\rho> 0$ be such that $\tilde{\theta}_{i,j}(\rho+s^0_{i,j}) =\bar{\theta}_{i,j}^+ $ and $\tilde{\theta}_{i,j}(-\rho+s^0_{i,j})= - \bar{\theta}_{i,j}^-$ for $\bar{\theta}_{i,j}^\pm \in (0,1)$.
We define the rescaled one-dimensional equilibrium profiles $\theta_{i,j} : \R \rightarrow \gamma_{i,j}$ as
\begin{align} \label{eq:equilibriumprofile}
\theta_{i,j}(s) :=
 \begin{cases}
\gamma_{i , j } \Big( \frac1{\bar{\theta}_{i,j}^+}{ \tilde{\theta}_{i,j}( s^0_{i,j}  \vee  (s+s^0_{i,j} ) \wedge (\rho + s^0_{i,j}) ) } \Big) \quad \text{for } s \in [0,\infty)\,,\\
\gamma_{i , j } \Big( \frac1{\bar{\theta}_{i,j}^-} { \tilde{\theta}_{i,j}( (- \rho + s^0_{i,j} )\vee  (s+s^0_{i,j} ) \wedge s^0_{i,j}  ) } \Big)
\quad \text{for } s \in (-\infty,0) \,,
\end{cases}
\end{align}
so that $	\theta_{i,j}(s) = \alpha_i$ for any $s \leq \rho$ and $	\theta_{i,j} (s) = \alpha_j$  for any $s \geq \rho$. Furthermore, we have
\begin{align} 
&(\theta_{i,j})^\prime(s) \\
&= \begin{cases}
\frac1{\bar{\theta}_{i,j}^+}\sqrt{2W(\theta_{i,j}(s))} \frac{\gamma_{i , j }^\prime}{|\gamma_{i , j }^\prime|} 
\Big( \frac1{\bar{\theta}_{i,j}^+} { \tilde{\theta}_{i,j}( s^0_{i,j}  \vee  (s+s^0_{i,j} ) \wedge (\rho + s^0_{i,j}) ) } \Big)   \;\; \text{for } s \in [0,  \rho),\\
\frac1{\bar{\theta}_{i,j}^-}\sqrt{2W(\theta_{i,j}(s))} \frac{\gamma_{i , j }^\prime}{|\gamma_{i , j }^\prime|} \Big( \frac1{\bar{\theta}_{i,j}^-}{ \tilde{\theta}_{i,j}( (- \rho  + s^0_{i,j}  )\vee  (s+s^0_{i,j} ) \wedge s^0_{i,j}  ) } \Big)\; \;   \text{for } s \in (-\rho,  0),\\
0 \;\; \text{for } s \in (- \infty, -\rho] \cup [ \rho, + \infty) .
\end{cases}
\end{align}
Note that, if $W$ satisfies additional symmetry properties along the path $\gamma_{i , j }$, then $\tilde{\theta}_{i,j}$ is odd, thus $s^0_{i,j}=0$ and $\bar{\theta}_{i,j}^-=\bar{\theta}_{i,j}^+$.
Moreover, if $W$ satisfies additional symmetry properties with respect to all the paths $\gamma_{i , j }$, then all $\bar{\theta}_{i,j}^\pm$ coincide.

\subsection{Geometry of the initial network} \label{subsec:geometry}
For simplicity of notation, we shall omit the evaluation at initial time throughout this chapter, i.\,e.\ we write $\bar{\chi}$ instead of $\bar{\chi}(\cdot,0)$.
Let $\bar{\chi} = (\bar{\chi}_1,..., \bar{\chi}_N)$ be an initial partition of $\R^2$ with interfaces $\partial\{\bar{\chi}_i = 1\} \cap \partial\{\bar{\chi}_j = 1\}=: I_{i,j}$ for distinct $ i,j \in \{1,..., N\}$. 
We decompose the network of interfaces according to its
topological features, i.e., into smooth two-phase interfaces and
triple junctions. 
Suppose that the network has $P$ of
such topological features $\mathcal{T}_n$, $n \in \{1,..., P\}$. We then split $\{1,..., P\} =: \mathcal{C} \cup \mathcal{P}$,
where $\mathcal{C}$ enumerates the connected components
of the two-phase interfaces and $\mathcal{P}$ enumerates the triple junctions. 
In particular, if  $p \in \mathcal{P}$, $\mathcal{T}_p$ is a triple junction, whereas if $c \in \mathcal{C}$,
$\mathcal{T}_c$ is a connected component of a two-phase
interface $I_{i,j}$ for some distinct $i,j \in \{1,..., N\}$.

In the following we use a suitable notion of neighborhood for a single connected component of the network of interfaces provided by \cite[Definition 21]{FischerHenselLauxSimon}. In particular, we adopt the notion of localization radius, which allows to define the diffeomorphism corresponding to a single connected component of a network as it follows.
Let $r_{i,j}$ be a localization radius for the interface $I_{i,j}$ and let $\vec n_{i,j}$ be the normal vector field to $I_{i,j}$ pointing towards $\{\bar{\chi}_j = 1\}$ for some distinct $ i,j \in \{1,..., N\}$.
Then, the map $\Psi_{i,j}: I_{i,j} \times (-r_{i,j}, r_{i,j}) \rightarrow \R^2$, $(x,s) \mapsto x + s \vec n_{i,j}(x)$ defines a diffeomorphism, whose inverse can be splitted as follows
$\Psi^{-1}_{i,j}: \text{im}(\Psi_{i,j}) \mapsto  I_{i,j}  \times (-r_{i,j}, r_{i,j}) $, 
$x \mapsto (P_{{I}_{i,j}} x, \dist^\pm(x,I_{i,j }))$, where
$P_{ {I}_{i,j}}: \operatorname{im}\left(\Psi_{i, j}\right)  \rightarrow {I}_{i,j}$ represent the projection onto the nearest point on the interface ${I}_{i,j}$, whereas $\dist^\pm(\cdot,I_{i,j }): \operatorname{im}\left(\Psi_{i, j}\right)  \rightarrow (-r_{i,j}, r_{i,j})$ a signed distance function.

Similarly as in \cite[Definition 24]{FischerHenselLauxSimon}, we provide a notion of admissible localization radius for a triple junction.
\begin{definition}
	Let $d = 2$. Let $\bar{\chi} = (\bar{\chi}_1,..., \bar{\chi}_N)$ be an initial partition of $\R^d$ with interfaces $\partial\{\bar{\chi}_i = 1\}$, $i = 1,...,N$. Let $\mathcal{T}$ be a triple junction present in the network of interfaces of $\bar{\chi}$, which we assume for simplicity to be formed by the phases 1, 2 and 3.
	For each $i \in \{1, 2, 3\}$, denote by $\mathcal{T}_{i, i+1}$ the connected component of $I_{i,i+1}$ with an endpoint at the triple junction $\mathcal{T}$  and let $r_{i, i+1} \in (0, 1]$ be an admissible localization radius for the interface ${I}_{i, i+1}$ in the sense of \cite[Definition 21]{FischerHenselLauxSimon}.
	We call a scale $r=r_\mathcal{T} \in (0, r_{1,2} \wedge r_{2,3} \wedge r_{3,1}]$ an admissible localization radius
	for the triple junction $\mathcal{T}$ if there exists a wedge decomposition of the
	neighborhood $B_r(\mathcal{T})$ of the triple junction in the following sense:
	
	For each $i \in \{1, 2, 3\}$ there exist sets $W_{i, i+1}$ and $W_i$ with the following properties:
	
	First, the sets $W_{i, i+1}$ and $W_i$ are non-empty subsets of $B_r(\mathcal{T})$ with pairwise disjoint interior such that
	\begin{align*}
	\bigcup_{i \in\{1,2,3\}} \overline{W_{i, i+1}} \cup \overline{W_{i}}=\overline{B_{r}(\mathcal{T})} \,.
	\end{align*}
	
	Second, each of these sets is represented by a cone with apex at the triple junction $\mathcal{T}$
	intersected with $B_r(\mathcal{T})$. More precisely, there exist six pairwise distinct unit-length vectors $\left(X_{i, i+1}^{i}, X_{i, i+1}^{i+1}\right)_{i \in\{1,2,3\}}$ such that for all $i \in\{1,2,3\}$ we have 
	\begin{align*}
W_{i, i+1}&=\left(\mathcal{T}+\left\{a X_{i, i+1}^{i}+ b  X_{i, i+1}^{i+1}: a, b \in(0, \infty)\right\}\right) \cap B_{r}(\mathcal{T}) \\ W_{i} &=\left(\mathcal{T}+\left\{a  X_{i, i+1}^{i}+ b  X_{i-1, i}^{i}: a,b  \in(0, \infty)\right\}\right) \cap B_{r}(\mathcal{T}) \,.
	\end{align*}
	The opening angles of these cones are numerically fixed by
     \begin{align*}
     	X_{i, i+1}^{i} \cdot X_{i, i+1}^{i+1} =  \cos(\tfrac{\pi}{2})=0  \,, \quad X_{i, i+1}^{i}\cdot  X_{i-1, i}^{i} = \cos(\tfrac{\pi}{6})\,.
     \end{align*}
     
     Third, we require that for all $i \in\{1,2,3\}$ it holds
     \begin{align*}
     	B_{r}(\mathcal{T}) \cap \mathcal{T}_{i, i+1} & \subset W_{i, i+1} \cup \mathcal{T} \subset \mathbb{H}_{i, i+1} \\ 
     	W_{i} & \subset \mathbb{H}_{i, i+1} \cap \mathbb{H}_{i, i-1} 
     \end{align*}
     with the domains $\mathbb{H}_{i, i+1}:= \left\{x \in \mathbb{R}^{2}:x \in \operatorname{im}\left(\Psi_{i, i+1}\right)\right\} \cap B_{r}(\mathcal{T})$, where $\Psi_{i, i+1}$ is the diffeomorphism defining the neighboorhood of ${I}_{i, i+1}$ in the
     sense of \cite[Definition 21]{FischerHenselLauxSimon}.
\end{definition}

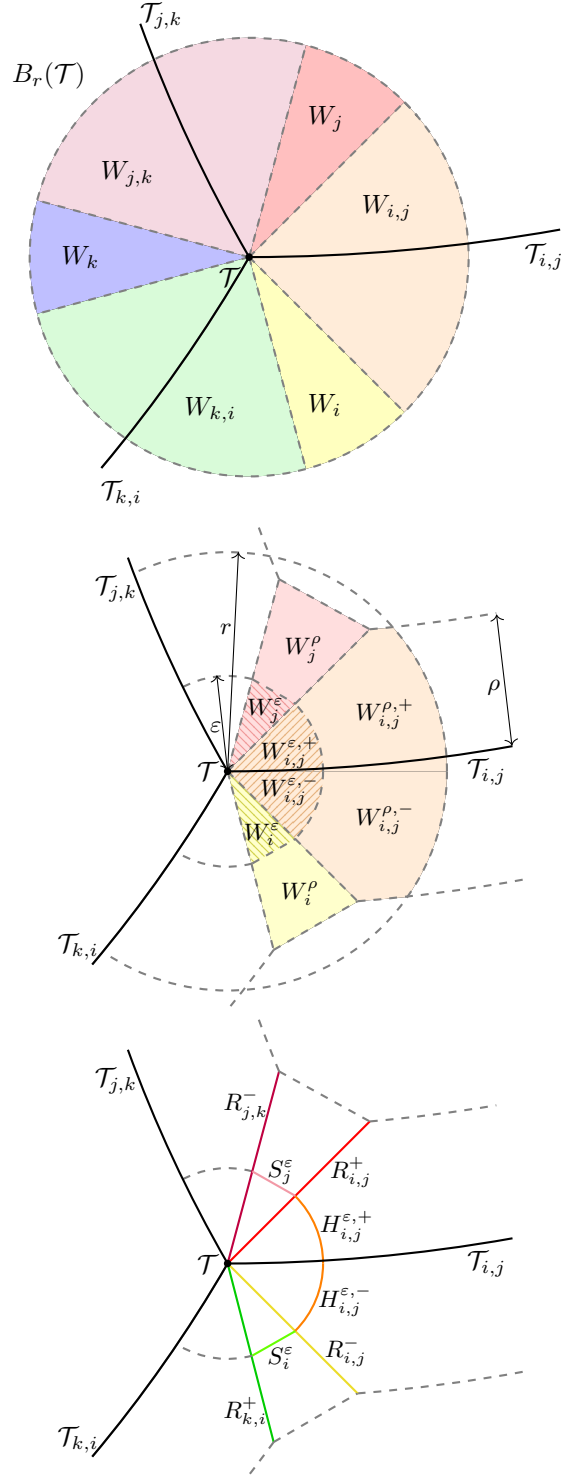
\begin{figure}
	\begin{center}
		\begin{tikzpicture}[scale=0.655]
		\fill[scale=6,inner color = orange!15!, outer color = orange!15!](0,0)--
		(0.525,0.525) arc(45:-45:0.7425);;
		\fill[scale=6,inner color = red!25!, outer color = red!25!](0,0)--
		(0.525,0.525) arc(45:75:0.7425);;
		\fill[scale=6,inner color = yellow!25!, outer color = yellow!25!](0,0)--
		(0.525,-0.525) arc(-45:-75:0.7425);;
		\fill[scale=6,inner color = blue!25!, outer color = blue!25!](0,0)--
		(-0.7172,0.19217) arc(165:195:0.7425);;
		\fill[scale=6,inner color = purple!15!, outer color = purple!15!](0,0)--
		(-0.7172,0.19217) arc(165:75:0.7425);;
		\fill[scale=6,inner color = black!10!green!15!, outer color = black!10!green!15!](0,0)--
		(-0.7172,-0.19217) arc(195:285:0.7425);;
		\draw[scale=6,thick,dashed,gray]  (0.525,0.525) arc(45:-320:0.7425);;
		\draw[scale=6,thick,gray,dashed] (0,0) -- (0.525,0.525);;
		\draw[scale=6,thick,gray,dashed] (0,0) -- (-0.7172,-0.19217);;
		\draw[scale=6,thick,gray,dashed] (0,0) -- (-0.7172,0.19217);;
		\draw[scale=6,thick,dashed,gray] (0,0) -- (0.19,0.71);;
		\draw[scale=6,thick,dashed,gray] (0,0) -- (0.525,-0.525);;	
		\draw[scale=6,thick,dashed,gray] (0,0) -- (0.19,-0.71);;		
		\draw[scale=6,thick,black]  (0,0) arc(270:280.1:6);
		\draw[scale=6,thick,gray]  (0,0) arc(330:320:5);
		\node[right,black]  (none) at (-3.2,-4.8) {$ \mathcal{T}_{k,i}$};;
		\node[right,black] (none) at (-2.4,4.9) {$ \mathcal{T}_{j,k}$};
		\draw[scale=6,thick,black]  (0,0) arc(330:320:5);
		\draw[scale=6,thick,black]  (0,0) arc(210:200:5);
		\node (none) at (0,0) [circle,fill,black,inner sep=1pt]{};
		\node[black,left] (none) at (0,-0.4) { 	$\mathcal{T}$};
		\node[black,right] (none) at (5.4,0) { 	$ \mathcal{T}_{i,j}$};
		\node[black,right] (none) at (-5,3.7) { 	$ B_{r}(\mathcal{T})$};
		\node[black,right] (none) at (1,-3) { 	$W_i$};
		\node[black,right] (none) at (2.1,1) { 	$W_{i,j}$};
		\node[black,right] (none) at (1,2.8) { 	$W_j$};
		\node[black,right] (none) at (-3.2,1.7) { 	$W_{j,k}$};
		\node[black,right] (none) at (-1.5,-3.1) { 	$W_{k,i}$};
		\node[black,right] (none) at (-4,0) { 	$W_k$};
		\end{tikzpicture}
	\end{center}
	
	\begin{center}
		\begin{tikzpicture}[scale=0.45]
				\fill[scale=8,inner color = orange!15!, outer color = orange!15!] (0,0) --  (0.81,0) arc(0:41:0.81) -- (0.525,0.525);;	
			\fill[scale=8,inner color = orange!15!, outer color = orange!15!] (0,0) --  (0.81,0) arc(0:-35:0.81) -- (0.48,-0.48) ;;	
	\fill[scale=8,inner color = red!13!, outer color =  red!13!] (0.19,0.71)-- (0.525,0.525) -- (0,0);
				\fill[scale=8,inner color = yellow!20!, outer color = yellow!20!] (0, 0) --	(0.481,-0.481) -- (0.17,-0.66);
					\fill [pattern=north east lines, pattern color=black!20!orange!50,scale=8]
					(0,0) --  (0.25,0.25) arc(45:-45:0.3525);;				
						\fill [pattern=north west lines, pattern color=black!15!red!40,scale=8]
						(0.091,0.3405)-- (0.25,0.25) -- (0,0)	;
							\fill [pattern=north west lines, pattern color=black!25!yellow!70,scale=8]
							(0.091,-0.3405)-- (0.25,-0.25) -- (0,0)	;
			\draw[scale=8,thick,dashed,gray]  (0.81,0) arc(0:116:0.81);;
				\draw[scale=8,thick,dashed,gray]  (0.81,0) arc(0:-124:0.81);;
		\draw[scale=8,thick,dashed,gray]  (0.091,0.3405) arc(75:120:0.3525);;
		\draw[scale=8,thick,gray,dashed] (0,0) -- (0.525,0.525);;	
		\draw[scale=8,thick,dashed,gray] (0,0) -- (0.19,0.71);;	
		\draw[scale=8,thick,dashed,gray]  (0.525,0.525) arc(275:279.5:6);
		\draw[scale=8,thick,dashed,gray] (0.48,-0.48) arc(274.5:280.5:6) ;;	
		\draw[scale=8,thick,dashed,gray]  (0.19,0.71)--(0.525,0.525);;	
		\draw[scale=8,thick,dashed,gray] (0,0) -- (0.481,-0.481);;	
		\draw[scale=8,thick,dashed,gray] (0,0) -- (0.17,-0.66);;		
		\draw[scale=8,thick,dashed,gray]  (0.17,-0.66)	--(0.48,-0.48);;
		\draw[scale=8,thick,dashed,gray]  (0.3525,0) arc(0:45:0.3525);;	
		\draw[scale=8,thick,dashed,gray]  (0.3525,0) arc(0:-45:0.3525);;	
		\node[black] (none) at (-0.35,1.3) {\small$\vareps$};
		\draw[scale=8,thick,dashed,gray]  (0.091,0.3405)-- (0.25,0.25);;	
		\draw[scale=8,thick,dashed,gray]  (0.091,-0.3405)-- (0.25,-0.25);;	
		\draw[scale=8,thick,dashed,gray]  (0.091,-0.3405) arc(285:240:0.3525);;	
		\draw[scale=8,thick,black]  (0,0) arc(270:280.1:6);
		\draw[scale=8,thick,gray]  (0,0) arc(330:320:5);
		\node[right,black]  (none) at (-5.3,-5.2) {$ \mathcal{T}_{k,i}$};;
		\node[right,black] (none) at (-4.2,5.4) {$ \mathcal{T}_{j,k}$};
		\draw[scale=8,thick,black]  (0,0) arc(330:320:5);
		\draw[scale=8,thick,gray,dashed]  (0.17,-0.66) arc(324:321:5);
		\draw[scale=8,thick,black]  (0,0) arc(210:200:5);
		\draw[scale=8,thick,gray,dashed]  (0.19,0.71) arc(202.5:200:5);
		\node (none) at (0,0) [circle,fill,black,inner sep=1pt]{};
		\draw[scale=8,black,<->]  (0.0,0.0)-- (-0.04,0.355);;
		\node[black,left] (none) at (0,0) { 	$\mathcal{T}$};
		\node[black,right] (none) at (6.8,0) { 	$ \mathcal{T}_{i,j}$};
			\node[black,right] (none) at (3.5,1.7) { 	\small $W^{\rho,+}_{i,j} $};
				\node[black,right] (none) at (3.5,-1.4) { 	\small $W^{\rho,-}_{i,j} $};
					\node[black,right] (none) at (0.7,0.55) { 	\small $ 	W^{\vareps,+}_{i,j}$};
				\node[black,right] (none) at (0.7,-0.55) {\small $W^{\vareps,-}_{i,j}$};
				\draw[scale=8,black,<->]  (1.05,0.095)-- (0.993,0.58);;
					\node[black,right] (none) at (7.4,2.5) { 	\small $\rho$};
							\draw[scale=8,black,->]  (0,0.0)-- (0.04,0.81);;
								\node[black,right] (none) at (-0.5,4.2) {\small $r$};
									\node[black,right] (none) at (1.4,3.6) {\small $W^{\rho}_{j}$};
								\node[black,right] (none) at (1.3,-3.6) {\small $W^{\rho}_{i}$};
								\node[black,right] (none) at (0.3,1.8) {\small $ 	W^{\vareps}_{j} $};
								\node[black,right] (none) at (0.2,-1.85) {\small $ 	W^{\vareps}_{i} $};
		\end{tikzpicture}
	\end{center}
	\begin{center}
	\begin{tikzpicture}[scale=0.45]
	\draw[scale=8,thick,dashed,gray]  (0.091,0.3405) arc(75:120:0.3525);;
	\draw[scale=8,thick,red] (0,0) -- (0.525,0.525);;	
	\draw[scale=8,thick,purple] (0,0) -- (0.19,0.71);;	
	\draw[scale=8,thick,dashed,gray]  (0.525,0.525) arc(275:279.5:6);
	\draw[scale=8,thick,dashed,gray] (0.48,-0.48) arc(274.5:280.5:6) ;;	
	\draw[scale=8,thick,dashed,gray]  (0.19,0.71)--(0.525,0.525);;	
	\draw[scale=8,thick,black!10!yellow] (0,0) -- (0.481,-0.481);;	
	\draw[scale=8,thick,black!20!green] (0,0) -- (0.17,-0.66);;		
	\draw[scale=8,thick,dashed,gray]  (0.17,-0.66)	--(0.48,-0.48);;
	\draw[scale=8,thick,orange]  (0.3525,0) arc(0:45:0.3525);;	
	\draw[scale=8,thick,orange]  (0.3525,0) arc(0:-45:0.3525);;	
	\draw[scale=8,thick,purple!20!pink]  (0.091,0.3405)-- (0.25,0.25);;	
	\draw[scale=8,thick,green!60!yellow]  (0.091,-0.3405)-- (0.25,-0.25);;	
	\draw[scale=8,thick,dashed,gray]  (0.091,-0.3405) arc(285:240:0.3525);;	
	\draw[scale=8,thick,black]  (0,0) arc(270:280.1:6);
	\draw[scale=8,thick,gray]  (0,0) arc(330:320:5);
	\draw[scale=8,thick,black]  (0,0) arc(330:320:5);
	\draw[scale=8,thick,gray,dashed]  (0.17,-0.66) arc(324:322:5);
	\draw[scale=8,thick,black]  (0,0) arc(210:200:5);
	\draw[scale=8,thick,gray,dashed]  (0.19,0.71) arc(202.5:200:5);
	\node (none) at (0,0) [circle,fill,black,inner sep=1pt]{};
	\node[black,left] (none) at (0,0) { 	$\mathcal{T}$};
	\node[right,black]  (none) at (-5.3,-5.2) {$ \mathcal{T}_{k,i}$};;
	\node[right,black] (none) at (-4.2,5.4) {$ \mathcal{T}_{j,k}$};
		\node[black,right] (none) at (6.8,0) { 	$ \mathcal{T}_{i,j}$};
	\node[black,right] (none) at (2.75,2.7) { 	\small $ R^+_{i,j}$};
		\node[black,right] (none) at (-0.4,4.7) { 	\small $ R^-_{j,k}$};
	\node[black,right] (none) at (2.6,-2.6) { 	\small $ R^-_{i,j}$};
		\node[black,right] (none) at (-0.4,-4.4) { 	\small $ R^+_{k,i}$};
\node[black,right] (none) at (2.45,1.1) { \small $H^{\vareps,+}_{i,j} $};
\node[black,right] (none) at (2.4,-1.1) { \small $H^{\vareps,-}_{i,j} $};
	\node[black,right] (none) at (0.9,2.7) { \small $ S^\vareps_{j}$};
		\node[black,right] (none) at (0.9,-2.8) { \small $ S^\vareps_{i}$};
	\end{tikzpicture}
\end{center}
	\caption{Notation and geometry for the construction of well-prepared initial data at a triple junction $\mathcal{T}$ ($d=2$) formed by the phases $i,j$ and $k$ for mutually distinct $i,j,k \in \{1,...,N\}$.}
\end{figure}

Let $r := \min_{p \in \mathcal{P}} r_p$, where $r_p$ admissible localization radius for the triple junction $\mathcal{T}_p$. Let $\rho \in (0, r) $.
Consider a triple junction $\mathcal{T}$, which we assume for simplicity formed by the phases $1,2$ and $3$.
Let $\vareps < \rho $, then for all $i \in \{1, 2, 3\}$ we define
\begin{enumerate}[label=(\roman*)]
	\item the two-dimensional regions \begin{align*}
	W^{\rho,+}_{i,i+1} &:= W_{i,i+1} \cap \{x \in \R^2: 0 \leq \dist^\pm(x, \mathcal{T}_{i,i+1}) \leq \rho \} \,, \\
	 W^{\rho,-}_{i,i+1} &:= W_{i,i+1} \cap \{x \in \R^2: -\rho \leq \dist^\pm(x, \mathcal{T}_{i,i+1}) \leq 0 \} \,, \\
	 	W^\rho_{i,i+1} &:= W^{\rho,+}_{i,i+1}  \cup  W^{\rho,-}_{i,i+1} \,,\\
	W^{\vareps,\pm}_{i,i+1} &:= W^{\rho,\pm}_{i,i+1} \cap B_\vareps(\mathcal{T}) \,, \\
	W^\vareps_{i,i+1} &:= W^{\vareps,+}_{i,i+1}  \cup  W^{\vareps,-}_{i,i+1}\,,
	\end{align*}
	satisfying the inclusions
	\begin{align*}
	  \mathcal{T}_{i,i+1}\cap B_r (\mathcal{T})  \subset \ W^\rho_{i,i+1}\,, \quad 
	\mathcal{T}_{i,i+1} \cap B_\vareps (\mathcal{T})  \subset 	W^\vareps_{i,i+1}  \subset  W^\rho_{i,i+1} \,;
	\end{align*}
	\item the one-dimensional segments resp. arcs  \begin{align*}
	 R^i_{i,i+1}&:= \left(\mathcal{T}+ \left\{a X_{i,i+1}^{i}: a \in(0, \infty)\right\}\right) \cap \partial{	W^\rho_{i, i+1} }\,, \\
	 R^{+}_{i,i+1}&:= R^{i+1}_{i,i+1}, \quad  R^{ -}_{i,i+1}:= R^i_{i,i+1}\,, \\
	H^\vareps_{i,i+1} &:= W^\rho_{i,i+1} \cap \partial B_\vareps(\mathcal{T}) \,, 	\\
	H^{\vareps,+}_{i,i+1} &:= H^\vareps_{i,i+1} \cap \{x \in \R^2: 0 \leq \dist^\pm(x, \mathcal{T}_{i,i+1}) \leq \rho \} \,, \\
	H^{\vareps,-}_{i,i+1} &:= H^\vareps_{i,i+1} \cap \{x \in \R^2: - \rho \leq \dist^\pm(x, \mathcal{T}_{i,i+1}) \leq 0 \} \,, 
	\end{align*}
	\item $S^\vareps_{i}$ as the segment connecting $ R^i_{i,i+1} \cap \partial B_\vareps(\mathcal{T}) $ to $R^i_{i,i-1} \cap \partial B_\vareps(\mathcal{T}) $;
	\item  the two-dimensional triangular regions $W^\rho_{i}$ resp. $ W^\vareps_{i}$
		 as the one with sides $R^i_{i,i+1}$ and $R^i_{i,i-1}$  resp. the one delimitated by $S^\vareps_{i}$, $R^i_{i,i+1}$ and $R^i_{i,i-1}$, thus satisfying the inclusions
		 \begin{align*}
		 	W^\vareps_{i} \subset W^\rho_{i} \subset  W_{i} \,.
		 \end{align*}
\end{enumerate}

Furthermore, we introduce $P_{ R^{+}_{i,i+1}}: W^\rho_{i} \cup  (W^\rho_{i,i+1} \cap \{x \in \Psi_{i, i+1}({I}_{i,i+1} \times [0,\rho))\} ) \rightarrow R^{+}_{i,i+1} $, $P_{ R^-_{i,i+1}}: W^\rho_{i} \cup  (W^\rho_{i,i+1} \cap \{x \in \Psi_{i, i+1}({I}_{i,i+1} \times (-\rho,0])\} ) \rightarrow R^-_{i,i+1} $, $P_{ H^\vareps_{i,i+1} }: W^\vareps_{i,i+1}  \rightarrow H^\vareps_{i,i+1}$ and $P_{ S^\vareps_{i} }:  W^\vareps_{i}  \rightarrow S^\vareps_{i}$ as the orthogonal projections onto the nearest point on $R^{+}_{i,i+1}$, $R^{-}_{i,i+1}$, $H^\vareps_{i,i+1} $ and $S^\vareps_{i} $, respectively.

\subsection{Construction of the initial datum}

We construct the initial datum $u_\vareps(0)$ seperately in each two-dimensional region identified by the geometry of the network as introduced above. Then, we show that in each of these regions the initial relative energy estimate \eqref{eq:initialestimate} holds true.

\paragraph{\textit{Neighborhood of a connected component of a two-phase interface.}}
Let $\mathcal{T}_{i, j}$ be a connected component of $I_{i,j}$ with either one or two endpoints at a triple junction for some distinct $i,j \in  \{1,...,N\}$. Let $\mathcal{P}_{i, j} \subset \mathcal{P}$ enumerate the numbers of triple junctions as endpoints of $\mathcal{T}_{i, j}$. Then, we can define 
\begin{align*}
M_{i,j} :=& \bigg(\{x \in  \R^2: - \rho \leq \dist^\pm(x, \mathcal{T}_{i,i+1}) \leq \rho\} \setminus \bigcup_{p \in \mathcal{P}_{i, j} } B_r(\mathcal{T}_p)\bigg)\\
&\cup \bigg(\bigcup_{p \in \mathcal{P}_{i, j} } (W^\rho_{i,j} \setminus W^\vareps_{i,j} ) \bigg)  \,.
\end{align*}
In the two-dimensional region $M_{i,j}$ we define the initial datum $u_0$ by means of the rescaled one-dimensional equilibrium profile \eqref{eq:equilibriumprofile} as follows
\begin{align}
	u_0(x) := \theta_{i,j}(\vareps^{-1} \dist^\pm(x, \mathcal{T}_{i, j})) \quad \text{ for any } x \in M_{i,j}\,,
\end{align}
whence we obtain
\begin{align*}
	&E_{M_{i,j}}[u_\vareps|\xi](0) \\
	&:= \int_{M_{i,j}} \frac{1}{2 \vareps } | (\theta_{i,j})^\prime(\vareps^{-1} \dist^\pm(x, \mathcal{T}_{i, j})) |^2 + \frac{1}{\vareps} W(\theta_{i,j}(\vareps^{-1} \dist^\pm(x, \mathcal{T}_{i, j}))) \\
	&\quad\quad\quad  -  \frac{1}{2 \vareps }( \xi_{i,j} \cdot n_{i,j}) (\theta_{i,j})^\prime (\vareps^{-1} \dist^\pm(x, \mathcal{T}_{i, j}))\cdot \partial_u \psi_{i,j}(\theta_{i,j}(\vareps^{-1} \dist^\pm(x, \mathcal{T}_{i, j}))) \dx \\
	&\leq \int_{M_{i,j}} \frac{1}{2 \vareps } | (\theta_{i,j})^\prime (\vareps^{-1} \dist^\pm(x, \mathcal{T}_{i, j}))|^2 + \frac{1}{\vareps} W(\theta_{i,j}(\vareps^{-1} \dist^\pm(x, \mathcal{T}_{i, j}))) \\
	&\quad\quad\quad-  \frac{1}{ \vareps } |\xi_{i,j} |^2 \frac{1}{\bar{\theta}^\pm_{i,j}}
	 2W(\theta_{i,j}(\vareps^{-1} \dist^\pm(x, \mathcal{T}_{i, j}))) \dx \\
	& \leq \int_{M_{i,j}} \hspace{-3mm}(1-  |\xi_{i,j} |^2 )	\bigg[ \frac{1}{2 \vareps } | (\theta_{i,j})^\prime(\vareps^{-1} \dist^\pm(x, \mathcal{T}_{i, j})) |^2 \hspace{-0.5mm}+ \frac{1}{\vareps} W(\theta_{i,j}(\vareps^{-1} \dist^\pm(x, \mathcal{T}_{i, j}))) \bigg]\hspace{-0.5mm} \dx.
\end{align*}
Here, we used that $\psi_k=0$ along $\gamma_{i,j}$ for any $k \in \{1,...,N \}\setminus\{i , j \}$ and  that $\partial_u \psi_{i,j}(\gamma_{i,j}) \cdot \frac{\gamma_{i,j}^\prime}{|\gamma_{i,j}^\prime|}  =2 \sqrt{2W(\gamma_{i , j }) }$.  Indeed, assumption (A4) implies $ \partial_u \psi_{i,j}(\gamma_{i , j }) \cdot {\gamma_{i,j}^\prime} \leq  2\sqrt{2W(\gamma_{i , j }) } {|\gamma_{i,j}^\prime|}$, then (A3) gives the equality sign by contradiction.
Then, in the last step we added a zero and we used $| (\theta_{i,j})^\prime (s)|^2  \leq \frac{1}{(\bar{\theta}^\pm_{i,j})^2} 2W(\theta_{i,j}(s))$.
Note that $1-  |\xi_{i,j} |^2  \leq c \dist^2(\cdot, \mathcal{T}_{i,j})$, and that $W(\theta_{i,j}(s))$ has an exponential resp. a power-law decay of order $\frac{2q}{q-2}$ for $q =2$ resp. $q >2$ as $s$ approaches the extrema of $(-\rho,\rho)$, then it vanishes for $s \in (-\infty,-\rho]\cup [\rho, \infty)$. As a consequence, we obtain $E_{M_{i,j}}[u_\vareps|\xi](0) \leq C \vareps^2$ for some constant $C>0$. 

\paragraph{\textit{Pure-phase region.}} 
Let $i \in \{1,...,N\}$. Let $\mathcal{P}_{i} \subset \mathcal{P}$ enumerate the numbers of triple junctions as endpoints of a connected component of an interface between the phase $i$ and any other one.
We set $u_0=\alpha_i$ in the pure-phase region
$\{\bar{\chi}_i=1\} \setminus \left(\bigcup_{p \in \mathcal{P}_{i} }W_i \cup \bigcup_{j: j\neq i } \{x - s n_{i,j}(x), \, x \in  {I}_{i,j}\,, s \in [0, \rho)\} \right)$. Then
$|\nabla u_0 |=0$ and, having $W(\alpha_i)=0$, the initial relative entropy is equal to zero in the pure-phase region.

\paragraph{\textit{Triple junction wedge containing a connected component of a two-phase interface.}} 
Given a triple junction $\mathcal{T}$, let $i,j \in \{1,...,N\}$,$i \neq j$, be two of the three phases forming $\mathcal{T}$ and let $k \in \{1,...,N\} \setminus\{i,j\}$ be the third one. 
The initial datum $u_0$ in the corresponding wedge $W^{\vareps, \pm}_{i,j}$ is given by interpolation via orthogonal projections $P_{ H^\vareps_{i,j} }, P_{ {I}_{i,j}}$ and $P_{ R^{\pm}_{i,j}}$, which reads as
\begin{align*}
	u_0(x) = &\frac{\dist(x, {I}_{i,j}) + \dist(x, R^{ \pm}_{i,j})}{\dist(x, H^\vareps_{i,j} ) + \dist(x,{I}_{i,j}) + \dist(x, R^{ \pm}_{i,j})} u_{ H^\vareps_{i,j} }( P_{ H^\vareps_{i,j} } x) \\
	&+ \frac{\dist(x,H^\vareps_{i,j} ) + \dist(x,R^{\pm}_{i,j})}{\dist(x, H^\vareps_{i,j} ) + \dist(x, {I}_{i,j}) + \dist(x, R^{\pm}_{i,j})} u_{ {I}_{i,j}}(P_{ {I}_{i,j}}x) \\
	&+ \frac{\dist(x,H^\vareps_{i,j} ) + \dist(x,{I}_{i,j} )}{\dist(x, H^\vareps_{i,j} ) + \dist(x, {I}_{i,j}) + \dist(x, R^{ \pm}_{i,j})}  u_{ R_{i,j}} ( P_{ R^{ \pm}_{i,j}} x )
\end{align*}
for any $x \in W^{\vareps, \pm}_{i,j}$, where 
\begin{align*}
	&u_{ H^\vareps_{i,j} }(x) = \frac{h^{ \pm}_{i,j}(x, \mathcal{T})}{h^{\vareps, \pm}_{i,j}} \theta_{i,j}( \pm \vareps^{-1} \bar{h}^{\vareps, \pm}_{i,j}) + \frac{h^{\vareps, \pm}_{i,j} - h^{ \pm}_{i,j} (x, \mathcal{T})}{h^{\vareps, \pm}_{i,j}} \theta_{i,j}(0) \notag \\
	&\text{for any } x \text{ along } H^\vareps_{i,j} \,,\\
	& u_{ {I}_{i,j}} (x) = \frac{l_{i,j}(x, \mathcal{T})}{l^\vareps_{i,j}} \theta_{i,j}(0) + \frac{l^\vareps_{i,j} - l_{i,j}(x, \mathcal{T})}{l^\vareps_{i,j}} \bar{\alpha} \notag\\
		&\text{for any } x \text{ along } {I}_{i,j} \,, \\
			& u_{ R_{i,j}} (x) = \frac{r^\pm_{i,j}(x, \mathcal{T})}{\vareps} \theta_{i,j}( \pm \vareps^{-1}  \bar{h}^{\vareps,\pm}_{i,j}) + \frac{\vareps - r^\pm_{i,j}(x, \mathcal{T})}{\vareps} \bar{\alpha} \notag\\
		&\text{for any } x \text{ along } R^{ \pm}_{i,j} \,, 
\end{align*}
where $\bar{\alpha} = \frac{\alpha_i + \alpha_j + \alpha_k}{3} $, $ \bar{h}^{\vareps, \pm}_{i,j}:= \dist^\pm(H^\vareps_{i,j} \cap R^{\pm}_{i,j}, {I}_{i,j} )$, $l^\vareps_{i,j}$ resp. $h^{\vareps, \pm}_{i,j}$ is the length of ${I}_{i,j} \cap B_\vareps(\mathcal{T})$ resp. $\partial B_\vareps(\mathcal{T}) \cap W^{\vareps, \pm}_{i,j}$, whereas $h^{ \pm}_{i,j}(\cdot, \mathcal{T})$,
$l_{i,j}(\cdot, \mathcal{T})$ resp. $r^\pm_{i,j}(\cdot, \mathcal{T})$ is the length of the path along $H^{\vareps,  \pm}_{i,j}$ , ${I}_{i,j}$ resp. $R^{\pm}_{i,j}$ connecting to $\mathcal{T}$.
Since $l^\vareps_{i,j}$ and $h^{\vareps, \pm}_{i,j} $ are of order $\vareps$, then this construction gives 
that $|\nabla u_0(x) |$ is of order $ {1}/{\vareps}$ $ \text{for any } x \in W^\vareps_{i,j} $.
Hence, being the area of $W^\vareps_{i,j}$ of order $\vareps^2$, we can deduce 
\begin{align*}
E_{W^{\vareps}_{i,j}}[u_\vareps|\xi](0) :=& \int_{W^{\vareps}_{i,j}} \frac{\vareps}{2  } |\nabla u_0|^2 + \frac{1}{\vareps} W(u_0)  
+ \sum_{\ell=1}^N \xi_\ell \cdot \nabla (\psi_{\ell} \circ u_0) \dx \\
\leq & \int_{W^{\vareps}_{i,j}} \frac{\vareps}{2  } |\nabla u_0|^2 + \frac{C}{\vareps} 
+  C | \nabla  u_0| \dx 
\leq C \vareps 	\,,
\end{align*}
for some constant $C>0$ varying from line to line.

\paragraph{\textit{Triple junction wedge not containing any connected component of a two-phase interface.}} 
Given a triple junction $\mathcal{T}$, let $i,j,k \in \{1,...,N\}$ be three distinct phases forming $\mathcal{T}$.
The initial datum $u_0$ in the region $W^\vareps_{j}$ is given by interpolation via orthogonal projections $P_{ S^\vareps_{j} }, P_{ R^-_{j,k}}$ and $P_{ R^+_{i,j}}$, which reads as
\begin{align*}
u_0(x) = &\frac{\dist(x, R^-_{j,k}) + \dist(x, R^+_{i,j})}{\dist(x, S^\vareps_{j} ) + \dist(x, R^-_{j,k}) + \dist(x, R^+_{i,j})} u_{ H^\vareps_{i,j} }( P_{ S^\vareps_{j} } x) \\
&+ \frac{\dist(x,S^\vareps_{j} ) + \dist(x,R^+_{i,j})}{\dist(x, S^\vareps_{j} ) + \dist(x, R^-_{j,k}) + \dist(x, R^+_{i,j})} u_{ R^-_{j,k}}(P_{ R^-_{j,k}}x) \\
&+ \frac{\dist(x,S^\vareps_{j} ) + \dist(x,R^-_{j,k} )}{\dist(x, S^\vareps_{j} ) + \dist(x, R^-_{j,k}) + \dist(x, R^+_{i,j})} u_{ R^+_{i,j} }( P_{ R^{+}_{i,j}} x ) \,,
\end{align*}
for any $x \in W^\vareps_{j}$, where
\begin{align}
&u_{ S^\vareps_{j} }(x) = \frac{s_j(x)}{s^\vareps_{j}} \theta_{i,j}( \vareps^{-1} \bar{h}^\vareps_{i,j}) + \frac{s^\vareps_{j} - s_j(x)}{s^\vareps_{j}} \theta_{j,k}(-\vareps^{-1} \bar{h}^\vareps_{j,k}) \notag \\
&\text{for any } x \text{ along } S^\vareps_{j} \,,\\
&u_{ R^-_{j,k}}(x) = \frac{r_{j,k}(x, \mathcal{T})}{\vareps} \theta_{j,k}(-\vareps^{-1} \bar{h}^\vareps_{j,k}) + \frac{\vareps - r_{j,k}(x, \mathcal{T})}{\vareps} \bar{\alpha} \notag\\
&\text{for any } x \text{ along } R^-_{j,k} \,, \\
&u_{ R^+_{i,j}}(x) = \frac{r_{i,j}(x, \mathcal{T})}{\vareps} \theta_{i,j}(\vareps^{-1} \bar{h}^\vareps_{i,j}) + \frac{\vareps - r_{i,j}(x, \mathcal{T})}{\vareps} \bar{\alpha} \notag\\
&\text{for any } x \text{ along } R^+_{i,j} \,, 
\end{align}
where $ \bar{h}^\vareps_{i,j}:= \dist(H^\vareps_{i,j} \cap R^{+}_{i,j}, {I}_{i,j} )$, $ \bar{h}^\vareps_{j,k}:= \dist(H^\vareps_{j,k} \cap R^{-}_{j,k}, {I}_{j,k} )$,
$s^\vareps_j$ is the lenght of the segment $S^\vareps_{j}$,
whereas $s_j$ resp. $r_{i,j}(\cdot, \mathcal{T})$ is the lenght of the path along $S^\vareps_{j}$ resp. $R^{+}_{i,j}$ connecting to $R^-_{j,k}$ resp. $\mathcal{T}$.
Since $s^\vareps_{j}$ is of order $ \vareps$, we have that
$
|\nabla u_0(x) |$ is of order ${1}/{\vareps}$ { for any } $x \in W^\vareps_{j}$,
whence as above we can deduce that
$
E_{W^\vareps_{j}}[u_\vareps|\xi](0) :=\int_{W^{\vareps}_{j}} \frac{\vareps}{2  } |\nabla u_0|^2 + \frac{1}{\vareps} W(u_0)  
+ \sum_{\ell=1}^N \xi_\ell \cdot \nabla (\psi_{\ell} \circ u_0) \dx \leq C \vareps,
$
for some constant $C>0$.

\paragraph{\textit{Interpolation between two rescaled one-dimensional equilibrium profiles.}}
First, we introduce for any $x \in W^\rho_{j} \setminus W^\vareps_{j}$ the set of coordinates $(s,h)$, where $h$ denotes the distance from $S_j^\vareps$ while $s$ is such that $\nabla s \cdot \nabla h =0$ and $s=0$ whenever $x \in R^-_{j,k}$. 
Hence, $h \in [0,\tilde{r}_\vareps ]$, where $\tilde{r}_\vareps:=( r- \varepsilon)\cos(\frac{\pi}{12})$ is fixed, and $s \in [0,2 \tilde{h}_\vareps \sin(\frac{\pi}{12})]$, where $\tilde{h}_\vareps:= ( \cos(\frac{\pi}{12}))^{-1}h + \vareps $.
The initial datum $u_0$ in $W^\rho_{j}\setminus W^\vareps_{j}$ is given as follows
\begin{align}
	u_0(x)=&  \frac{s}{2 \tilde{h}_\vareps \sin(\frac{\pi}{12})}  \theta_{j,k}(-\vareps^{-1} \dist(P^s_{R^-_{j,k}}x, {I}_{j,k})) \notag\\
	&+ \frac{2 \tilde{h}_\vareps \sin(\frac{\pi}{12})-s}{2 \tilde{h}_\vareps \sin(\frac{\pi}{12}) }  \theta_{i,j}(\vareps^{-1} \dist(P^s_{R^+_{i,j}}x, {I}_{i,j})) \,,
\end{align}
for any $x \in W^\rho_{j}\setminus W^\vareps_{j}$, where $P^s_{R^-_{j,k}}$ and $P^s_{R^+_{i,j}}$ projections along the $s$-axis onto $R^-_{j,k}$ and $R^+_{i,j}$, respectively.
Hence, we compute 
\begin{align*}
&\partial_s u_0(x)\\
&=  \frac{1}{2 \tilde{h}_\vareps \sin(\frac{\pi}{12})}  \left(\theta_{j,k}(-\vareps^{-1} \dist(P^s_{R^-_{j,k}}x, {I}_{j,k}))  - \theta_{i,j}(\vareps^{-1} \dist(P^s_{R^+_{i,j}}x, {I}_{i,j})) \right) \,, 
\end{align*}
and
\begin{align*}
&\partial_h u_0(x) \\
&=  \frac{s}{2 \tilde{h}^2_\vareps \sin(\frac{\pi}{12})  \cos(\frac{\pi}{12})} \left(\theta_{j,k}(-\vareps^{-1} \dist(P^s_{R^-_{j,k}}x, {I}_{k,j}))  - \theta_{i,j}(\vareps^{-1} \dist(P^s_{R^+_{i,j}}x, {I}_{i,j})) \right) \\
&\quad- \frac{s}{2 \vareps \tilde{h}_\vareps \sin(\frac{\pi}{12})}  \partial_h \dist(P^s_{R^-_{j,k}}x, {I}_{j,k}) ( \theta_{j,k})^\prime(-\vareps^{-1} \dist(P^s_{R^-_{j,k}}x, {I}_{j,k})) \\
&\quad + \frac{2  \tilde{h}_\vareps \sin(\frac{\pi}{12})-s}{2 \vareps \tilde{h}_\vareps \sin(\frac{\pi}{12}) } \partial_h \dist(P^s_{R^+_{i,j}}x, {I}_{i,j}) ( \theta_{i,j})^\prime(\vareps^{-1} \dist(P^s_{R^+_{i,j}}x, {I}_{i,j})) \,.
\end{align*}
For the sake of brevity, we introduce the notation $ \tilde{\lambda}(x):= \frac{s}{2 \tilde{h}_\vareps \sin(\frac{\pi}{12})} \in [0,1]$ for any $x \in W^\rho_{j}\setminus W^\vareps_{j}$. 
By adding zeros and using the Young inequality together with the fact that $s^2 \leq 4 \tilde{h}^2_\vareps \sin^2(\frac{\pi}{12})$, one can obtain
\begin{align*}
&\frac{\vareps}{2} | \nabla u_0 (x)|^2 =\frac{\vareps}{2} |\partial_s u_0(x)|^2 +\frac{\vareps}{2} |\partial_h u_0(x)|^2\\
& \leq C 
 \frac{\vareps}{ \tilde{h}^2_\vareps }  |\theta_{j,k}(- \vareps^{-1} \dist(P^s_{R^-_{j,k}}x, {I}_{j,k}))  - \alpha_j |^2\\
 &\quad + C 
 \frac{\vareps}{ \tilde{h}^2_\vareps }  | \theta_{i,j}(\vareps^{-1} \dist(P^s_{R^+_{i,j}}x, {I}_{i,j}))- \alpha_j |^2 \\
 &\quad + C  \tilde{\lambda}^2(x) \frac{1}{ \vareps}  | ( \theta_{j,k})^\prime(- \vareps^{-1} \dist(P^s_{R^-_{j,k}}x, {I}_{j,k}))|^2 \\
 &\quad+ C(1-\tilde{\lambda}(x))^2 \frac{1}{ \vareps } | (\theta_{i,j})^\prime(\vareps^{-1} \dist(P^s_{R^+_{i,j}}x, {I}_{i,j})|^2\,,
\end{align*}
for some constant $C>0$.
First, we consider
\begin{align*}
&\int_{W^\rho_{j}\setminus W^\vareps_{j}}  \frac{\vareps}{ \tilde{h}^2_\vareps }  |\theta_{j,k}(-\vareps^{-1} \dist(P^s_{R^-_{j,k}}x, {I}_{j,k}))  - \alpha_j |^2\dx \\
&= 2 \sin(\tfrac{\pi}{12}) \int_{0}^{\tilde{r}_\vareps}  \frac{\vareps}{ \tilde{h}_\vareps }  |\theta_{j,k}(-\vareps^{-1} \dist(P^s_{R^-_{j,k}}x, {I}_{k,j}))  - \alpha_j |^2\, \mathrm{d} h\\
& \leq C \int_{0}^{\tilde{r}_\vareps}  |\theta_{j,k}(- \vareps^{-1} \dist(P^s_{R^-_{j,k}}x, {I}_{j,k}))  - \alpha_j |^2\,  \mathrm{d} h \\
&\leq C \int_{0}^{\tilde{r}_\vareps}  \frac{h}{\vareps}   |\theta_{j,k}(\vareps^{-1} \dist(P^s_{R^-_{j,k}}x, {I}_{j,k}))  - \alpha_j | |( \theta_{j,k})^\prime(- \vareps^{-1} \dist(P^s_{R^-_{j,k}}x, {I}_{j,k}))|  \, \mathrm{d} h \\
&\leq C  \int_{0}^{\tilde{r}_\vareps}   \frac{h}{\vareps}    |( \theta_{j,k})^\prime(- \vareps^{-1} \dist(P^s_{R^-_{j,k}}x, {I}_{j,k}))|  \, \mathrm{d} h \,,
\end{align*}
where we integrated by parts and $C>0$ is a suitable constant varying from line to line. Then, we observe that $\dist(P^s_{R^-_{j,k}}x, {I}_{j,k})$ is a homogeneous and increasing function with respect to $h$. On the other hand, we recall that $| (\theta_{j,k})^\prime (s)|$  has an exponential resp. a power-law decay of order $\frac{q}{q-2}$ for $q =2$ resp. $q >2$ as $s$ approaches the extrema of $(-\rho,\rho)$, then it vanishes for $s \in (-\infty,-\rho]\cup [\rho, \infty)$.
As a consequence, we obtain  
\begin{align*}
&\int_{W^\rho_{j}\setminus W^\vareps_{j}}  \frac{\vareps}{ \tilde{h}^2_\vareps }  |\theta_{j,k}(- \vareps^{-1} \dist(P^s_{R^-_{j,k}}x, {I}_{j,k}))  - \alpha_j |^2\dx  \leq C \vareps\,,
\end{align*}
and analogously
\begin{align*}
&\int_{W^\rho_{j}\setminus W^\vareps_{j}}  \frac{\vareps}{ \tilde{h}^2_\vareps }  |\theta_{i,j}(\vareps^{-1} \dist(P^s_{R^+_{i,j}}x, {I}_{i,j}))  - \alpha_j |^2\dx  \leq C \vareps\,.
\end{align*}
Second, we have 
\begin{align*}
	&\int_{W^\rho_{j}\setminus W^\vareps_{j}}  \tilde{\lambda}^2(x) \frac{1}{ \vareps}  | ( \theta_{j,k})^\prime(-\vareps^{-1} \dist(P^s_{R^-_{j,k}}x, {I}_{j,k}))|^2\dx \\&=  2 \sin(\tfrac{\pi}{12}) \int_{0}^{\tilde{r}_\vareps}
	\tilde{\lambda}^2(x) \frac{\tilde{h}_\vareps}{ \vareps}  | ( \theta_{j,k})^\prime(-\vareps^{-1} \dist(P^s_{R^-_{j,k}}x, {I}_{j,k}))|^2 \mathrm{d} h\\
	& \leq C\int_{0}^{\tilde{r}_\vareps} \Big(\frac{{h} }{ \vareps} + 1 \Big)
 | ( \theta_{j,k})^\prime(-\vareps^{-1} \dist(P^s_{R^-_{j,k}}x, {I}_{j,k}))|^2 \mathrm{d} h\\
 &\leq C \vareps\,, 
\end{align*}
and analogously
\begin{align*}
\int_{W^\rho_{j}\setminus W^\vareps_{j}} (1- \tilde{\lambda}(x))^2 \frac{1}{ \vareps}  | ( \theta_{i,j})^\prime(\vareps^{-1} \dist(P^s_{R^+_{i,j}}x, {I}_{i,j}))|^2 \dx 
\leq C \vareps\,.
\end{align*}
Observe that by adding zeros we can write 
\begin{align*}
W(u_0(x)) =  &\,  \tilde{\lambda}(x) \left(W(u_0(x)) -  W(\theta_{j,k}(-\vareps^{-1} \dist(P^s_{R^-_{j,k}}x, {I}_{j,k}))  )  \right)\\
&+ (1-  \tilde{\lambda}(x)) \left( W(u_0(x)) -  W(\theta_{i,j}(\vareps^{-1} \dist(P^s_{R^+_{i,j}}x, {I}_{i,j}))  )\right)\\
 &+  \tilde{\lambda}(x) W(\theta_{j,k}(-\vareps^{-1} \dist(P^s_{R^-_{j,k}}x, {I}_{j,k}))  )\\
&+  (1-  \tilde{\lambda}(x)) W(\theta_{i,j}(\vareps^{-1} \dist(P^s_{R^+_{i,j}}x, {I}_{i,j}))  ) \,,
\end{align*} 
where $ \tilde{\lambda}(x), 1-  \tilde{\lambda}(x) \in [0,1]$ for any $x \in W^\rho_{j}\setminus W^\vareps_{j}$.
Since $W$ is Lipschitz (see (A1) in Sec. \ref{sec:mainresults}), then by adding zeros we obtain
\begin{align*}
	&\left| \frac{1}{\vareps} W(u_0(x)) - \frac{1}{\vareps} W(\theta_{j,k}(-\vareps^{-1} \dist(P^s_{R^-_{j,k}}x, {I}_{j,k}))  ) \right|\\
	&\leq  \frac{C}{\vareps} | {u_0(x)} - \theta_{j,k}(-\vareps^{-1} \dist(P^s_{R^-_{j,k}}x, {I}_{j,k}))  | \\
	&\leq \frac{C}{\vareps}	(1 - \tilde{\lambda}(x)) \left(
	|  \theta_{j,k}(-\vareps^{-1} \dist(P^s_{R^-_{j,k}}x, {I}_{j,k}))  - \alpha_j | + | \theta_{i,j}(\vareps^{-1} \dist(P^s_{R^+_{i,j}}x, {I}_{i,j})) - \alpha_j | 
	\right)
\end{align*}
and 
\begin{align*}
&\left| \frac{1}{\vareps}  W(u_0(x)) - \frac{1}{\vareps} W(\theta_{i,j}(\vareps^{-1} \dist(P^s_{R^+_{i,j}}x, {I}_{i,j}))  ) \right|\\
&\leq  \frac{C}{\vareps} | u_0(x) - \theta_{i,j}(\vareps^{-1} \dist(P^s_{R^+_{i,j}}x, {I}_{i,j}))  | \\
&\leq \frac{C}{\vareps} \tilde{\lambda}(x) \left(
|  \theta_{j,k}(-\vareps^{-1} \dist(P^s_{R^-_{j,k}}x, {I}_{j,k}))  - \alpha_j | + | \theta_{i,j}(\vareps^{-1} \dist(P^s_{R^+_{i,j}}x, {I}_{i,j})) - \alpha_j | 
\right) .
\end{align*}
In particular, we have
\begin{align*}
&\int_{W^\rho_{j}\setminus W^\vareps_{j}}  \frac{1}{ \vareps }  |\theta_{j,k}(- \vareps^{-1} \dist(P^s_{R^-_{j,k}}x,{I}_{j,k}))  - \alpha_j | \dx \\
&= 2 \sin(\tfrac{\pi}{12}) \int_{0}^{\tilde{r}_\vareps}  \frac{\tilde{h}_\vareps }{ \vareps}  |\theta_{j,k}(-\vareps^{-1} \dist(P^s_{R^-_{j,k}}x, {I}_{k,j}))  - \alpha_j | \, \mathrm{d} h\\
& \leq C \int_{0}^{\tilde{r}_\vareps}  \frac{( \cos(\frac{\pi}{12}))^{-1}h + \vareps  }{ \vareps}  |\theta_{j,k}(-\vareps^{-1} \dist(P^s_{R^-_{j,k}}x, {I}_{j,k}))  - \alpha_j | \,  \mathrm{d} h \\
&\leq C \int_{0}^{\tilde{r}_\vareps}    \left( \frac{h^2}{\vareps^2}  + \frac{h}{\vareps}  \right) |( \theta_{j,k})^\prime(-\vareps^{-1} \dist(P^s_{R^-_{j,k}}x, {I}_{j,k}))|  \, \mathrm{d} h \,,
\end{align*}
where we integrated by parts and $C>0$ is a suitable constant varying from line to line. Once again since $\dist(P^s_{R^-_{j,k}}x, {I}_{j,k})$ is a homogeneous and increasing function with respect to $h$, recalling the decay of $| (\theta_{j,k})^\prime (s)|$ in $s$ mentioned above, then we get
\begin{align*}
&\int_{W^\rho_{j}\setminus W^\vareps_{j}}  \frac{1}{ \vareps }  |\theta_{j,k}(-\vareps^{-1} \dist(P^s_{R^-_{j,k}}x, {I}_{j,k}))  - \alpha_j | \dx   \leq C \vareps\,,
\end{align*}
and also
\begin{align*}
&\int_{W^\rho_{j}\setminus W^\vareps_{j}}  \frac{1}{ \vareps }  |\theta_{i,j}(\vareps^{-1} \dist(P^s_{R^+_{i,j}}x, {I}_{i,j}))  - \alpha_j | \dx   \leq C \vareps\,.
\end{align*}
Similarly, we estimate
\begin{align*}
	&\int_{W^\rho_{j}\setminus W^\vareps_{j}}  \frac{1}{\vareps} W(\theta_{j,k}(\vareps^{-1} \dist(P^s_{R^-_{j,k}}x, {I}_{j,k}))  )
	\, \mathrm{d} x \\
	&\leq C \int_{0}^{\tilde{r}_\vareps} \frac{\tilde{h}_\vareps }{\vareps} W(\theta_{j,k}(-\vareps^{-1} \dist(P^s_{R^-_{j,k}}x, {I}_{j,k}))  )
	\, \mathrm{d} h 
	 \leq C \vareps \,, 
\end{align*}
due to the fact that $\dist(P^s_{R^-_{j,k}}x, {I}_{j,k})$ is a homogeneous and increasing function with respect to $h$ and the decay of $W(\theta_{j,k}(s))$ in $s$. Analogously, we have  
\begin{align*}
	\int_{M_{j}}   \frac{1}{\vareps} W(\theta_{i,j}(\vareps^{-1} \dist(P^s_{R^+_{i,j}}x,{I}_{i,j}))  ) \, \mathrm{d} x  \leq C \vareps \,.
\end{align*} 

Finally, \eqref{eq:sum} together with (A4) (see Sec. \ref{sec:mainresults}) and an application of the Young inequality give
\begin{align*}
		\left| \sum^N_{\ell=1} \int_{W^\rho_{j}\setminus W^\vareps_{j}} \xi_\ell \cdot \nabla (\psi_{\ell} \circ u_0) \dx \right| &\leq C  \int_{W^\rho_{j}\setminus W^\vareps_{j}} \sqrt{2W(u_0)} |\nabla u_0| \dx  \\
		&\leq C  \left[   \int_{W^\rho_{j}\setminus W^\vareps_{j}} \frac{\vareps}{2} |\nabla u_0|^2 \dx  +   \int_{W^\rho_{j}\setminus W^\vareps_{j}} \frac1{\vareps} W(u_0)  \dx \right] .
\end{align*}
Then, from the estimates above we can conclude that
\begin{align*}
		E_{W^\rho_{j}\setminus W^\vareps_{j}}[u_\vareps|\xi](0) := \int_{W^\rho_{j}\setminus W^\vareps_{j}} \frac{\vareps}{2  } |\nabla u_0|^2 + \frac{1}{\vareps} W(u_0)  
		+ \sum_{\ell=1}^N \xi_\ell \cdot \nabla (\psi_{\ell} \circ u_0) \dx \leq
		C \vareps 	
\end{align*} for some constant $C>0$.

\section{Suitable multi-well potentials and a construction for the $\psi_i$}

We next proceed to show that the class of $N$-well potentials satisfying assumptions (A1)--(A4) is in fact sufficiently broad.

\subsection{A class of multi-well potentials} \label{subsec:multi-wellpot}
Let \(\triangle^{N-1}\) be a $(N-1)$-simplex with edges of unit length in $\R^{N-1}$. 
We denote by \(\{\gamma_{i,j}\}_{i,j: i \neq j }\) its edges and by \(\{\alpha_{i}\}_{i}\) its $N$ vertices, so that $|\alpha_i -\alpha_j|=1$ for any mutually distinct \(i,j \in \{1,...,N\}\). 
We can decompose $\triangle^{N-1}$ (almost) symmetrically into a disjoint partition $\{\mathcal{T}_{i,j}\}_{i,j: i < j }$ such that each point $x\in \triangle^{N-1}$ is assigned to the set $\mathcal{T}_{i,j}$ if $\gamma_{i,j}$ is the edge of $\triangle^{N-1}$ that is closest to $x$, with $x$ being assigned to the edge with the lowest $i$ and $j$ in case of ties.
Each $\mathcal{T}_{i,j}$ can be further split nearly symmetrically into $\mathcal{T}^i_{j}$ and $\mathcal{T}^j_{i}$ by defining $\mathcal{T}^i_{j}$ to consist of the points in $\mathcal{T}_{i,j}$ that are closer to $\alpha_i$ than to $\alpha_j$. For an illustration of this partition, we refer to Figure~\ref{Fig_sectionstriangle} for $N=3$.

\begin{figure}
	\begin{center}
		\begin{tikzpicture}	
		\fill[scale=6,inner color = black!10!blue!15!, outer color = black!10!blue!15!] (0, 0) --	(1,0) -- (0.5,0.2887);
		\fill[scale=6,inner color = yellow!17!, outer color = yellow!17!] (0, 0) -- (0.5,0.2887)-- (0.5,0.866);
		\fill[scale=6,inner color = black!30!green!15!, outer color = black!30!green!15!] (1, 0) -- (0.5,0.2887)-- (0.5,0.866);
		\draw[scale=6,thick,blue] (0,0) node[left,black] {$\alpha_i$} -- (1,0) node[right,black] {$\alpha_k$};;	
		\draw[scale=6,thick,dashed,gray]  (0,0)-- (0.75,0.4325);;
		\draw[scale=6,thick,dashed,gray]  (1,0)-- (0.25,0.4325);;
		\draw[scale=6,thick,dashed,gray] (0.5,0) -- (0.5,0.866);;
		\draw[scale=6,thick,black!60!green]  (1,0) -- (0.5,0.866) node[above,black] {$\alpha_j$};;
		\draw[scale=6,thick,black!10!orange]  (0.5,0.866) --(0,0);;
		\node[black!10!orange] (none) at (1.7,1.7) { $\mathcal{T}^i_j$};
		\node[black!60!green] (none) at (4.4,1.7) { $\mathcal{T}^k_j$};
		\node[blue] (none) at (2.3,0.6) { $\mathcal{T}^i_k$};
		\node[blue] (none) at (3.8,0.6) { $\mathcal{T}^k_i$};
		\node[black!10!orange] (none) at (2.3,2.9) { $\mathcal{T}^j_i$};
		\node[black!60!green] (none) at (3.8,2.9) { $\mathcal{T}^j_k$};
		\node[black!10!orange] (none) at (1,2.7) { $\gamma_{i,j}$};
		\node[black!60!green] (none) at (5,2.7) { $\gamma_{j,k}$};
		\node[blue,below] (none) at (3,-0.2) { $\gamma_{k,i}$};
		\end{tikzpicture}
	\end{center}
	\caption{Illustration of the disjoint partition $\{\mathcal{T}^i_{j}\}_{i,j: i \neq j }$ of the simplex $\triangle^2$.}
	\label{Fig_sectionstriangle}
\end{figure}
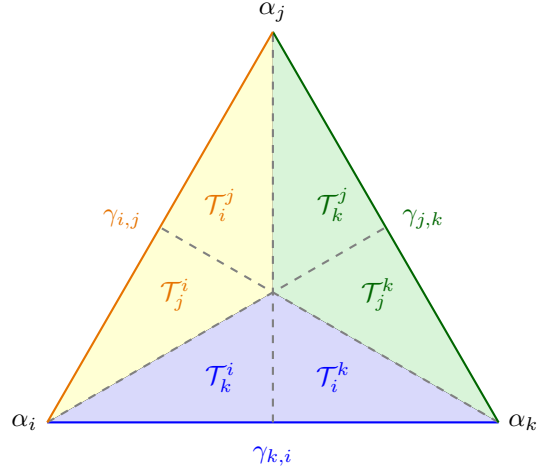

For the purpose of our construction of the $\psi_i$ from condition (A4), we introduce some further notation:
\begin{itemize}
\item For $ i \in \{1,...,N\}$, we denote by $\mathcal{U}_i:= \overline{B_{r_\mathcal{U}}(\alpha_i)}$ a ball around the vertex $\alpha_i$ with radius $r_\mathcal{U} \in \left(0,  \frac14\right]$.
\item For $i<j, \, i,j \in \{1,...,N\}$, we denote by $\mathcal{N}_{i,j}:=\{u \in \R^{N-1}: \dist(u, \gamma_{i,j}) \leq r_\mathcal{U} \sin(\beta_\mathcal{N})\}$ a neighborhood of the edge $\gamma_{i,j}$. Here, $\beta_\mathcal{N} \in (0,\frac\pi{12}(N-2)]$ is a fixed positive angle.
\end{itemize}
For a depiction of the resulting partition in the case $N=3$, we refer to Figure~\ref{Fig_neighbortriangle}.

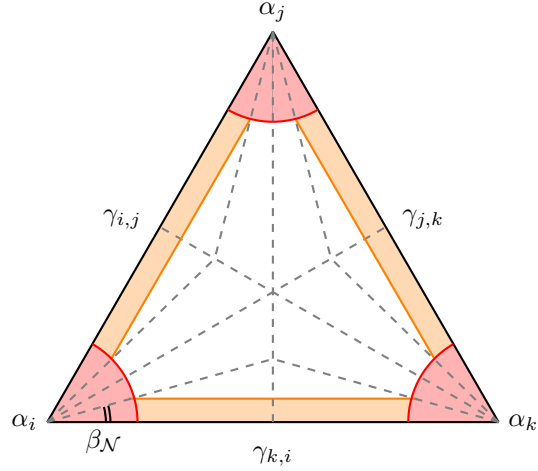
\begin{figure}
	\begin{center}
		\begin{tikzpicture}	
		   \fill[scale=6,inner color = orange!30!, outer color = orange!30!]  (0.82,0.052) -- (0.82,0) -- (0.18,0) -- (0.18,0.052);
		    \fill[scale=6,inner color = orange!30!, outer color = orange!30!]  (0.5,0.866) -- (0,0)  --(0.143,0.142)-- (0.45,0.672);
		     \fill[scale=6,inner color = orange!30!, outer color = orange!30!]  (0.5,0.866) -- (0.55,0.672)  -- (0.857,0.142)-- (1,0);
		   	\fill[scale=6,inner color = red!30!, outer color = red!30!] (0, 0) -- (0.2,0) arc (0:60:0.2);
		   	\fill[scale=6,inner color = red!30!, outer color = red!30!] (1,0) -- (0.8,0) arc (180:120:0.2);
		   	\fill[scale=6,inner color = red!30!, outer color = red!30!] (0.5,0.866) -- (0.5,0.666) arc (270:240:0.2);
		   		\fill[scale=6,inner color = red!30!, outer color = red!30!] (0.5,0.866) -- (0.5,0.666) arc (270:300:0.2);
		\draw[scale=6,thick,black] (0,0) node[left,black] {$\alpha_i$} -- (1,0) node[right,black] {$\alpha_k$};;
		\draw[scale=6,thick,dashed,gray]  (0,0)-- (0.75,0.4325);;
		\draw[scale=6,thick,dashed,gray]  (1,0)-- (0.25,0.4325);;
		\draw[scale=6,thick,dashed,gray] (0.5,0) -- (0.5,0.866);;
		\draw[scale=6,thick,black]  (1,0) -- (0.5,0.866) node[above,black] {$\alpha_j$};;
		\draw[scale=6,thick,black]  (0.5,0.866) --(0,0);;
		\node[black] (none) at (1,2.7) { $\gamma_{i,j}$};
		\node[black] (none) at (5,2.7) { $\gamma_{j,k}$};
		\node[black,below] (none) at (3,-0.2) { $\gamma_{k,i}$};
			\draw[scale=6,orange,thick]  (0.82,0.052) --(0.18,0.052);			
				\draw[scale=6,orange,thick]  (0.45,0.672) --(0.143,0.142);
			\draw[scale=6,orange,thick]  (0.55,0.672) --(0.857,0.142);
			\draw[scale=6,red,thick] (0.2,0) arc (0:60:0.2);
				\draw[scale=6,red,thick] (0.8,0) arc (180:120:0.2);
					\draw[scale=6,red,thick] (0.5,0.666) arc (270:240:0.2);
						\draw[scale=6,red,thick] (0.5,0.666) arc (270:300:0.2);	
	   \draw[scale=6,thick,dashed,gray]  (0,0)-- (0.5,0.14);;
	   \draw[scale=6,thick,dashed,gray]  (1,0)-- (0.5,0.14);;
	   \draw[scale=6,thick,dashed,gray]  (0,0)-- (0.3711,0.3625);;
	   \draw[scale=6,thick,dashed,gray]  (1,0)-- (0.6289,0.3625);;
	   \draw[scale=6,thick,dashed,gray]  (0.5,0.866)-- (0.3711,0.3625);;
	   \draw[scale=6,thick,dashed,gray]  (0.5,0.866)-- (0.6289,0.3625);;
	   	\draw[scale=6,black,thick] (0.13,0) arc (0:15:0.13);
	   	\draw[scale=6,black,thick] (0.14,0) arc (0:15:0.14);
	   	\node[black] (none) at (0.75,-0.25) { ${\beta}_\mathcal{N}$};
		\end{tikzpicture}
	\end{center}
	\caption{Illustration of the partition of the simplex $\triangle^2$ given by $\{\mathcal{U}_i \cap \triangle^2\}_i$, $\{(\mathcal{N}_{i,j} \setminus \mathcal{U}_i  )\cap \triangle^2\}_{i<j}$, and $\triangle^2\setminus \mathcal{N}_{i,j}$.}
	\label{Fig_neighbortriangle}
\end{figure}

We furthermore make use of a couple of additional abbreviations.
\begin{itemize}
\item For any $i<j, \, i,j \in \{1,...,N\}$, we denote by $P_{i,j}: \mathcal{T}_{i,j} \rightarrow \gamma_{i,j}$ the standard orthogonal projection onto $\gamma_{i,j}$, i.\,e.\ the projection onto the nearest point on $\gamma_{i,j}$.
\item We denote by $P^{rad,i}_{i,j}: \mathcal{T}_{i,j} \rightarrow \gamma_{i,j}$ the radial projection onto $\gamma_{i,j}$ with respect to $\alpha_i$, i.e.  $P^{rad,i}_{i,j} u$ denotes the point on $\gamma_{i,j}$ with $|P^{rad,i}_{i,j} u-\alpha_i|=|u-\alpha_i|$.
\item For any $u \in \mathcal{T}_{i,j}$, we denote by \(\beta^i_{j}(u)\) the angle formed by $u - \alpha_i$ and $\gamma_{i,j}$.
\end{itemize}
For an illustration of these notions, we refer to Figure~\ref{Fig_projections} (again in the case $N=3$).

	\begin{figure}
	\begin{tikzpicture}[scale=1.1]
	\draw[scale=7,thick,black] (0,0) node[left,black] {$\alpha_i$} -- (0.5,0); 
	\draw[scale=7,thick,dashed,gray]  (0,0)-- (0.5,0.29);;
	\draw[scale=7,thick,dashed,gray] (0.5,0) -- (0.5,0.29);;
	\draw[scale=7,black,thick,dashed] (0.4,0) arc (0:21:0.4);
	\draw[scale=7,black,thick,dashed] (0.373,0) -- (0.373,0.141);
	\draw[scale=7,black,thick,dashed] (0,0) -- (0.373,0.141);
	\node (none) at (2.243,0.84) [circle,fill,black,inner sep=0.5pt]{};
	\draw[scale=7,black,thick] (0.15,0) arc (0:21:0.15);
	\draw[scale=7,black,thick] (0.16,0) arc (0:21:0.16);
	\node[black] (none) at (1.62,0.22) { \small $\beta^i_{j}(u)$};
	\node[black] (none) at (2.3,-0.25) { \small $P_{i,j} u $};
	\node[black] (none) at (3.2,-0.25) { \small $P^{rad,i}_{i,j} u $};
	\node[black] (none) at (2.65,1.15) { $ u $};
		\node[black] (none) at (3.15,1.4) { $ \mathcal{T}^i_{j}$};
			\node[black] (none) at (1.1,-0.25) { $\gamma_{i,j} $};
	\end{tikzpicture}
	\caption{Projections of $u \in \mathcal{T}^i_{j}$ onto $\gamma_{i , j }\subset \triangle^2$.}
	\label{Fig_projections}
\end{figure}
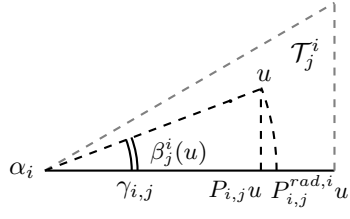

\begin{definition}[Strongly coercive $N$-well potential on the simplex]\label{def:3wellW}
We call a function $W: \triangle^{N-1} \rightarrow [0,\infty)$ a \emph{strongly coercive symmetric $N$-well potential on the simplex} if it satisfies the following list of properties:
\begin{enumerate}
	\item The nonnegative function $W \in C^{1,1}(\triangle^{N-1}; [0,\infty))$ vanishes exactly in the $N$ vertices $\{\alpha_1,..., \alpha_N\}$ of the simplex $\triangle^{N-1}$. It furthermore has the same symmetry properties as the simplex $\triangle^{N-1}$.
	\item Given the geodesic distance
	\begin{align}\label{eq:distW}
	\dist_W(v,w) := \inf  \Big\{ \int_{0}^{1} \sqrt{2W(\gamma(s))} |\gamma'(s)|  \, \ds : \notag \gamma \in C^1([0,1];\R^2)& \\  \text{ with } 
	\gamma(0)=v, \gamma(1)= w& \Big\} ,
	\end{align}
	the infimum for $\dist_W(\alpha_i, \alpha_j)$ is achieved by $\gamma_{i,j}$ and $\dist_W(\alpha_i, \alpha_j)=1$ for any $i,j \in \{1,...,N\}$, $i\neq j$.
	\item (Growth near the minima $\alpha_i$ depending on the angle.) For any distinct $ i,j \in \{1,...,N\}$ and any $u \in \mathcal{U}_i \cap \mathcal{T}_{i,j}$, we have the estimate
	\begin{align}
		(1 + \omega(\beta^i_{j}(u))) W(P^{rad,i}_{i,j}u) \leq W(u),
	\end{align}
	where $\omega:[0, \tfrac{\pi}{6(N-2)}] \rightarrow [0, \infty)$ is a  $C^{1}$ increasing function such that 
	\begin{subequations}\label{eq:omegabound}
		\begin{align} 
		\omega(\beta)&=0 \quad \text{for } \beta = 0 	\,,
		\\
		\omega\left(\beta\right) &\geq 0 
		 \quad \text{for } \beta \in (0, \beta_\mathcal{N}),\\
		\omega\left(\beta\right) &> 	C_\omega
		 \quad \text{for } \beta \in [\beta_\mathcal{N}, \tfrac{\pi}{6(N-2)}],
		\end{align}
	\end{subequations}
where $C_\omega >0$ is a suitable large constant depending on $N$ and where $\beta_{\mathcal{N}}\leq \frac{\pi}{12(N-2)}$.
	
	\item (Growth properties of $W$ and Lipschitz estimate for $\sqrt{2W(u)}$ on the edges $\gamma_{i,j}$.) There exist constants $c_\gamma, C_\gamma>0$ such that
	\begin{align}  \label{eq:W2}
			c_\gamma (u- \alpha_i)^2(u- \alpha_j)^2 \leq W(u) \leq C_\gamma(u- \alpha_i)^2(u- \alpha_j)^2 ,
	\end{align}
holds for all $u\in \gamma_{i,j}$ and any distinct $ i,j \in \{1,...,N\}$.
Furthermore, there exists a constant $L_\gamma>0$ such that for any $u_1, u_2 \in \gamma_{i,j}$
	\begin{align} \label{eq:WLip}
	|	\sqrt{2W(u_1)} - \sqrt{2W(u_2)}| \leq L_\gamma |u_1 - u_2|\,.
	\end{align}
	\item (Growth behavior as one leaves the shortest paths $\gamma_{i,j}$.) For any distinct $i,j \in \{1,...,N\}$ and any $u \in  \mathcal{T}_{i,j}  \cap (\mathcal{N}_{i,j} \setminus (\mathcal{U}_i \cup \mathcal{U}_j))$, the lower bound
	\begin{align} \label{eq:W3}
(1+  C_{\mathcal{N}}\dist^2(u_\vareps, \gamma_{i,j})) W(P_{i,j}u_\varepsilon ) \leq W( u_\varepsilon)
	\end{align}
holds, where $C_{\mathcal{N}} >0$ is a suitable large constant depending on $r_\mathcal{U}, {\beta}_\mathcal{N},$ $L_\gamma, c_\gamma$.
\item (Lower bound away from the paths $\gamma_{i,j}$.) For any $u \in  \triangle \setminus \left( \cup_{i} \mathcal{U}_i \cup_{i < j } \mathcal{N}_{i,j}  \right) $
\begin{align} \label{eq:W4}
\max_{v \in \cup_{i<j } \gamma_{i , j } } W(v) \leq \frac{1}{C_{\text{int}}} W(u)\,,
\end{align}
where $C_{\text{int}}>0$ is a suitable large constant depending on $N, r_\mathcal{U}, {\beta}_\mathcal{N}, c_\gamma, C_\gamma$.
\end{enumerate}
\end{definition}

\subsection{Construction of the approximate phase indicator functions $\psi_i$}
In this subsection we provide an ansatz for the set of functions $\psi_i:\triangle^{N-1}  \rightarrow [0,1]$, $1\leq i\leq N$, in the case of a strongly coercive $N$-well potential on the simplex $W:\triangle^{N-1} \rightarrow [0,\infty)$. Recall that the goal is to construct the $\psi_i$ to satisfy condition (A4) (as introduced Section \ref{sec:mainresults}).
 
Since \((\psi_1 \circ u_\varepsilon,..., \psi_N\circ u_\varepsilon)\) represents an approximation of the limit partition \((\bar{\chi}_1,..., \bar{\chi}_N)\) and since by assumption we have $\dist_W(\alpha_i,\alpha_j)=\delta_{ij}$,
our ansatz for $\psi_i$ on the edges $\gamma_{i,j}$ is 
\begin{align}
\psi_i(u_\varepsilon):= \begin{cases}
1- \dist_{W}(\alpha_i, u_\varepsilon) \text{ along } \gamma_{i,j} \quad  \text{for }  j \in  \{1,...,N\}\setminus \{i\}  \,,  \\
0  \quad  \text{ along } \gamma_{j,k}  \text{ for } j,k \in  \{1,...,N\}\setminus \{i\} :  j <k\,.
\end{cases}
\end{align}
In the following we extend this definition of the set of functions $\psi_i$ on the domain ${\triangle}^{N-1}$. In order to do this, we introduce three interpolation and/or cutoff functions.
\begin{lemma}[Interpolation functions] \label{lemma:interpolation}
Let $\beta_{\mathcal{N}}\in (0,\frac{\pi}{12(N-2)}]$ and $r_{\mathcal{U}}\in (0,\frac{1}{4}]$.
The following statements hold:
	\begin{enumerate}
		\item 	There exists a function $\lambda: [0, \tfrac{\pi}{6}] \rightarrow [0, 1]$ of at least $C^{1}$ regularity satisfying the properties
		\(\lambda(\beta) = 0\) and $ \partial_\beta \lambda =0$ for all $ \beta \in [0, \beta_\mathcal{N}]$,  $\lambda(\frac{\pi}{6(N-2)}) = 1 $, and 
		\begin{align} \label{eq:partiallambda} 
		\max_{\beta \in \left[0,\tfrac{\pi}{6(N-2)} \right]} | \partial_\beta \lambda | \leq 4 (N-2) \,.
		\end{align}
		\item There exists a function $\eta: [0, 1] \rightarrow [0, 1]$ of at least $C^{1}$ regularity satisfying the properties:
		\(\eta(s) = 1\) for $s \in [0, r_\mathcal{U}\cos(\beta_\mathcal{N})]$, \(\eta(s) = 0\) for $s \in [1- r_\mathcal{U}\cos(\beta_\mathcal{N}),1]$, $\eta(s)+\eta(1-s)=1$ for all $s\in [0,1]$,
		and 
		\begin{align} \label{eq:partialeta} 
		\max_{s \in [0,1]} | \partial_s \eta |   \leq \tfrac{5}{2} \,.
		\end{align}
	\end{enumerate}
\end{lemma}
We omit the proof of the lemma, as it is straightforward. We finally proceed to the construction of the functions $\psi_i$ from (A4) in the case of a strongly coercive $N$-well potential on the simplex.

\begin{construction} \label{def:approxpartition}
	Let $W:\triangle^{N-1} \rightarrow [0,\infty)$ be a strongly coercive symmetric $N$-well potential on the simplex in the sense of Definition~\ref{def:3wellW}. 
	We define the associated set of functions $\psi_i:\triangle \rightarrow [0,1]$, $1\leq i\leq N$, as it follows.
	For $i \in \{1,...,N\}$, we construct $\psi_i$ on the edge between $\alpha_i$ and $\alpha_j$ ($j \in  \{1,...,N\}\setminus \{i\}$) by
	\begin{align}
	\label{DefPsiiOnEdge}
	\psi_i(u):= 1- \dist_{W}(\alpha_i, u)\quad \text{for }  u \in \gamma_{i,j}.
	\end{align}
	Let $j \in  \{1,...,N\}\setminus \{i\}$.
	For any $u \in \mathcal{T}^i_j$, we set
	\begin{align*}
	\psi_k(u) :=  0\quad\text{for any }k \in  \{1,...,N\}\setminus \{i,j\}.
	\end{align*}
	Furthermore, we define $\psi_i$ and $\psi_j$ on $\mathcal{T}^i_j \cap (\mathcal{N}_{i,j} \cup   \mathcal{U}_i )$ as follows:
	\begin{itemize}
		\item 	If $u\in \mathcal{U}_i   \cap \mathcal{T}^i_{j}$, we set 
		\begin{subequations} \label{eq:psiU}
		\begin{align} 
		\psi_i(u ) &:=  \psi_i(P^{rad,i}_{i,j}u )  \,,\\	
		\psi_j(u) &:=  (1 - \lambda(\beta^i_{j}(u))) \psi_j(P^{rad,i}_{i,j}u ) \,.
		\end{align}
		\end{subequations}
		\item If $ u \in (\mathcal{N}_{i,j} \setminus  \mathcal{U}_i )\cap \mathcal{T}^i_{j} $, we set
		\begin{subequations} \label{eq:psiN}
		\begin{align}
		\psi_i(u ) &:= \eta(|P_{i,j}u - \alpha_i|) \psi_i(P^{rad,i}_{i,j}u ) 
		+ \eta(|P_{i,j}u - \alpha_j|)  \psi_i(P^{rad,j}_{i,j}u )  \,,\notag\\	
		\psi_j(u ) &:=  \eta(|P_{i,j}u - \alpha_i|)  \psi_j(P^{rad,i}_{i,j}u )  
		+ \eta(|P_{i,j}u - \alpha_j|)  \psi_j(P^{rad,j}_{i,j}u ) .
		\end{align}		
			\end{subequations}
	\end{itemize}
	Finally, outside of the domain $\mathcal{M}_i:=\bigcup_j \mathcal{U}_j\cup \bigcup_{j:j\neq i} \mathcal{N}_{i,j} \cup\bigcup_{j,k:j,k\neq i,j<k}\mathcal{T}_{j,k}$ on which we have defined $\psi_i$ so far, we define $\psi_i$ as a suitable $C^{1,1}$ extension:
	\begin{itemize}
		\item If $u \in \triangle^{N-1}\setminus \mathcal{M}_i$, we define
		\begin{subequations} \label{eq:psiint}
			\begin{align}
			\psi_i(u_\varepsilon ) :=  \psi_{i}^{\text{int}}(u)
			\end{align}
		\end{subequations}
		where $\psi_{i}^{\text{int}}:   \triangle^{N-1} \rightarrow [0,1] $ is a suitable $C^{1,1}$ extension of $\psi_i:\mathcal{M}_{i}\cap \triangle^{N-1}\rightarrow [0,1]$ that almost preserves the Lipschitz constant of $\psi_i:\mathcal{M}_{i}\cap \triangle^{N-1}\rightarrow [0,1]$.
	\end{itemize}
\end{construction}

\subsection{Existence of a set of suitable approximate phase indicator functions} \label{sec:maincoerprop}

\begin{proof}[Proof of Proposition \ref{prop:hpW}]	
	It directly follows from Construction~\ref{def:approxpartition} that the set of functions $\psi_i:\triangle \rightarrow [0,1]$, $1\leq i\leq N$, satisfy $\psi_i=1$ at $\alpha_i$ and $\psi_i(u)<1$ for $u\neq \alpha_i$.
	
We next show the validity of (A4) in a given set $\mathcal{T}^i_{j}$, which we further decompose into $\mathcal{U}_i  \cap \mathcal{T}^i_{j} $ , $(\mathcal{N}_{i,j} \setminus  \mathcal{U}_i )\cap \mathcal{T}^i_{j}$, and $ \mathcal{T}^i_{j} \setminus (  \mathcal{U}_i \cup \mathcal{N}_{i,j})$.
	
\emph{Step 1: Proof of (A4) in $\mathcal{U}_i  \cap \mathcal{T}^i_{j}$.} Let $u \in \mathcal{U}_i   \cap \mathcal{T}^i_{j}$. Recall $\psi_0 := 1 - \sum_{\ell=1}^N \psi_\ell $. Due to $\dist_W(\alpha_i,\alpha_j)=1$ and \eqref{DefPsiiOnEdge}, it follows that $\psi_i(P^{rad,i}_{i,j}u)=1-\psi_j(P^{rad,i}_{i,j}u)$. Thus, we have $\psi_0(u ) =  \lambda(\beta^i_{j}(u)) \psi_j(P^{rad,i}_{i,j}u)$. We also have $(\tfrac{\alpha_j-\alpha_i}{|\alpha_j-\alpha_i|}\cdot \nabla)\psi_i(P^{rad,i}_{i,j}u)=\sqrt{2W(P^{rad,i}_{i,j}u)}$.
	Using \eqref{eq:psiU}, we can compute
	\begin{align*}
	\partial_u \psi_{i,j} (u) &=  (2 - \lambda(\beta^i_{j}(u)) ) \sqrt{2W(P^{rad,i}_{i,j}u )} \vec{e}_{rad,i}(u) \\
	&\quad  - \partial_\beta \lambda(\beta^i_{j}(u)) \frac{1}{|u  - \alpha_i|} \psi_j(P^{rad,i}_{i,j}u )   \vec{e}_{\beta^i_{j}}(u)
	\,,\\	
	 \partial_u \psi_0(u ) &= \lambda(\beta^i_{j}(u))  \sqrt{2W(P^{rad,i}_{i,j}u )} \vec{e}_{rad,i}(u) \\
	&\quad +  \partial_\beta \lambda(\beta^i_{j}(u)) \frac{1}{|u  - \alpha_i|} \psi_j(P^{rad,i}_{i,j}u )   \vec{e}_{\beta^i_{j}}(u),
	\end{align*}
	where $\vec{e}_{rad,i}(u), \vec{e}_{\beta^i_{j}}(u) $ are orthogonal vectors associated to the $(N-1)$-dimensional spherical coordinates pointing in the direction of steepest ascent of $|u-\alpha_i|$ respectively $\beta^i_{j}(u)$; i.\,e.\ , in particular we have $\vec{e}_{rad,i}(u) := \frac{u - \alpha_i}{|u - \alpha_i|}$.
    For the sake of brevity, we omit the dependencies on $ \beta^i_{j}(u)$ in the following.
    Then, it follows that
	\begin{align*}
	&|\partial_u \psi_{i,j} (u )|^2 \leq \left( (2 - \lambda)^2 + 
	|\partial_\beta \lambda|^2 \right) 2W(P^{rad,i}_{i,j}u ) 
	\,,\\	
	&	|\partial_u \psi_0(u )|^2 \leq   \big(\lambda^2 + 
	|\partial_\beta \lambda|^2  \big) 2W(P^{rad,i}_{i,j}u )  \,,\\
	&	|\partial_u \psi_{i,j} (u ) \cdot \partial_u \psi_0(u )|  \leq  \left( \lambda(2-\lambda ) + 
	|\partial_\beta \lambda|^2 \right) 2W(P^{rad,i}_{i,j}u ) \,.
	\end{align*} 			
For $\delta>0$ small enough, we have
\begin{align*}
&\big|\tfrac{1}{2}\partial_u \psi_{i,j}(u)\big|^2 
+ \left( 1 +   \delta \right)  \big|\tfrac{1}{2\sqrt{3}}\partial_u \psi_0(u) \big|^2		
+  \delta   \left| \partial_u \psi_{i,j} (u)\cdot   \partial_u \psi_{0} (u) \right| \\
	& \leq  
\Big[ \tfrac14(2 - \lambda)^2 + 
\tfrac14 |\partial_\beta \lambda|^2 
+ (1+\delta)
  \tfrac{1}{12}\big( \lambda^2 + 
|\partial_\beta \lambda|^2  \big) \\
& \quad  \;\;   +  \delta 	\left( \lambda(1-\lambda) + 
C|\partial_\beta \lambda|^2\right) \Big] 2W(P^{rad,i}_{i,j}u) 	\\
& \leq  
\bigg[ 1 - \lambda + \tfrac{1}{3}\lambda^2 + 
\tfrac{1}{3} |\partial_\beta \lambda|^2
+ C\delta \bigg] 2W(P^{rad,i}_{i,j}u) 	\\
 &\leq 
	(1 + \omega(\beta^i_{j}(u))) 2W(P^{rad,i}_{i,j}u )\\
	&\leq 
	2W(u)\,,
	\end{align*}
	where we used \eqref{eq:partiallambda} and \eqref{eq:omegabound} together with the fact that $\delta_{\text{cal}}$ and hence $\delta$ can be chosen arbitrarly small.

\emph{Step 2: Proof of (A4) in $u \in (\mathcal{N}_{i,j} \setminus  \mathcal{U}_i )\cap \mathcal{T}^i_{j}$.} 
	Let $ u \in (\mathcal{N}_{i,j} \setminus  \mathcal{U}_i )\cap \mathcal{T}^i_{j} $. First, note that $\psi_0 = 1 - \psi_i - \psi_j  \equiv 0$ on $ (\mathcal{N}_{i,j} \setminus  \mathcal{U}_i )\cap \mathcal{T}^i_{j} $.
	Using \eqref{eq:psiN}, we compute
	\begin{align*}
	&\partial_u \psi_{i,j} (u ) \\
	&=  \eta(|P_{i,j}u- \alpha_i|)  2 \sqrt{2W(P^{rad,i}_{i,j}u )} \vec{e}_{rad,i}(u)
	-  \eta(|P_{i,j}u - \alpha_j|) 2 \sqrt{2W(P^{rad,j}_{i,j}u )} \vec{e}_{rad,j}(u)\\
	& \quad + \partial_u \eta(|P_{i,j}u- \alpha_i|) \left[\psi_j(P^{rad,i}_{i,j}u ) - \psi_j(P^{rad,j}_{i,j}u ) + \psi_i(P^{rad,j}_{i,j}u ) - \psi_i(P^{rad,i}_{i,j}u ) \right]
	\,,
	\end{align*}
	where $\vec{e}_{rad,i}(u) := \frac{u - \alpha_i}{|u - \alpha_i|}$ and $\vec{e}_{rad,j}(u) := \frac{u - \alpha_j}{|u - \alpha_j|}$.
	Note that we have
	\begin{align}
	|P^{rad,i}_{i,j}u - \alpha_i|  &\leq |P_{i,j} u  - \alpha_i| + \frac{\dist^2(u, \gamma_{i,j})}{2|P_{i,j} u - \alpha_i|}\,,		\\
	\label{EstimateDifferenceProjections}
		|P^{rad,i}_{i,j}u - P^{rad,j}_{i,j}u | &\leq  \frac{\dist^2(u, \gamma_{i,j})}{2|P_{i,j} u - \alpha_i||P_{i,j} u - \alpha_j|} \,.
		\end{align}
	As a consequence, we obtain
	\begin{align*}
	&	\max\{ \psi_j(P^{rad,i}_{i,j}u ) - \psi_j(P^{rad,j}_{i,j}u), \psi_i(P^{rad,j}_{i,j}u ) - \psi_i(P^{rad,i}_{i,j}u)\}\\
	&\leq  \sqrt{2W(v_u )} \frac{\dist^2(u, \gamma_{i,j})}{2|P_{i,j} u- \alpha_i||P_{i,j} u - \alpha_j|}\,,
	\end{align*} 
	where $v_u \in \gamma_{i , j }$ maximum of $\sqrt{2W}$ on the segment connecting $P^{rad,j}_{i,j}u$ to $P^{rad,i}_{i,j}u$.
	From \eqref{eq:partialeta} and $\eta(|P_{i,j}u-\alpha_i|)+\eta(|P_{i,j}u-\alpha_j|)=1$ it follows that
		\begin{align}
		\label{EstimatePsiijNearGamma}
	\left| \tfrac12 \partial_u \psi_{i,j} (u ) \right|^2
	\leq     \left( 1
+  \frac{5}{4|P_{i,j} u - \alpha_i||P_{i,j} u - \alpha_j|} \dist^2(u, \gamma_{i,j})\right)^2 {2W(v_u)}   \,.
	\end{align}
		Using that $1-|P_{i,j}u-\alpha_i|=|P_{i,j}u-\alpha_j|$ and ${|P_{i,j} u - \alpha_i|^{-1}(1- |P_{i,j} u - \alpha_i|)^{-1}} \leq {r_\mathcal{U}^{-1}(1-r_\mathcal{U})^{-1}}$, one can obtain
		\begin{align*}
   &\left( 1
	+   \frac{5}{4|P_{i,j} u - \alpha_i||P_{i,j} u_\vareps - \alpha_j|} \dist^2(u, \gamma_{i,j})\right)^2  \\
	&\leq   1
	+  C_1 \dist^2(u, \gamma_{i,j})
	\end{align*}
	for $C_1 = \frac{5}{2r_\mathcal{U}(1-r_\mathcal{U})} + \frac{25 \sin^2(\beta_\mathcal{N})}{16(1-r_\mathcal{U})^2}$ since $\dist(u, \gamma_{i,j}) \leq r_\mathcal{U} \sin(\beta_\mathcal{N})$.
	On the other hand, first by adding a zero and using \eqref{eq:WLip}, then noting that $|v_u - \alpha_i| \leq |P^{rad,i}_{i,j}u- \alpha_i| $ and using \eqref{EstimateDifferenceProjections}, from \eqref{eq:W2} together with the fact that $ |P_{i,j}u - \alpha_i| \geq  \frac12$  we can deduce
		\begin{align*}
 &{2W(v_u)}   
  = {2W(P_{i,j}u)} \left(1 + \frac{\sqrt{2W(v_u)}  - \sqrt{2W(P_{i,j}u)}}{\sqrt{2W(P_{i,j}u)}} \right)^2
 \\
	& \leq {2W(P_{i,j}u)} \left(  1
	+    \frac{L_\gamma}{2|P_{i,j} u - \alpha_i| \sqrt{2W(P_{i,j}u)}} \dist^2(u, \gamma_{i,j}) \right)^2\\
		& \leq {2W(P_{i,j}u)} \left(  1
	+    \frac{L_\gamma}{\sqrt{2c_\gamma} } \tan^2(\beta^i_{j}(u)) \right)^2 \\
		& \leq {2W(P_{i,j}u)} \left(  1
	+   C_2 \tan^2(\beta^i_{j}(u)) \right) \,,
	\end{align*}
	where $C_2= 2 \frac{L_\gamma}{\sqrt{2c_\gamma} } + \tan^2(\beta_\mathcal{N})\frac{L^2_\gamma}{2c_\gamma } $.
	Moreover, one can compute
	\begin{align*}
 &\left( 1
	+  C_1 \dist^2(u_, \gamma_{i,j})\right) \left(  1
	+   C_2 \tan^2(\beta^i_{j}(u)) \right)\\
	&\leq 1 	+  C_1 \dist^2(u, \gamma_{i,j})  
	+   \frac{C_2}{r^2_\mathcal{U}  \cos^2 \beta_{\mathcal{N}}} \dist^2(u, \gamma_{i,j})
	+ C_1 C_2 \tan^2(\beta_\mathcal{N})  \dist^2(u, \gamma_{i,j})  \\
& \leq 1 +  C_{\mathcal{N}} \dist^2(u, \gamma_{i,j})
	\end{align*}
	for $C_{\mathcal{N}} =  C_1 + \frac{C_2}{r^2_\mathcal{U} \cos^2 \beta_{\mathcal{N}}} + C_1C_2 \tan^2(\beta_\mathcal{N})$. Using our assumption \eqref{eq:W3}, we can conclude from \eqref{EstimatePsiijNearGamma} and the preceding three estimates that
		\begin{align*}
	\left| \tfrac12 \partial_u \psi_{i,j} (u ) \right|^2 
	&	\leq  \left( 1
	+  C_{\mathcal{N}} \dist^2(u, \gamma_{i,j})\right){2W(P_{i,j}u)} \\
	&\leq 2W(u)\,.
	\end{align*}
	
\emph{Step 3: Proof of (A4) in $u \in \mathcal{T}^i_{j}\setminus \mathcal{M}_i$.}
Let $u \in \mathcal{T}^i_{j}\setminus \mathcal{M}_i$. By \eqref{eq:psiint} we have
\begin{align*}
\psi_{i,j}(u_\varepsilon ) &=  \psi_{j}^{\text{int}}(u ) - \psi_{i}^{\text{int}}(u ) \,, 
\\	\psi_0(u_\varepsilon ) &= 1 - \psi_{i}^{\text{int}}(u ) - \psi_{j}^{\text{int}}(u )   \,,
\end{align*}
where $\psi_{\ell}^{\text{int}}$, $\ell \in \{i,j\}$, is a $C^{1,1}$ extension of $\psi_{\ell}$ from $\mathcal{M}_{\ell}\cap \triangle^{N-1}$ to $\triangle^{N-1}$ that approximately preserves the Lipschitz constant $L_{\text{int},\ell}>0$. Thus, we have
\begin{align*}
|\partial_u \psi_{i,j} | \leq  (1+\delta)(L_{\text{int},i}+ L_{\text{int},j}) \,,  ~
\text{ and similarly } ~|\partial_u \psi_{0} | \leq (1+\delta) (L_{\text{int},i}+ L_{\text{int},j}) \,.
\end{align*}
It is not too difficult to derive an estimate on the Lipschitz constants $L_{\text{int},i}$ and $L_{\text{int},j}$ in terms of $\max_{w \in \cup_{\ell,m:\ell< m} \gamma_{\ell,m}}{2W(w)}$. To this aim, we first estimate
\begin{align*}
&|\psi_i(u)|\leq \max_{w \in \gamma_{i,m}} \sqrt{2W(w)} r_{\mathcal{U}}
\text{ for }u\in \gamma_{i,m}\cap \mathcal{U}_m\text{ and }m\neq i,
\\
&|\psi_i(P^{rad,i}_{i,m}u) - \psi_\ell(P^{rad,m}_{i,m}u)|
\leq \max_{w \in \gamma_{i,m}} \sqrt{2W(w)} |P^{rad,i}_{i,m}u-P^{rad,m}_{i,m}u|
\\&~~~~~~~~~
\stackrel{\eqref{EstimateDifferenceProjections}}{\leq} \max_{w \in \gamma_{i,m}} \sqrt{2W(w)} \frac{r_{\mathcal{U}}}{2(1-r_{\mathcal{U}})}  \sin^2(\beta_{\mathcal{N}})~~~
\text{ for }u\in \mathcal{N}_{i,m}\text{ and }m\neq i.
\end{align*}
Using these estimates, the definitions \eqref{eq:psiU}-\eqref{eq:psiN}, and the bounds \eqref{eq:partiallambda} and \eqref{eq:partialeta}, we obtain
\begin{align*}
& | \partial_u \psi_{i}  (u)| \leq  \max_{w \in \gamma_{i,j}} \sqrt{2W(w)}
\quad \text{for } u \in \mathcal{U}_i \,,
\\
& | \partial_u \psi_{i}  (u)| \leq  \big(1+4(N-2)r_\mathcal{U}\big) \max_{w \in \gamma_{i,m}} \sqrt{2W(w)}
\quad \text{for } u \in \mathcal{U}_m, m\neq i \,,
\\
& | \partial_u \psi_{i}  (u)|\leq \Big(1+ \frac{5r_{\mathcal{U}}}{4(1-r_{\mathcal{U}})}  \sin^2(\beta_{\mathcal{N}})\Big)\max_{w \in \gamma_{i,m}} \sqrt{2W(w)}     
\text{ for } u \in \mathcal{N}_{i,m},\, m \neq i.
\end{align*}
Furthermore, we have $|\partial_u \psi_{i}  (u)| =0$ in $\triangle^{N-1}\cap \left( \cup_{ m <n: m \neq i, n \neq i} \mathcal{T}_{m,n}\right)$. 
Defining
\begin{align*}
M := \max \bigg\{1, 1+4(N-2)r_\mathcal{U},1+ \frac{5r_{\mathcal{U}}}{4(1-r_{\mathcal{U}})}  \sin^2(\beta_{\mathcal{N}}) \bigg\} = 1+4(N-2)r_\mathcal{U}\,, 
\end{align*}
this yields by \eqref{eq:W4} for $u\in \mathcal{M}_i\cap \triangle^{N-1}$
\begin{align*}
| \partial_u \psi_{i} (u)| \leq M \max_{w \in  \bigcup_{\ell,m:\ell< m } \gamma_{\ell, m}   } \sqrt{2W(w)}\Big) =:M_{W}
\end{align*}
for any $u \in  \triangle^{N-1}\cap \mathcal{M}_i$.   
In order to estimate the Lipschitz constant $L_\text{int,i} $ of $\psi_{i}|_{\mathcal{M}_i\cap \triangle^{N-1}}$, one has to address the issue of nonconvexity of $\mathcal{M}_i$. It is not too difficult to see (but rather technical to prove) that for any pair of points $u,v\in \mathcal{M}_i$ there exists a connecting path $\tilde \gamma$ in $\mathcal{M}_i$ with $\text{len}(\tilde\gamma)\leq C_{\mathcal{M}}|u-v|$. This shows $L_{\text{int},i} \leq C_{\mathcal{M}} M_W$.
Having an upper bound for $L_{\text{int},i}$, using the fact that our extension of $\psi_\ell$ to $\triangle^{N-1}\setminus \mathcal{M}_i$ approximately preserves the Lipschitz constant, and choosing $C_{\text{int}}> \tfrac43 C^2_{\mathcal{M}} M^2$ in  \eqref{eq:W4}, we can compute for $u\in \mathcal{T}_j^i \setminus \mathcal{M}_i$ and $\delta>\sqrt{\delta_{\text{cal}}}$ (with $\delta_{\text{cal}}$ given by the gradient-flow calibration)
\begin{align*} 
&\big|\tfrac{1}{2}\partial_u \psi_{i,j}(u)\big|^2 
+ \left( 1 +   \delta \right) \big|\tfrac{1}{2\sqrt{3}}\partial_u \psi_0(u) \big|^2 	
+  \delta   \left| \partial_u \psi_{i,j} (u)\cdot   \partial_u \psi_{0} (u) \right|
\\& \leq  ( 1 +   \tfrac13 + 4\delta  )\max_m L^2_{\text{int},m}
\leq 2W(u).
\end{align*}
Here, we have used the fact that $\delta_{\text{cal}}>0$ and hence $\delta>0$ can be chosen arbitrarily small.
\end{proof}

\bibliographystyle{abbrv}
\bibliography{multiphase}

\end{document}